\documentclass[reqno,11pt]{amsart}
\usepackage{amssymb, amsmath,latexsym,amsfonts,amsbsy, amsthm,mathtools}
\usepackage{rotating}
\usepackage{tikz}
\usepackage{calc}
\usepackage{tikz-cd}
\usepackage{mathabx}
\usepackage{tikz-3dplot}
\usetikzlibrary{calc}
\usetikzlibrary{decorations.markings,arrows}
\usetikzlibrary{shapes.geometric}
\usepackage{color}
\usepackage{extarrows}
\usepackage{graphicx}
\allowdisplaybreaks[4]
\let\pa=\partial

\let\ve=\varepsilon
\let\pa=\partial

\def\no{\noindent}

\newcommand{\beq}{\begin{equation}}
\newcommand{\eeq}{\end{equation}}
\newcommand{\ben}{\begin{eqnarray}}
\newcommand{\een}{\end{eqnarray}}
\newcommand{\beno}{\begin{eqnarray*}}
\newcommand{\eeno}{\end{eqnarray*}}

\renewcommand{\theequation}{\thesection.\arabic{equation}}

\newtheorem{Theorem}{Theorem}[section]
\newtheorem{Definition}[Theorem]{Definition}
\newtheorem{Proposition}[Theorem]{Proposition}
\newtheorem{Lemma}[Theorem]{Lemma}
\newtheorem{Corollary}[Theorem]{Corollary}
\newtheorem{Remark}[Theorem]{Remark}

\newcommand{\ud}{\mathrm{d}}

\newcommand{\bl}{\mathbf{l}}
\newcommand{\mm}{\mathbf{m}}
\newcommand{\nn}{\mathbf{n}}

\newcommand{\BB}{\mathbf{B}}
\newcommand{\CC}{\mathbf{C}}
\newcommand{\DD}{\mathbf{D}}
\newcommand{\FF}{\mathbf{F}}
\newcommand{\EE}{\mathbf{E}}
\newcommand{\GG}{\mathbf{G}}
\newcommand{\HH}{\mathbf{H}}
\newcommand{\II}{\mathbf{I}}
\newcommand{\JJ}{\mathbf{J}}
\newcommand{\KK}{\mathbf{K}}
\newcommand{\LL}{\mathbf{L}}
\newcommand{\MM}{\mathbf{M}}
\newcommand{\NN}{\mathbf{N}}

\newcommand{\QQ}{\mathbf{Q}}
\newcommand{\PP}{\mathbf{P}}
\newcommand{\RR}{\mathbf{R}}

\newcommand{\TT}{\mathbf{T}}
\newcommand{\UU}{\mathbf{U}}
\newcommand{\VV}{\mathbf{V}}
\newcommand{\WW}{\mathbf{W}}

\newcommand{\A}{\mathbf{A}}
\newcommand{\B}{\mathbf{B}}
\newcommand{\Bp}{\mathbf{\Psi}}

\newcommand{\CB}{\mathcal{B}}

\newcommand{\CH}{\mathcal{H}}

\newcommand{\BM}{\mathbb{M}}

\newcommand{\BS}{\mathbb{S}}
\newcommand{\BA}{\mathbb{A}}

\newcommand{\BV}{\mathbb{V}}

\newcommand{\BR}{\mathbb{R}}

\newcommand{\BTe}{\mathbf{\Theta}}

\definecolor{grey}{rgb}{0.5,0.5,0.5}

\begin{document}
\title[Matrix-Valued Allen-Cahn]{Matrix-Valued Allen-Cahn Equation and the Keller-Rubinstein-Sternberg Problem}

\author{Mingwen Fei}
\address{School of Mathematics and Statistics, Anhui Normal University, Wuhu 241002, China}
\email{mwfei@ahnu.edu.cn}

\author{Fanghua Lin}
\address{Courant Institute of Mathematical Science, New York University}
\email{linf@cims.nyu.edu}

\author{Wei Wang}
\address{Department of Mathematics, Zhejiang University, Hangzhou 310027, China}
\email{wangw07@zju.edu.cn}

\author{Zhifei Zhang}
\address{School of  Mathematical Sciences, Peking University, Beijing 100871, China}
\email{zfzhang@math.pku.edu.cn}

\renewcommand{\theequation}{\thesection.\arabic{equation}}
\setcounter{equation}{0}


\begin{abstract}
In this paper, we consider the sharp interface limit of a matrix-valued Allen-Cahn equation, which takes the form:
\begin{align*}
  \partial_t\A=\Delta \A-\ve^{-2}(\A\A^{\mathrm{T}}\A-\A)\quad \text{with}\quad \A:\Omega\subset\mathbb{R}^m\to \mathbb{R}^{n\times n}.
\end{align*}
We show that the sharp interface system is a two-phases flow system: the interface evolves according to the motion by mean curvature; in the two bulk phase regions, the solution obeys the heat flow of harmonic maps with values in $O^+(n)$ and $O^-(n)$ (represent the sets of $n\times n$ orthogonal matrices with determinant $+1$ and $-1$ respectively); on the interface, the phase matrices in two sides satisfy a novel mixed boundary condition.

The above result provides a solution to the Keller-Rubinstein-Sternberg's (conjecture) problem in this ${O}(n)$ setting.
Our proof relies on two key ingredients. First, in order to construct the approximate solution by matched asymptotic expansions as the standard approach does not seem to work, we introduce the notion of quasi-minimal connecting orbits. They satisfy the usual leading order equations up to some small higher order terms. In addition, the linearized systems around these quasi-minimal orbits needs to be solvable up to some good remainders. These flexibilities are needed for the possible ``degenerations" or higher dimensional kernels for the linearized operators due to intriguing boundary conditions at the sharp interface. The second key point is to establish a spectral uniform lower bound estimate for the linearized operator around approximate solutions. To this end, we introduce additional decompositions to reduce the problem into the coercive estimates of several linearized operators for scalar functions and some singular product estimates which is accomplished by exploring special cancellation structures between eigenfunctions of these linearized operators.

\end{abstract}

\date{\today}
\maketitle

\numberwithin{equation}{section}

\tableofcontents

\section{Introduction}
\subsection{Background and related results}

The phase transition problem has drawn great interests in both analysis and applications.
The simplest model for the phase transition is the scalar Allen-Cahn equation, which was introduced by Allen-Cahn \cite{AlCa} to model the motion of antiphase boundaries in crystalline solids. Let $u:\Omega\subseteq\mathbb{R}^m\to \mathbb{R}$, and $F(u)$ is a potential function with double well \big(ex. $F(u)=(u^2-1)^2/4$\big). The equation reads as follows
\begin{align}
\partial_tu=\Delta u-\frac{1}{\ve^2}F'(u).\label{allen cahn}
\end{align}
As $\ve\to 0$, the domain $\Omega$ will be separated into two regions $\Omega_\pm$, where $u\to \pm 1$ respectively. Moreover, the interface between these two regions evolves according to the mean curvature flow. This sharp interface limit has been rigorously justified in both static and dynamic cases by numerous authors via different methods, see \cite{BroKo2, Ch-jde92, dMS2, ESS, Ilm, Modi, ModiMo, Ster} for examples.

In \cite{RSK1,RSK2}, Rubinstein-Sternberg-Keller introduced a vector-valued system for fast reaction and slow diffusion:
\begin{align}\label{equation:RSK}
  \partial_tu=\Delta u-\frac{1}{\ve^2}\partial_uF(u),
\end{align}
where $u:\Omega\subseteq\mathbb{R}^m\to\mathbb{R}^n$ is a phase-indicator function, and the nonnegative, smooth potential function $F:\mathbb{R}^n\to\mathbb{R}$  vanishes exactly on two disjoint connected sub-manifolds in $\mathbb{R}^n$. By a formal asymptotic expansion, they found that, when $\ve\to 0$, the interface  moves by its mean
curvature, while away from the interface $u$ tends to the heat flow of harmonic maps into the sub-manifolds
(potential wells).

Rigorous analysis for the Keller-Rubinstein-Sternberg problem (\ref{equation:RSK}) are few; and each studies this general problem in a some what special setting. The problem has remained largely open including the $O(n)$-model.
In \cite{LPW}, Lin-Pan-Wang rigorously analyzed the asymptotic behaviour as $\ve\to0$ for the energy minimizing static solutions to a class of the Keller-Rubinstein-Sternberg problem. There is also a mixed Dirichlet-Neumann boundary condition for the phase field along the interface. A regularity theory for minimizing maps of the limit problem was obtained in \cite{LinWang2019}.  

For asymptotics of the gradient flow \eqref{equation:RSK}, Bronsard-Stoth \cite{BroSt} studied the radially symmetric case for a special $\mathbb{R}^2$-valued problem, and a Neumann-jump boundary condition for the limit system was derived. In \cite{FWZZ}, without the radially symmetric assumption, Fei-Wang-Zhang-Zhang rigorously justified the asymptotic limit for a physical $\mathbb{R}^{3\times3}$-valued model, which describes the isotropic-nematic phase transition for liquid crystals (the radially symmetric case is studied in \cite{MMS}). For the latter problem, we refer to \cite{LW, LL} for  more recent progresses. There are several interesting studies for related problems, see for examples, \cite{Bethuel, FT}.

In this paper, through a careful and systematic analysis of the matrix-valued Allen-Cahn equations, we obtain these intriguing boundary conditions for the phase field along the sharp interface. In fact, a much clear picture emerges. The boundary conditions for the phase field (and also for the limiting problem) are actually dictated by the geometric properties of the potential energy wells and the so-called minimal connecting orbits. The existence and geometrical properties of these minimal connecting orbits will show up also in deriving matched asymptotic expansions and the construction of higher order approximate solutions with drastic variations of analytical difficulties. It is something never exists in the scalar case or several vector-valued cases studied by various authors before. It also illustrate why there are different types of boundary conditions along the interface in earlier works \cite{LPW, LinWang2019, FWZZ, LL, LW}, for examples. Let us describe it roughly below. Some related discussions are in Section \ref{sec:inner1} of the paper.

Following \cite{RSK1, RSK2} and \cite{LPW}, we consider the energy functional for $\Omega\subset\mathbb{R}^m$:
\begin{align*}
  E_\ve(u)=\int_{\Omega}\Big(|\nabla u|^2+\frac{1}{\ve^2}F(u)\Big)\ud x,
\end{align*}
where the potential function $F(u)$ is smooth for $u\in\mathbb{R}^N$ and satisfies the properties that
\begin{align*}
  c_1d^2(u, \Sigma)\le F(u)\le C_2 d^2(u,\Sigma)
\end{align*}
for $u\in \Sigma_{\delta_0}$ (the $\delta_0$-neighborhood of $\Sigma$), and that $F(u)\ge c_3$ whenever $d(u,\Sigma)\ge \delta_0$. Here $\Sigma=\Sigma_+\cap\Sigma_-$ is the union of two disjoint, smooth, compact and connected submanifolds in $\mathbb{R}^N$ on which $F$ vanishes. For any two points $p_+, p_-$ in $\mathbb{R}^N$, one can define the distance with weight  $\sqrt{F/2}$:
\begin{align*}
  d_F(p^+,p^-)=\inf\Big\{\int_\mathbb{R}\Big(|\xi'(t)|^2+F(\xi)\Big)\ud t: \xi\in H^1(\mathbb{R}, \mathbb{R}^N),\xi(\pm\infty)=p_\pm\Big\},
\end{align*}
and let
\begin{align*}
 c_0^F  =\inf_{q^+\in\Sigma_+,q^-\in \Sigma_-}d_F(q^+,q^-).
\end{align*}

Under some very nature assumption, one has the following expansion for the energy $E_\ve(u)$ for the so-called well-prepared data $u$:
\begin{align*}
  E_\ve(u)=\frac{c_0^F}{\ve}\CH^{m-1}(\Gamma_u)+D(u)+O_\ve(1).
\end{align*}
Here $\Gamma_u$ is a sharp interface between two phases $\Sigma_+$ and $\Sigma_-$, and $D(u)$ is the Dirichlet energy of the map from $\Omega_\pm(u)$ (the sub-regions of $\Omega$ which are separated by $\Gamma_u$ into $\Sigma_\pm$).

If one considers energy minimizers (as in \cite{LPW}), then it is intuitively clear that $\Gamma_u$ will be area minimizing surface, and $u:\Omega_\pm(u)\to\Sigma_\pm$  will be energy minimizing maps. There will be also natural boundary conditions for $u$ as the two sides of the sharp interface.

The difficult point for the gradient flow for this type energy functional is that the sharp interface motion $\{\Gamma(t)\}$ and the heat flow of harmonic map in the bulk are in the same time scale (unlike the dynamics of Ginzburg-Landau vortices \cite{Lin}).
These two motions may be coupled in $O(\ve)$ order terms; and due to the interfacial energy is of order $O(\ve^{-1})$, it may lead to undesired $O(1)$ changes in motion of the interface.  This coupling occurs
through the intriguing boundary conditions for the maps at the sharp interfaces. As phases evolves in the bulk, the boundary values of $u_\pm$ at the sharp interface would change accordingly. The latter may alter the weights in calculating the weighted surface area of the sharp interface, that is the distance with weight $\sqrt{F/2}$ for points $u_\pm(x,t)$ from two sides of $\Gamma(t)$.
Of course, if the weighted distance between any pair of points $p_{\pm}\in\Sigma_\pm$ is
the same, then there would be no coupling effect and the weight for the sharp interface area will not change in the evolution. This is the case we shall call that $F(u)$ is {\it fully minimally paired} (see discussions in Section \ref{sec:inner1}).
From this point of view, \cite{LPW} studied a class of problem which is only partially minimally paired, but in a certain specific way. The $O(n)$-model we study in this paper has a great generality for situations of partially minimally paired  $F(u)$. We refer to the discussions in Section \ref{sec:inner1}.

This paper is the first in a series, and we shall concentrated on the study of the evolution of well-prepared initial data for the matrix-valued Allen-Cahn equations. The construction of approximate solutions in this paper is based on the matched asymptotic expansion method,  which has been widely used to study the sharp interface limit problem in the past decades \cite{dMS2, ABC,FWZZ,AL}. In particular, we find the approach developed by Alikakos-Bates-Chen \cite{ABC} to be useful. This method does not involve the local parametrization of the interface, thus avoiding many complicate calculations near the interface.

\subsection{Presentation of the problem and main results}
We consider a matrix-valued Allen-Cahn equation with a small parameter $\ve$  in a bounded domain $\Omega\subset\mathbb{R}^m$
with periodic boundary conditions, which was introduced in \cite{OW}:
\begin{align}\label{eq:main-intro}
\partial_t\A^{\ve}=\Delta \A^{\ve}-\ve^{-2}f(\A^{\ve}),
\end{align}
where
\begin{align}
f(\A^{\ve})=\A^{\ve}({\A}^{\ve})^{\mathrm{T}}\A^{\ve}-\A^{\ve}\label{intro:def-f}
\end{align}
 with ${(\A^\ve)^{{\mathrm{T}}}}$ denoting the transpose of $\A^{\ve}\in \mathbb{R}^{n\times n}\triangleq \mathbb{M}_n$.
The system \eqref{eq:main-intro}, which was first introduced by \cite{OW}, could be viewed as the gradient flow for the energy functional
\begin{align}
\mathcal{E}(\A,\nabla \A)=\int_{\Omega}\bigg(\frac{1}{2}\|\nabla\A\|^2+\varepsilon^{-2}F(\A)\bigg)\ud x, \quad F(\A)=\frac{1}{4}\|{\A}^{\mathrm{T}}\A-\II\|^2,\label{intro:def-FullEnergy}
\end{align}
where $\II$ is the identity matrix.
In this paper, we are interested in the asymptotical behavior of solutions to the system \eqref{eq:main-intro} when $\ve\to 0$.

Note that $F(\A)$ attains its minimum at the orthogonal group $O(n)={O}^+(n) \cup{O}^-(n)$, where $O^\pm(n)$ denotes the
set of orthogonal matrices with determinant $\pm1$. Formally, one has that, in the limit of $\ve\to 0$, the domain $\Omega$ can be divided into two disjoint parts $\Omega_t^+$ and $\Omega_t^-$ with the property that
\begin{align*}
  \A^\ve(x,t)\to \A_\pm(x,t)\in{O}^\pm(n), \quad \text{ for $x\in \Omega_t^\pm$. }
\end{align*}
Then we need to determine  the evolution of the interface $\Gamma_t=\partial\Omega_t^+\cap\partial\Omega_t^-$, the evolution of $\A_\pm$ on $\Omega_t^\pm$, and the boundary conditions of $\A_\pm$ on $\Gamma$.
Indeed, the limit sharp interface model of the system \eqref{eq:main-intro} takes the form:
\begin{subequations}\label{SharpInterfaceSys}
\begin{align}\label{SharpInterfaceSys:HeatFlow}
&\partial_t \A_\pm=\Delta\A_\pm-\sum_{i=1}^m\partial_i\A_\pm \A_\pm^{\mathrm{T}}\partial_i\A_\pm \quad (\A_\pm\in {O}^\pm(n)),&\text{in}\  \Omega_t^{\pm},\\ \label{SharpInterfaceSys:MCF}
&V=\kappa,&\text{on}\ \Gamma_t,\\\label{SharpInterfaceSys:MiniPair}
&(\A_+, \A_-) \text{ is  a minimal pair},&\text{on}\ \Gamma_t,\\ \label{SharpInterfaceSys:Neumann}
&\frac{\partial \A_+}{\partial \nu}=\frac{\partial \A_-}{\partial \nu}, &\text{on}\ \Gamma_t.
\end{align}
\end{subequations}
Here $V$, $\kappa$, and $\nu$ are the normal velocity, the mean curvature and the unit outward normal vector of $\Gamma_t$ respectively.
The condition \eqref{SharpInterfaceSys:MiniPair} is equivalent to that for each $x\in\Gamma_t$, there exists $\nn(x,t)\in S^{n-1}$ such that
$\A_+= \A_-(\II-2\nn\nn)$.

The heat flow of harmonic maps (\ref{SharpInterfaceSys:HeatFlow})  and the mean curvature flow  (\ref{SharpInterfaceSys:MCF})  for the interface are rather natural,  which have been formally argued in \cite{RSK1, RSK2}.
The relation \eqref{SharpInterfaceSys:MiniPair} has been obtained in \cite{LPW} for minimizing equilibrium solutions.
The boundary condition (\ref{SharpInterfaceSys:Neumann}) is new and special for this matrix-valued Allen-Cahn equation on $O(n)$ due to the underlying geometric structure and properties of minimal connecting orbits (see also \cite{LPW} for different boundary conditions in another geometric setting). We shall  provide a derivation in Appendix \ref{App:BC}.

In the case of $n=2$, the minimal pair condition holds for any $(\A_-, \A_+)\in O^-(2)\times O^+(2)$. So,  the condition (\ref{SharpInterfaceSys:MiniPair}) is redundant. The Neumann-jump condition in this case (\ref{SharpInterfaceSys:Neumann})  is reduced to  the usual Neumann boundary condition for functions on both sides:
\begin{align*}
  \partial_\nu\A_-=\partial_\nu\A_+=0.
\end{align*}
Indeed, in this case, we can write
\begin{align*}
\A_+=\left(
        \begin{array}{cc}
          \cos \alpha_+ & \sin\alpha_+ \\
          -\sin\alpha_+ & \cos \alpha_+
        \end{array}
      \right),\quad \A_-=\left(
        \begin{array}{cc}
          \cos \alpha_- & \sin\alpha_- \\
          \sin\alpha_- & -\cos \alpha_-
        \end{array}
      \right)
\end{align*}
($\alpha_\pm$ are the rotation angles),  then \eqref{SharpInterfaceSys:Neumann}  becomes  to
\begin{align}\label{intro:BC-2d}
\partial_\nu  \alpha_+=\partial_\nu\alpha_-=0.
\end{align}
This boundary condition \eqref{intro:BC-2d} for $n=2$ has been observed in \cite{WOW}, in which the asymptotic dynamics of \eqref{eq:main-intro} under different time scales are studied.

The main goal of this paper is to provide a rigorous justification of the limit from
the regularized system (\ref{eq:main-intro}) to the sharp interface model  (\ref{SharpInterfaceSys}), which is so called the sharp interface limit. The proofs will follow the approach in de Mottoni-Schatzman \cite{dMS2} and Alikakos-Bates-Chen \cite{ABC}: we first construct an approximate solution $ \A^K$ solving \eqref{eq:main-intro} up to sufficiently high order small terms; and then we prove a spectral lower bound for the linearized operator around $\A^K$;  and finally we estimate the difference $\A^\ve-\A^K$.

Our main results are stated as follows. The first main result is concerned with the existence of approximate solutions, whose proof turns out to be most difficult for this article.

\begin{Theorem}\label{thm:main1}
Assume that $(\Gamma_t, \A_+, \A_-)$ is a smooth solution on $[0,T]$ to the sharp interface system \eqref{SharpInterfaceSys}. Then for any $K\ge 1$, there exists $\A^K$ such that
\begin{align}\label{mthm:1}
  \partial_t\A^K=\Delta \A^K-\ve^{-2}f(\A^K)+\mathfrak{R}^{K-1},
\end{align}
where $\mathfrak{R}^{K-1}\sim O(\ve^{K-1})$, and for any $(x,t)$ with $x\in\Omega_t^\pm$, $\|\A^K(x,t)-\A_\pm(x,t)\| \to 0$ as $\ve\to 0$.
\end{Theorem}
\begin{Remark}
The existence of smooth solutions to (\ref{SharpInterfaceSys}) is not a trivial issue due to the complicated  boundary conditions.
We will address it and some related issues in a forthcoming paper.  For a similar vector-valued two phase flow in a different geometric setting and hence with different boundary conditions, the well-posedness has been established by the second author and Wang \cite{LinWang2019}.
\end{Remark}
The second main result is the spectral lower bound for the linearized operator around the approximate solution $\A^K$.
Compared with the proof in the scalar case, the proof is much more involved, and several new ideas are introduced to obtain the desired conclusions.
\begin{Theorem}\label{thm:main2}
Assume that $\A^K(K\ge1)$ is an approximate solution constructed as in Theorem \ref{thm:main1}. Then there exists $C_0\ge 0$ such that, for any $\A\in H^1(\Omega)$ and $t\in[0,T]$ it holds
\begin{align}\label{mthm:2}
\int_\Omega \Big(\frac{1}{2}\|\nabla\A\|^2+\frac{1}{\ve^2}\CH_{\A^K}\A:\A\Big) \ud x\ge -C_0 \int_\Omega \|\A\|^2 \ud x,
\end{align}
where
\begin{align}\label{mthm:2-1}
  \CH_{\BB}\A=\BB\BB^{\mathrm{T}}\A+\A\BB^{\mathrm{T}}\BB+\BB\A^{\mathrm{T}}\BB-\A
\end{align}
is the linearized operator of $f$ defined in (\ref{intro:def-f}).
\end{Theorem}

With the hand of Theorem \ref{thm:main1} and Theorem \ref{thm:main2},  the following nonlinear stability result follows easily via an energy method.

\begin{Theorem}\label{thm:main3}
Let $L=3([\frac{m}{2}]+2)$ and $K=L+1$. Assume that $\A^K$ is an approximate solution constructed as in Theorem \ref{thm:main1}, and  $\A^\ve$ is a solution to (\ref{eq:main-intro}) with initial data $\A^\ve(\cdot,0)$ satisfying
\begin{align*}
\bar{\mathcal{E}}(\A^\ve(\cdot,0)-\A^K(\cdot,0)) \le C_0 \ve^{2L},\quad
\end{align*}
where
\begin{align*}
\bar{\mathcal{E}}(\A)= \sum_{i=1}^{[\frac{m}{2}]+1}\ve^{6i}\int_\Omega\|\nabla^i\A\|^2 \ud x.
\end{align*}
Then there exists constants $\ve_0, C_1>0$ such that for $\forall\ve\le \ve_0$ it holds
\begin{align}
\bar{\mathcal{E}}(\A^\ve(\cdot,t)-\A^K(\cdot,t)) \le C_1 \ve^{2L}, \qquad \text{for }\forall t\in[0,T].
\end{align}
\end{Theorem}

\subsection{Main difficulties and key ideas}\label{subsec:intro-idea}

The proof are based on two key ingredients: construction of approximate solutions and the spectral lower bound estimate.
Compared with the scalar case (e.g. \cite{dMS2, ABC}) or the case that the potential function $F(u)$ is fully minimally paired (e.g. \cite{FWZZ}), there exist several serious difficulties, which we need to introduce various new arguments to overcome. Let us give a sketch here.

Assume $(\Gamma_t, \A_+, \A_-)$ is a smooth solution of \eqref{SharpInterfaceSys} on $[0,T]$. To proceed, let $d(x,\Gamma_t)$ be the signed distance from $x$ to $\Gamma_t$, and $\nu=\nabla d|_{\Gamma_t}$ be the unit outer normal of $\Omega_t^-$. We denote
\begin{align*}
 & \Gamma_t(\delta)=\{x\in\Omega: |d(x,\Gamma_t)|<\delta \},\qquad \Omega_t^\pm=\{x\in\Omega: d(x,\Gamma_t)\gtrless 0\}~~\text{ for }t\in[0,T];\\
 &   \Gamma(\delta)=\{(x,t)\in\Omega\times[0,T]: x\in\Gamma_t(\delta) \},\qquad Q_\pm =\{(x,t)\in\Omega\times[0,T]: x\in\Omega_t^\pm\}.
\end{align*}
We also write $\Gamma=\{(x,t)\in\Omega\times[0,T]:x\in\Gamma_t\}$ for simplicity.
\subsubsection{Construction of approximate solutions}

As in \cite{ABC}, we construct approximate solutions in the two regions: outer region $Q^\pm\backslash \Gamma(\delta/2)$ and inner region $\Gamma(\delta)$, separately.  In the outer region, we assume the solution of \eqref{eq:main-intro} has the expansion
\begin{align}
 \A^\ve(x,t) = \A_O^\ve(x,t):= \sum_{i=0}^{+\infty}\ve^i\A_\pm^{(i)}(x,t),
\end{align}
where $\A_\pm^{(i)}(x,t)$ are smooth functions defined on $ Q_\pm$.

Substituting the above expansion into (\ref{eq:main-intro}) and collecting the terms with same power of $\ve$, we can obtain the equations for $\A_\pm^{(i)}(x,t)$,  where the leading order equation for $\A_\pm^{(0)}(x,t)$ is actually the equation (\ref{SharpInterfaceSys:HeatFlow}).
The boundary/jump condition for $\A_\pm^{(i)}(x,t)$ will be determined to ensure the solvability of  expansions in inner regions (particularly on the interface $\Gamma_t$).
Moreover, we will use the value of $\A_\pm^{(i)}(x,t)(i > 0)$ not only in the domain $ Q_\pm$, but also in the domain $\Gamma(\delta)\backslash Q_\pm$. Hence, we have to extend $\A_\pm^{(i)}(x,t)$ from $Q_\pm$ to $Q_\pm\cup\Gamma(\delta)$.

\smallskip

In the region near the interface, we try to find functions for $(x,t)\in\Gamma(\delta)$
\begin{align}\label{intro:expan:A}
{\A}_{I}^{\ve}(z,x,t)&=\A_{I}^{(0)}(z,x,t)+\ve \A_{I}^{(1)}(z,x,t)+\ve^2 \A_{I}^{(k)}(z,x,t)+\cdots,\\
d^\ve(x,t)&=d_0(x,t)+\ve d_1(x,t)+\ve^2 d_2(x,t)+\cdots,\label{intro:expan:phi}
\end{align}
such that $d^\ve(x,t)$ is a signed distance function with respect to a surface $\Gamma^\ve$ and
\begin{align*}
  \widetilde\A_{I}^{\ve}(x,t)={\A}_{I}^{\ve}\big(\ve^{-1}{d^\ve(x,t)},x,t\big)
\end{align*}
solves the original equation (\ref{eq:main-intro}) in $\Gamma(\delta)$. Here $\A_{I}^{(k)}(z,x,t)$ and $d_k(x,t)$ ($k\ge 0$) are functions independent of $\ve$.  In addition, to match the boundary conditions, we require that
\begin{align}\label{intro:matching condition}
\lim_{z\to\pm\infty} |\partial_t^i\partial^j_x\partial^l_z (\A_{I}^{(k)}(z,x,t)- \A_{\pm}^{(k)}(x,t))|\le O(e^{-\alpha_0 z}),\quad \text{for } (x,t)\in\Gamma(\delta), \quad i,j,l\ge0.
\end{align}
Then the approximate solution on the whole domain $\Omega$ can be constructed by gluing two solutions $\A_O^\ve$ and $\widetilde\A_I^\ve$ in the overlapped region $\Gamma(\delta)\setminus\Gamma(\delta/2)$.

Once we substitute the inner expansion \eqref{intro:expan:A}  into \eqref{eq:main-intro},  the leading order ($O(\ve^{-2})$) system is an ODE system in $z$, which reads as
\begin{align}\label{intro:leading}
-\partial_z^2\A_{I}^{(0)}+f(\A_{I}^{(0)})=0,\quad \A_{I}^{(0)}(\pm\infty,x,t)=\A_\pm^{(0)}(x,t),
\end{align}
for given $(x,t)\in\Gamma(\delta)$.
However, as the discussion in Section \ref{subsec:minipair}, this equation has no solution
unless $\A_{+}^{\mathrm{T}}\A_-$ is symmetric. More importantly, only when $(\A_{+}, \A_-)$ is a minimal pair,
(\ref{intro:leading}) has a stable solution which takes the form
\begin{align}\label{intro:def:AI}
\BTe(\A_+, \A_-; z):=s(z)\A_++(1-s(z))\A_-,\qquad  s(z)=1-(1+e^{\sqrt{2}z})^{-1}.
\end{align}
Such a $\BTe(\A_+, \A_-; z)$ with $(\A_{+}, \A_-)$ being a minimal pair is called a minimal
connecting orbit (see Definition \ref{def:mini-connect-orbit} for precise definition).

The boundary condition (\ref{SharpInterfaceSys:MiniPair}) gives us that $(\A_{+}(x,t), \A_-(x,t))$
forms a minimal pair for $(x,t)\in\Gamma$. A key question is: \smallskip

{\it Whether $\A_\pm(x,t)$ can be smoothly extended to $\Gamma(\delta)$ such that
$(\A_{+}(x,t),$ $\A_-(x,t))$ remains to be a minimal pair for every $(x,t)\in\Gamma(\delta)$?}\smallskip

\noindent Unfortunately, for general solutions of the system \eqref{SharpInterfaceSys},
such an extension does not exist unless $n=2$.
Thus, for $(x,t)\in\Gamma(\delta)\setminus\Gamma$, one can not expect that $(\A_{+}(x,t), \A_-(x,t))$ is a minimal pair,
and consequently one can not find a solution to (\ref{intro:leading}) for all $(x,t)\in\Gamma(\delta)$. Moreover, the equations  of  $\A_I^{(k)}(z,x,t)$ ($k\ge 1$),
whose main part is the linearization of \eqref{intro:leading}, take the form
\begin{align}\label{intro:next}
 \mathcal{L}_{\A_I^{(0)}} \A_I^{(k)}:=-\partial_z^2\A_I^{(k)}+\CH_{\A_I^{(0)}} \A_I^{(k)}=\FF.
\end{align}
If $\A_I^{(0)}= \BTe(\A_+, \A_-; z)$ with $(\A_{+}, \A_-)$ being a minimal pair, the linearized system \eqref{intro:next} can be explicitly solved.
Otherwise, it may have no solutions. Therefore, for $(x,t)\in\Gamma(\delta)\setminus\Gamma$, the construction of $\A_I^{(k)}(z,x,t)$ ($k\ge 1$) can not proceed either. \smallskip

{\bf The solvability of (\ref{intro:leading}) and  (\ref{intro:next}) have thus become major obstacles for the construction of approximate solutions.} As a result, the traditional matched asymptotic expansion method does not work for our problem.
We remark that, in several sharp interface problems such as
the scalar Allen-Cahn problem \cite{dMS2} or the isotropic-nematic interface problem \cite{FWZZ},
such difficulties do not exist. It is because any pair of points from the two potential wells of $F$ is a minimal pair (that is, $F(u)$ is fully minimally paired, see Definition \ref{def:fullmini}), in such cases
any smooth extension will work, and the phase field equations on two sides of the sharp interfaces are decoupled.\smallskip

{\bf Now let us introduce new key ideas.}\smallskip

To solve (\ref{intro:leading}), a key idea is to construct a profile $\A_0(z,x,t)$ which fulfills
the boundary conditions and ``almost" satisfies (\ref{intro:leading}).
Here ``almost" means the remainder
\begin{align}
  \RR(z,x,t)=-\partial_z^2\A_0+f(\A_0)
\end{align}
is small and can be absorbed  into the equations of next orders by using the relation of $z=d^\ve/\ve$. Precisely, we require that
it has the form
\begin{align}\label{intro:requi-1}
\RR(z,x,t)=\GG(z,x,t)d_0=\ve\GG(z,x,t)(z-d_1-\ve d_2-\ve^2 d_3-\cdots),
\end{align}
with $\GG(z,x,t)$ being a bounded smooth function with an exponential decay in $z$.
Then it can be viewed as source terms for the  systems of next orders. Such a profile is not hard to construct, as a direct choice \eqref{intro:def:AI}
fulfills (\ref{intro:requi-1}) even if $(\A_-,\A_+)$ is not a minimal pair.

The requirement ensuring (\ref{intro:next}) to be solvable is much more restrictive.
Applying a similar idea as above, one may expect to find a minimal connecting orbit
$\widetilde{\BTe}(\widetilde\A_+, \widetilde\A_-; z)$ such that the difference
\begin{align*}
 \mathcal{L}_{\A_I^{(0)}} \A_I^{(k)}- \mathcal{L}_{\widetilde{\BTe}} \A_I^{(k)}
\end{align*}
can be absorbed into the  expansions of next order. However, once $\big(\A_I^{(0)}(-\infty),\A_I^{(0)}(+\infty)\big) $ is not a minimal pair,
it is hard, and in general impossible, to find such a $\widetilde{\BTe}$  so that the difference decays exponentially to zero in $z$.
Our idea is to introduce a profile
\begin{align*}
\BTe(z,x,t) =\Phi(z,x,t)\PP_0(z,x,t)\quad \text{with}\quad \PP_0(z,x,t)=\II-2s(z)\nn(x,t)\nn(x,t),
\end{align*}
such that
\begin{align*}
  \Phi(-\infty,x,t)=\A_-(x,t),\quad   \Phi(+\infty,x,t)=\A_+(x,t)(\II-2\nn(x,t)\nn(x,t)),\\
\Phi(z,x,t)\in{O}^-(n),\quad\Phi(z,x,t)\to  \Phi(\pm\infty,x,t) \text{ exponentially as }z\to\pm\infty.
\end{align*}
By a direct checking, one obtains that
\begin{align*}
\partial_z^2\BTe-f(\BTe)=  \partial_z^2\Phi\PP_0+2\partial_z\Phi\partial_z\PP_0.
\end{align*}
A crucial observation is that, if we write
\begin{align*}
 \A(z,x,t)= \Phi(z,x,t)\PP(z,x,t),
\end{align*}
then we have
\begin{align*}
  \mathcal{L}_{\BTe}\A=\Phi\mathcal{L}_{\PP_0}\PP-\partial_z^2\Phi\PP-2\partial_z\Phi\partial_z\PP.
\end{align*}
Since $\PP_0$ is a trivial minimal connecting orbit, $\mathcal{L}_{\PP_0}$ can be explicitly inverted by diagonalizing it into several differential operators $\mathcal{L}_i(1\le i\le 5)$ (see \eqref{pre:decomp-Op-L}) acting on scalar parameter functions (very much modulation); see Section \ref{subsec:pre:diagonal} for details.
Thus, $\mathcal{L}_{\BTe}$ is solvable up to some exponentially decaying terms, which is the key property that enables us to construct solutions for the expansions of next order. Such profiles $\BTe$ are called {\it quasi-minimal connecting orbits},  which play a crucial role throughout the whole construction of inner expansions.

The above procedure increases dramatically the complexity of the inner expansions,
and thus the process for solving the system in the expansion has to be carefully examined.
This will be accomplished in Section \ref{sec:construct} with a sketched illustration in Figure \ref{fig:expansion}.

\subsubsection{Spectral lower bound estimate}

Another key ingredient in our analysis is to prove a spectral lower bound estimate for the linearized operator $-\Delta +\CH_{\A^K}$,
which is clearly more difficult than the scalar case. In $\Gamma(\delta/2)$, our approximate solution takes
\begin{align*}
  &\A_I^K(x,t)=\Phi(\ve^{-1}d^K,x,t)\sum_{k=0}^{K} \ve^k \PP_k(\ve^{-1}d^K,x,t),
 \qquad \text{with}~  d^K(x,t)=\sum_{k=0}^{K} \ve^k d_k(x,t).
   \end{align*}

First of all, it suffices to prove the following inequality in the inner region  $\Gamma_t^K(\delta/4)$ (as it holds  on the outer region clearly):
\begin{align}
\int_{\Gamma_t^K(\delta/4)}\Big(\|\nabla\A\|^2 +\ve^{-2}\big(\mathcal{H}_{\A_0}\A:\A\big)
+\ve^{-1}\big(\TT_f(\A_0,\A_1,\A):\A\big)\Big)\ud x \geq-C\int_{\Gamma_t^K(\delta/4)}\|\A\|^2\ud x,
\label{spectral inequality-inner-intro}
\end{align}
where the trilinear form $\TT_f$ is defined by \eqref{def:BB} and $\A_i=(\Phi\PP_i)(\ve^{-1}d^K,x,t)$ for $i=0,1$.

We then introduce a coordinate transformation $x\mapsto (\sigma, r)$, which  is a diffeomorphism from $\Gamma_t^K(\delta/4)$ to $\Gamma_t^K\times (-\delta/4, \delta/4)$, and let $J$ be the Jacobian of the transformation, and let $\BB(x,t)=\Phi^{\mathrm{T}}\big(\ve^{-1}d^K(x,t),x,t\big)\A(x,t)$. Then \eqref{spectral inequality-inner-intro} is equivalent to
\begin{align}\label{ineq:one-dim-2-intro}
&\int_{-\frac{\delta}{4}}^{\frac{\delta}{4}}\Big(\|\partial_r \BB\|^2 +\ve^{-2}\mathcal{H}_{\PP_0}\BB:\BB
+\ve^{-1}\TT_f(\PP_0,\PP_1,\BB):\BB+2\Phi\partial_r \BB : \partial_r \Phi\BB \Big)J\ud r\nonumber\\
&\geq-C\int_{-\frac{\delta}{4}}^{\frac{\delta}{4}} \|\BB\|^2J\ud r.
\end{align}
Hence Theorem \ref{thm:main2} can be deduced from the following estimate for the cross term:
\begin{align}\label{ineq:1-dim-cross-intro}
&\int_{-\frac{\delta}{4}}^{\frac{\delta}{4}} \Phi\partial_r \BB : \partial_r \Phi\BB J\ud r
\le
\frac{1}{4}\int_{-\frac{\delta}{4}}^{\frac{\delta}{4}}\Big(\|\partial_r \BB\|^2 +\ve^{-2}\mathcal{H}_{\PP_0}\BB:\BB\Big)J\ud r+C\int_{-\frac{\delta}{4}}^{\frac{\delta}{4}} \|\BB\|^2J\ud r,
\end{align}
and the control of singular correction term of the next order:
\begin{align} \label{ineq:1-dim-nextorder-intro}
&\int_{-\frac{\delta}{4}}^{\frac{\delta}{4}}\ve^{-1}\TT_f(\PP_0,\PP_1,\BB):\BB J\ud r
\le
\frac{1}{4}\int_{-\frac{\delta}{4}}^{\frac{\delta}{4}}\Big(\|\partial_r \BB\|^2 +\ve^{-2}\mathcal{H}_{\PP_0}\BB:\BB\Big)J\ud r+C\int_{-\frac{\delta}{4}}^{\frac{\delta}{4}} \|\BB\|^2J\ud r.
\end{align}

The verifications of (1.23) and (1.24) are based on several ingredients including: decomposing into scalar inequalities which are  related to the scalar linearized operators $\{\mathcal{L}_i(1\le i\le5)\}$, coercive estimates and endpoints $L^\infty$-control for these operators, and some novel
 product estimates which rely on important special symmetric structures between their first eigenfunctions. These will be carried out in Section 7.

We remark that, in proofs of coercive estimates,
endpoints $L^\infty$-controls and product estimates, we have repeatedly applied elementary decompositions based on the eigenfunctions of scalar linearized operators $\{\mathcal{L}_i(1\le i\le5)\}$. This method gives new and elementary proofs for the spectral estimates of these operators, which do not rely on the maximum/comparison principle or the Harnack inequality, and might have their own interests.

\subsection{Notations}

\begin{itemize}
\item  For any two matrices $\A$ and $\B$, we denote $\A : \B=\A_{ij}\B_{ij}$ and $\A\perp \B$ means $\A : \B=0$. Let $\|\A\|^2=\A : \A$.

\item   For any two vectors $\mathbf{m}$ and $\mathbf{n}$, we use $\mathbf{m}\mathbf{n}$ to denote $\mathbf{m}\otimes\mathbf{n}$ when no ambiguity is possible.

\item  We use $\nn\A$ to denote the vector $(\nn_j\A_{ji})_{1\le i\le n}$, and $\A\nn$ to denote the vector $(\A_{ij}\nn_j)_{1\le i\le n}$. Then $\A\mm\nn$ is understood as $(\A_{ik}\mm_k\nn_j)_{1\le i, j\le n}$ and similarly $\mm\nn\A=(\mm_i\nn_k\A_{kj})_{1\le i, j\le n}$.

\item $\mathbb{M}_n$: the space of $n\times n$ matrices.
\item $\mathbb{S}_n$, $\mathbb{A}_n$: the spaces of symmetric and asymmetric $n\times n$ matrices.
\item ${O}(n), {O}^\pm(n):$ $n\times n$ orthogonal group, the set of $n\times n$ orthogonal matrices with determinant $\pm1$.
\item {$O(e^{-\alpha_0|z|})$: denotes the terms can be bounded by $C|z|^ke^{-\alpha_0|z|}$ for some $k\ge 0$ as $z\to\infty$.}
\end{itemize}
The following simple fact will be constantly used:
\begin{align}
\text{for }\A,\BB, \CC\in \mathbb{M}_n,\quad (\A\BB):\CC=\A:(\CC\BB^{\mathrm{T}})=\BB:(\A^{\mathrm{T}}\CC).\label{eq:ABC}
\end{align}

\medskip

\section{Outer expansion}\label{sec:outer}

\subsection{Formal outer expansion}

We perform {\it outer expansion} in $Q_\pm$ rather than $Q_\pm\setminus\Gamma(\delta/2)$
by using the Hilbert expansion method as in \cite{WZZ1, WZZ2}. Assume that
\begin{align}\label{out:expansion}
  \A^\ve(x,t) = \sum_{i=0}^{+\infty}\ve^i\A_\pm^{(i)}(x,t),\qquad \text{ for }(x,t)\in Q_\pm.
\end{align}
Substituting it into (\ref{eq:main-intro}), it is easy to find the leading order $O(\ve^{-2})$ equation reads as
\begin{align*}
\A_\pm^{(0)}(\A_\pm^{(0)})^{\mathrm{T}}\A_\pm^{(0)}=\A_\pm^{(0)},
\end{align*}
which is satisfied by taking
\begin{align}\label{out:sol-A0}
\A_\pm^{(0)}=\A_\pm\in {O}^\pm(n).
\end{align}

Now we assume that
\begin{align*}
    \A^\ve(x,t) = \A_\pm  \UU^\ve(x,t) = \A_\pm\sum_{i=0}^{+\infty}\ve^i\UU_\pm^{(i)}(x,t),\qquad \text{ for }(x,t)\in Q_\pm.
\end{align*}
Here $\UU_\pm^{(0)}=\II$.  A direct calculation yields that
\begin{align}\nonumber
\A_\pm^\mathrm{T}f(\A^{\ve})&= f(\UU^\ve)\\\label{out:expan-f}
&=\ve \big(\UU_\pm^{(1)} +(\UU_\pm^{(1)})^{\mathrm{T}}\big)+
\sum_{k\ge 1}\ve^{k+1} \Big(\big(\UU_\pm^{(k+1)} +(\UU_\pm^{(k+1)})^{\mathrm{T}}\big)+\BB_\pm^{(k)}+{\CC}_\pm^{(k-1)}\Big),
\end{align}
where
\begin{align}
\label{def:Bpmk}
&\BB_\pm^{(k)}=\sum_{\{i,j, l\}=\{0,1,k\}}\UU_\pm^{(i)}\big(\UU_\pm^{(j)}\big)^{\mathrm{T}}\UU_\pm^{(l)},\\
&{\CC}_\pm^{(k-1)}=\sum_{\substack{i+j+l=k+1,\\0\le i,j,l\le k-1}}
\UU_\pm^{(i)}\big(\UU_\pm^{(j)}\big)^{\mathrm{T}}\UU_\pm^{(l)}.\label{def:Cpmk}
\end{align}
Here and in what follows we use the convention that $\UU_\pm^{(i)}=0$ for $i<0$. Note that
\begin{align}\label{out:def-C1}
  {\CC}_\pm^{(0)}=0,\quad {\CC}_\pm^{(1)}=\UU_\pm^{(1)}\big(\UU_\pm^{(1)}\big)^{\mathrm{T}}\UU_\pm^{(1)},
\end{align}
and ${\CC}_\pm^{(k-1)}$ only involves  $\UU_\pm^{(0)}$, $\UU_\pm^{(1)}$, ..., $\UU_\pm^{(k-1)}$.

On the other hand, we define the linear operator
\begin{align}\nonumber
  \mathcal{J}_\pm \Bp=&\A_\pm^\mathrm{T}\Big(\partial_t(\A_\pm\Bp) -\Delta(\A_\pm\Bp)\Big)\\\label{out:def-Op-J}
  =&\partial_t\Bp -\Delta \Bp +\A_\pm^\mathrm{T}(\partial_t\A_\pm -\Delta\A_\pm)\Bp
-2\A_\pm^\mathrm{T}\nabla\A_\pm\nabla\Bp.
\end{align}
Then we have
\begin{align}\label{out:expan-DA}
\A_\pm^\mathrm{T}(\partial_t\A^\ve -\Delta \A^\ve )= \mathcal{J}_\pm \UU^\ve.
\end{align}

Substituting \eqref{out:expan-f} and \eqref{out:expan-DA} into (\ref{eq:main-intro}), and then equating the $O(\ve^k)(k\geq-1)$ system yields that
\begin{align} \label{expan:out-pm1}
O(\ve^{-1}):&&\UU_\pm^{(1)} +\big(\UU_\pm^{(1)}\big)^{\mathrm{T}}&=0,\\
O(\ve^{k})(k\ge0):&&\UU_\pm^{(k+2)} +\big(\UU_\pm^{(k+2)}\big)^{\mathrm{T}}&=-\mathcal{J}_\pm\UU_\pm^{(k)}
-\BB_\pm^{(k+1)}-{\CC}_\pm^{(k)}.
\label{expan:out-pmk}
\end{align}

In the sequel, we will use the decomposition
\begin{align*}
\MM_\pm^{(k)}=\frac{1}{2}\big(\UU_{\pm}^{(k)}+(\UU_{\pm}^{(k)})^{\mathrm{T}}\big)\,\,\in\mathbb{S}_n,\quad
\VV_\pm^{(k)}=\frac{1}{2}\big(\UU_{\pm}^{(k)}-(\UU_{\pm}^{(k)})^{\mathrm{T}}\big)\,\,\in\mathbb{A}_n,
\end{align*}
and solve $\MM_\pm^{(k)}, \VV_\pm^{(k)}$ separately.

\subsection{The leading order equation}
It yields from (\ref{expan:out-pm1})  that
\begin{align}\label{out:sol-M1}
  \MM_\pm^{(1)}=0.
\end{align}

The equation (\ref{expan:out-pmk}) for $k=0$ gives us that
\begin{align*}
2\MM_\pm^{(2)}&=-\mathcal{J}_\pm\UU_\pm^{(0)}-\BB_\pm^{(1)}.\end{align*}
which leads to
\begin{align*}
&\mathcal{J}_\pm\UU_\pm^{(0)}+\BB_\pm^{(1)} \in \mathbb{S}_n.
\end{align*}
Since $ \MM_\pm^{(1)}=0$, we have that $\UU_\pm^{(1)}=\VV_\pm^{(1)}$ is antisymmetric, and thus
\begin{align*}
 \BB_\pm^{(1)}&=-(\VV_\pm^{(1)})^2  \in \mathbb{S}_n.
\end{align*}
 which implies
\begin{align}\label{equation:Apm0}
\mathcal{J}_\pm\UU_\pm^{(0)}\in \mathbb{S}_n.
\end{align}
As $\UU_\pm^{(0)}=\II$, we deduce from \eqref{out:expan-DA}  that
\begin{align*}
  \A_\pm^\mathrm{T}(\partial_t\A_\pm -\Delta\A_\pm)\in \mathbb{S}_n,
\end{align*}
which is actually equivalent to {\bf  the heat flow of harmonic maps to ${O}^\pm(n)$} given in \eqref{SharpInterfaceSys:HeatFlow}.

\subsection{The next order equations}
For general $k$, (\ref{expan:out-pmk}) can be equivalently written as
\begin{align}\label{eq:in-Ak}
& \mathcal{J}_\pm\UU_\pm^{(k)}
+\BB_\pm^{(k+1)}+{\CC}_\pm^{(k)}\in \mathbb{S}_n,\\ \label{eq:out-Ak}
&\MM_\pm^{(k+2)}=-\frac12 \Big(\mathcal{J}_\pm\UU_\pm^{(k)}+\BB_\pm^{(k+1)}+{\CC}_\pm^{(k)}\Big).
\end{align}
The second equation implies that $\MM_\pm^{(i)}$ is uniquely determined from $\A_\pm^{(j)}(0\le j\le i-1)$ for $i\ge 2$.

From the definition (\ref{def:Bpmk}) of $\BB_\pm^{(k)}$  and \eqref{out:sol-M1}, we have
\begin{align}\label{out:def-B1}
\BB_\pm^{(1)}&=\VV_\pm^{(1)}(\VV_\pm^{(1)})^{\mathrm{T}}=-(\VV_\pm^{(1)})^2\in \mathbb{S}_n,
\end{align}
and for  $k\ge 2$,
\begin{align}\label{out:def-Op-B}
  \BB_\pm^{(k)}
  &=\VV_\pm^{(1)}(\UU_\pm^{(k)})^{\mathrm{T}}+(\UU_\pm^{(k)})^{\mathrm{T}}\VV_\pm^{(1)}\triangleq \CB_\pm \UU_\pm^{(k)}.
\end{align}
One can directly verify that
\begin{align} \label{property:Bk}
  \CB_\pm \VV \in \mathbb{S}_n,\quad \text{ for }\VV\in\mathbb{A}_n.
\end{align}

Therefore, it follows from (\ref{eq:in-Ak}) and (\ref{property:Bk}) that
\begin{align*}
\mathcal{J}_\pm\UU_\pm^{(k)}
+{\CB}_\pm\MM_\pm^{(k+1)}+{\CC}_\pm^{(k)}\in \mathbb{S}_n.
\end{align*}
which further gives
\begin{align}\label{out:eq-JUk}
\mathcal{J}_\pm\UU_\pm^{(k)}
-\frac12{\CB}_\pm\Big(\mathcal{J}_\pm\UU_\pm^{(k-1)}+  \BB_\pm^{(k)}+{\CC}_\pm^{(k-1)}\Big)+{\CC}_\pm^{(k)}\in \mathbb{S}_n.
\end{align}
For $k\ge2$, once $\{\UU_\pm^{(i)}\}_{0\le i\le k-1}$ is determined, the above equation
 indeed gives a heat flow type evolution equation for $\VV_\pm^{(k)}$:
\begin{align}\label{eq:Vk}
\mathcal{J}_\pm\VV_\pm^{(k)}-\frac{1}{2}\CB_\pm^2 \VV_\pm^{(k)}+{\CC}_\pm^{(k)}
+\mathcal{J}_\pm\MM_\pm^{(k)}-\frac12{\CB}_\pm\Big(\mathcal{J}_\pm\UU_\pm^{(k-1)}+ \CB_\pm \MM_\pm^{(k)}+{\CC}_\pm^{(k-1)}\Big)\in \mathbb{S}_n,
\end{align}
 since $\MM_\pm^{(k)}$ is already given by \eqref{eq:out-Ak}.

In addition, when $k\ge 3$, the equation (\ref{eq:Vk}) is linear for $\VV_\pm^{(k)}$, since the coefficients in the operator $\CB_\pm$ (see \eqref{out:def-Op-B}) and ${\CC}_\pm^{(k)}$
depend only on $\UU_\pm^{(1)}$ and $\UU_\pm^{(2)}$.

For $k=1$ or $2$,  (\ref{out:eq-JUk}) or (\ref{eq:Vk}) seems to be nonlinear at a first glance.
However, by a careful checking, we could find that it is also a linear equation for $k=1$ or $2$.

Indeed, for $k=1$, from (\ref{out:def-B1}), (\ref{out:def-Op-B}) and (\ref{out:def-C1}) one has that
\begin{align*}
\CB_\pm \BB_\pm^{(1)}
&=\VV_\pm^{(1)}(\BB_\pm^{(1)})^{\mathrm{T}}+(\BB_\pm^{(1)})^{\mathrm{T}}\VV_\pm^{(1)}=-2(\VV_\pm^{(1)})^3=2\CC_\pm^{(1)}.
\end{align*}
Thus, the equation (\ref{out:eq-JUk}) for $k=1$ is reduced to
\begin{align}\label{eq:V1}
  \mathcal{J}_\pm\VV_\pm^{(1)}
-\frac12{\CB}_\pm\mathcal{J}_\pm\UU_\pm^{(0)}\in \mathbb{S}_n,
\end{align}
which is apparently a linear equation for  $\VV_\pm^{(1)}$.

For $k=2$, the only nonlinear (in $\VV_\pm^{(2)}$) terms are contained in ${\CC}_\pm^{(2)}$, which can be written as
\begin{align*}
&\UU_\pm^{(0)}\big(\UU_\pm^{(2)}\big)^{\mathrm{T}}\UU_\pm^{(2)}+\UU_\pm^{(2)}\big(\UU_\pm^{(0)}\big)^{\mathrm{T}}
\UU_\pm^{(2)}+\UU_\pm^{(2)}\big(\UU_\pm^{(2)}\big)^{\mathrm{T}}\UU_\pm^{(0)}\\
&=\big(\VV_\pm^{(2)}\big)^{\mathrm{T}}\VV_\pm^{(2)}+\VV_\pm^{(2)}\big(\VV_\pm^{(2)}\big)^{\mathrm{T}}+(\VV_\pm^{(2)})^2+\text{ linear terms }\\
&=\text{ symmetric terms }+\text{ linear terms }.
\end{align*}
Therefore, by eliminating symmetric terms, (\ref{eq:Vk}) for $k=2$ is indeed a linear equation of $\VV_\pm^{(2)}$.

Note that for each $\A_\pm^{(k)}$ or $\UU_\pm^{(k)}$, the symmetric part $\MM_\pm^{(k)}$ is solved explicitly from \eqref{eq:out-Ak}. Thus, we do not need boundary/jump conditions for $\MM_\pm^{(k)}$ on $\Gamma_t$.
While, the antisymmetric part $\VV_\pm^{(k)}$ is solved from a linear heat-flow type equation, thus
their boundary/jump conditions on $\Gamma_t$ are needed.
These conditions will be determined in the inner expansion to ensure that outer/inner expansions match each other in the overlap region.

Once $\A_\pm|_{Q_\pm}$, $\MM_\pm^{(k)}|_{Q_\pm}$, $\VV_\pm^{(k)}|_{Q_\pm}$ are determined, we extend them to $\Gamma(\delta)$, such that
\begin{align*}
\A_\pm:\Gamma(\delta)\cup Q_\pm\mapsto {O}^\pm(n),\quad \MM_\pm^{(k)}:\Gamma(\delta)\cup Q_\pm\mapsto \mathbb{S}_n, \quad \VV_\pm^{(k)}:\Gamma(\delta)\cup Q_\pm\mapsto \mathbb{A}_n
\end{align*}
are all smooth functions.
Then $\A_{\pm}^{(k)}(x,t)=\A_\pm(\MM_\pm^{(k)}+\VV_\pm^{(k)})$ are also smooth in $\Gamma(\delta)\cup Q_\pm$.

\medskip

 \section{Minimal pair and quasi-minimal connecting orbits}\label{sec:inner1}

 \subsection{Motivation}

The aim of inner expansion is to find a good approximation, up to any order of $\ve$,
to the exact solution in the region $\Gamma(\delta)$ near the interface. The main strategy used here is that, we try to find functions for $(x,t)\in\Gamma(\delta)$
\begin{align}\label{expan:A}
{\A}_{I}^{\ve}(z,x,t)&=\A_{I}^{(0)}(z,x,t)+\ve \A_{I}^{(1)}(z,x,t)+\ve^2 \A_{I}^{(k)}(z,x,t)+\cdots,\\
d^\ve(x,t)&=d_0(x,t)+\ve d_1(x,t)+\ve^2 d_2(x,t)+\cdots,\label{expan:phi}
\end{align}
such that $d^\ve(x,t)$ is a signed distance function with respect to a surface $\Gamma_t^\ve$ and
\begin{align*}
  \widetilde\A_{I}^{\ve}(x,t)={\A}_{I}^{\ve}\big(\ve^{-1}{d^\ve(x,t)},x,t\big)
\end{align*}
solves the original equation \eqref{eq:main-intro} in $\Gamma_t(\delta)$.  In addition, we require that, for some $\alpha_0>0$,
\begin{align}\label{matching condition}
 |\partial_t^i\partial^j_x\partial^l_z (\A_{I}^{(k)}(z,x,t)- \A_{\pm}^{(k)}(x,t))|= O(e^{-\alpha_0 |z|})\quad \text{ as }{z\to\pm\infty},\text{ for }(x,t)\in\Gamma(\delta).
\end{align}

Since we would like to approximate the sharp interface system (\ref{SharpInterfaceSys}), $d_0(x,t)$ is naturally taken as the signed distance function to $\Gamma_t$. Thus, $\nabla d_0\cdot\nabla$  on $\Gamma_t$ is the normal derivative $\partial_\nu$.
In addition, if $\ve$ is sufficiently small, $\Gamma_t^\ve$ should be a good approximation of $\Gamma_t$. As $d^\ve$ is a signed distance function, one has $|\nabla d^{\varepsilon}|^2=1$, which gives
\begin{align} \label{inner:dist-01}
&|\nabla d_0|^2=1,\quad \nabla d_0\cdot\nabla d_1=0,\\ \label{inner:dist-k}
&2\nabla d_0\cdot\nabla d_k=
-\sum\limits_{1\le j\le k-1}\nabla d_j\cdot\nabla d_{k-j},\quad k\geq2.
\end{align}

Substituting the expansions (\ref{expan:A})-(\ref{expan:phi}) into the following equation
\begin{align*}
\partial_t\A^{\ve}=\Delta\A^{\ve}-\ve^{-2}f(\A^{\ve}),
\end{align*}
we  find that, to eliminate the leading $O(\ve^{-2})$ order terms, $\A=\A_I^{(0)}$ should be a solution of
\begin{align}\label{eq:leading}
  \partial_{z}^2\A =&~f(\A),\qquad \A(\pm\infty)=\A_\pm(x,t).
\end{align}
In addition, the $O(\ve^{k-2})$ $(k\ge1)$ system gives that  $\A_I^{(k)}$ satisfies an equation with the form
\begin{align}\label{eq:linear-leading}
\mathcal{L}_{\A}  \A_I^{(k)}=\FF,\qquad \A_I^{(k)}(\pm\infty)=\A_\pm^{(k)}(x,t),
\end{align}
where $\FF$ contains lower order terms, and $\mathcal{L}_{\A}$ is the linearized operator of \eqref{eq:leading} around $\A$ defined by
\begin{align}
 \mathcal{L}_{\A}\Psi:= -\partial_z^2\Psi+\A\A^{\mathrm{T}}\Psi+\A\Psi^{\mathrm{T}}\A+\Psi\A^{\mathrm{T}}\A-\Psi.
\end{align}
Therefore, the existence of solutions to the systems \eqref{eq:leading} and \eqref{eq:linear-leading} are at the heart of the inner expansion.
As we will show in Section \ref{subsec:minipair} and Section \ref{subsec:pre:diagonal}, when $(\A_+, \A_-)$ is a minimal pair, or equivalently, $\A_+=\A_-(\II-2\nn\nn)$ for some $\nn\in{S}^{n-1}$,
one can directly find a solution $\A$ to \eqref{eq:leading} with \eqref{eq:linear-leading} solvable.\smallskip

The boundary condition (\ref{SharpInterfaceSys:MiniPair}) gives us that $(\A_{+}(x,t), \A_-(x,t))$
forms a minimal pair for $(x,t)\in\Gamma$. However, after a smooth extension, $(\A_{+}(x,t), \A_-(x,t))$ may not
be a minimal pair in general for $(x,t)\in\Gamma(\delta)$. {This causes the main obstacle to the construction of approximated solutions in the inner expansion}.
To overcome this difficulty, we construct a solution $\BTe$ which satisfies \eqref{eq:leading} up to some ``good" remainders, which decay exponentially fast in $z$-variable and vanish on $\Gamma$. More importantly, the corresponding $\mathcal{L}_\BTe$ is also solvable up to some ``good" remainders. Such a solution is called {\it quasi-minimal connecting orbit} (see  Section \ref{subsec:quasi-minimal}).

\subsection{Minimal pair and minimal connecting orbits}\label{subsec:minipair}
We start from a general nonnegative smooth potential function $F: \mathbb{R}^N\to\mathbb{R}_{\ge 0}$ which vanishes exactly on two disjoint, compact, connected, smooth
Riemannian submanifolds $\Sigma^\pm\subset\mathbb{R}^N$ without boundaries. A simple choice of such potential function $F(u)$ is giving by the square of the distance from $u$ to $\Sigma^+\cup\Sigma^-$, for $u$ near $\Sigma^+\cup\Sigma^-$, otherwise it could be a positive constant, see for example \cite{LPW}.
Giving two points $p_{\pm}\in \Sigma^\pm$, the solution of the following ODE
\begin{align}\label{equation:hetero}
\partial_z^2 u=\partial_uF,\qquad u(\pm\infty)=p_\pm,
\end{align}
describes the way of phase transition from the state $p_-$ to another state $p_+$. The existence of solutions to \eqref{equation:hetero}
is so called the {\bf heteroclinic connection problem}, which has been studied extensively for the case of $\Sigma_\pm=\{p_\pm\}$; see \cite{MS, ZS} and the references therein for examples.

\begin{Definition}\label{def:mini-connect-orbit}
A solution of \eqref{equation:hetero} is called a \emph{connecting orbit}, and $p_\pm$ are called its \emph{ends}.
\end{Definition}
In particular, we are interested in {minimal connecting orbits}, which is defined as follows.
\begin{Definition}
A solution of \eqref{equation:hetero} is called a \emph{minimal connecting orbit} \cite{LPW}, if it minimizes
the following energy
\begin{align}
\int_\BR  \frac{1}{2}|u'|^2+ F(u)\ud z, \qquad u(\pm\infty)\in\Sigma_\pm.
\end{align}
\end{Definition}
\begin{Remark}Let the trajectory of $u$ be defined as $\mathrm{Traj}(u)=\{u|-\infty <z<+\infty\}$.
By using an argument as in \cite{ZS} (see also discussions in \cite {LPW}), one can show that $u(z)$ is a minimal connecting orbit, only if:
\begin{enumerate}
  \item The closure of $\mathrm{Traj}(u)$ is a minimal geodesic curve with the weight $\sqrt{F/2}$ in $\mathbb{R}^N$;
  \item $\mathrm{Traj}(u)$ contains no other points in $\Sigma_-\cup\Sigma_+$.
\end{enumerate}
 Conversely, any minimal geodesic curve connecting $(p_-,p_+)\in\Sigma_-\times\Sigma_+ $ which contains no other points in $\Sigma_-\cup\Sigma_+$, is
trajectory of a minimal connecting orbit.
\end{Remark}

With some mild assumptions on the potential function $F$, the existences of minimal connecting orbits can be proved by variational methods
as in \cite{LPW, MS, ZS}. However, it is well-known that, for a general given pair $(p_-, p_+)\in \Sigma_-\times\Sigma_+$,
a minimal connecting orbit (and even a connecting orbit) with ends $p_\pm$ may not exist.
\begin{Definition}
 A pair $(p_-,p_+)\in\Sigma_-\times\Sigma_+$ connected by a minimal connecting orbit is called a \emph{minimal pair}.
\end{Definition}
\begin{Definition}\label{def:fullmini}
If any pair $(p_-,p_+)\in\Sigma_-\times\Sigma_+$ is a minimal pair, we say that $F$ is \emph{fully minimally paired}. Otherwise, i.e., if
there exists $(p_-,p_+)\in\Sigma_-\times\Sigma_+$ which is not a minimal pair, we say that $F$ is \emph{partially minimally paired}.
\end{Definition}

Most of classic models studied previously are fully minimally paired. For examples, for the scalar Allen-Cahn energy $F=\frac14(1-u^2)^2:\mathbb{R}\to\mathbb{R}_{\ge 0}$,
$\Sigma_{\pm}=\{\pm1\}$. Obviously, $F$ is fully minimally paired.
 For the isotropic-nematic phase transition problem in liquid crystals \cite{FWZZ, LW, LL}, the energy $F:\mathbb{Q}\to\mathbb{R}$ ($\mathbb{Q}$ denotes the space of $3\times 3$ symmetric trace free matrices) takes the form:
 \begin{align*}
   F(\QQ)=\frac{a}{2}|\QQ|^2-\frac{b}{3}\mathrm{tr} \QQ^3+\frac{c}{4}|\QQ|^4,\quad  a,b,c>0,\, b^2=27ac.
   \end{align*}
One has
$$\Sigma_{-}=\{0\}, \quad \Sigma_+=\Big\{s_+\Big(\nn\nn-\frac13\II\Big): \nn\in{S}^2, s_+=\frac{b+\sqrt{b^2-24ac}}{4c}\Big\}.$$
As $(0,\QQ_*)$ is a minimal pair for  $\forall~ \QQ_*\in\Sigma_+$, $F$ is fully minimally paired.

 Another geometric example is that $\Sigma_{\pm}$ are the linked spheres $S^k, S^l$ in $S^{k+l+1}$:
 \begin{align*}
 F(u)=\Big((|u_1|^2-1)^2+|u_2|^2\Big)\Big(|u_1|^2+(|u_2|^2-1)^2\Big),\qquad u=(u_1, u_2)\in \mathbb{R}^{k+1}\times\mathbb{R}^{l+1}.
 \end{align*}
Clearly, this problem is fully minimally paired.
 On the other hand, the class of $F(u)$ considered in \cite{LPW} is at the exactly other extremum. In the latter case, there are two compact, smooth sub-manifolds $M_{\pm}$ of $\Sigma_{\pm}$ respectively, such that there is a smooth diffeomorphism between points $p_-,p_+$ in $M_{\pm}$ so that the corresponding points $p_-,p_+$ form a minimal pair.

For the problem (\ref{eq:main-intro}) considered in this paper, $F=\frac{1}{4}\|\A\A^{\mathrm{T}}-\II\|^2$, and the equation \eqref{equation:hetero} becomes
\begin{align}\label{equation:ODE-0}
&\partial_{z}^2 \A = \A\A^{\mathrm{T}}\A-\A ,\qquad \A(\pm\infty)= \A_{\pm}\in {O}^\pm(n),
\end{align}
which is the Euler-Lagrange equation to the one dimensional energy functional:
\begin{align}
\int_\BR  \frac{1}{2}\|\partial_z\A\|^2+ \frac{1}{4}\|\A\A^{\mathrm{T}}-\II\|^2\ud z.
\end{align}
In this case, we have the following characterization of minimal pairs, which implies that this problem is partially minimally paired unless $n=2$.
\begin{Lemma}\label{lem:mini-pair}
A pair $(\A_-, \A_+)\in{O}^-(n)\times{O}^+(n)$ is  a \emph{minimal pair}, if and only if
\begin{align*}
\|\A_--\A_+\|=\min\big\{ \|\A-\B\|: (\A, \BB)\in{O}^-(n)\times{O}^+(n)\big\}.
\end{align*}
or equivalently, there exists $\nn\in{S}^{n-1}$ such that $\A_-=\A_+(\II-2\nn\nn).$
\end{Lemma}
\begin{proof}
See Appendix \ref{App:proof-lem-minipair}.
\end{proof}
Let
\begin{align}\label{equation:s}
  s(z)=1-(1+e^{\sqrt{2}z})^{-1},
\end{align}
which solves
\begin{align}\label{eq:s}
  s''=2s(1-s)(1-2s)\quad \text{for }z\in\mathbb{R};\quad s(+\infty)=1,~ s(-\infty)=0.
\end{align}
Apparently, $s(z)-s(\pm\infty), \partial_z^ks=O(e^{-\sqrt{2}|z|})$ as $z\to\pm\infty$. In the sequel, we will choose $\alpha_0\in(0,\sqrt{2})$.
\begin{Lemma}\label{lem:minipath}
All minimal connecting orbits are given by
\begin{align}\label{def:AI}
  \BTe_0(\A_+, \A_-; z):=s(z)\A_++(1-s(z))\A_-,
\end{align}
with $(\A_+, \A_-)$ being a minimal pair and $s(z)$ defined in \eqref{equation:s}.
\end{Lemma}
\begin{Remark}
We remark that, for any pair $(\A_-,\A_+)\in{O}^-(n)\times{O}^+(n) $ with $\A_+^\mathrm{T}\A_-$ symmetric,
$\BTe_0(\A_+, \A_-; z)$ defined in (\ref{def:AI}) is a solution to (\ref{equation:ODE-0}), i.e., a connecting orbit.
However, only if $\A_+^\mathrm{T}\A_-=\II-2\nn\nn$ for some $\nn\in{S}^{n-1}$, $\BTe_0(\A_+, \A_-; z)$ is minimal connecting orbit. In this example, the dimension of ${O}(n)$ is $ n(n-1)/2$, its symmetric group is
$S^{n-1}$. Thus for every point $p_+$ in $O^+(n)$ there is an embedded  $S^{n-1}$ in $O^-(n)$ which is minimum (and equal) distance to $p_+$. The similar statement is also true in the other way for points in $O^-(n)$. This is a rather interesting (and also typical at least locally) partially minimally paired situation. It leads to many mixed type boundary conditions.
\end{Remark}

\subsection{Quasi-minimal connecting orbit}\label{subsec:quasi-minimal}
As discussed in the beginning of this section, in general, for $(x,t)\in\Gamma(\delta)\setminus\Gamma$, one can not expect that
 $(\A_-(x,t), \A_+(x,t))$  is a minimal pair. Thus, the solution to (\ref{eq:leading}) may not exist. To this end, we construct a profile
$\BTe$ which approximately satisfies (\ref{eq:leading}) for $(x,t)\in\Gamma(\delta)$.

We assume that
there exists a smooth vector field $\nn(x,t):\Gamma\to S^{n-1}$ such that $\A_-=\A_+(\II-2\nn\nn)$ on $\Gamma$ (in general this assumption may be not true and this issue will be discussed in Remark \ref{Remark:lift}).
Then we extend $\nn(x,t)$ to be a smooth ${S}^{n-1}$-valued vector field in $\Gamma(\delta)$ with $\partial_\nu\nn(x,t)=0$ on $\Gamma$.
Define smooth orthogonal matrices
\begin{align*}
\Phi_-(x,t)=\A_-(x,t),\quad \Phi_+(x,t)=\A_+(x,t)(\II-2\nn(x,t)\nn(x,t)),\quad \text{ for }(x,t)\in\Gamma(\delta).
\end{align*}
It holds that $ \Phi_+(x,t)=\Phi_-(x,t)$ on the interface $\Gamma$.
Moreover, as $\partial_\nu\nn(x,t)=0$ on $\Gamma$, the boundary condition (\ref{SharpInterfaceSys:Neumann})
ensures that
\begin{align*}
 \partial_\nu\Phi_+(x,t)=\partial_\nu\Phi_-(x,t)\quad \text{ for }(x,t)\in\Gamma,
\end{align*}
which implies that
\begin{align*}
  \|\Phi_+(x,t)-\Phi_-(x,t)\|\le C d_0^2(x,t) \quad \text{ for }(x,t)\in\Gamma(\delta).
\end{align*}
This quadratical vanishing property near the interface is very important.

\begin{figure}
	\centering
	\begin{tikzpicture}
\node  at (1.5,1.5) {\small ${Q}_+$};

\node  at (-1.5,1.5) {\small ${Q}_-$};

	\draw [dotted] (0,0) to  (0,3);
{	\draw  (0,0) to [in=150,out=0,looseness=0.75] (3, -0.6);}
	\draw [thick](0,0) to [in=30,out=180,looseness=0.75] (-3, -0.6);
	\draw[dashed] (0,0) to [in=-150,out=0,looseness=0.75] (3, 0.6);
	\draw[dashed] (0,0) to [in=-30,out=180,looseness=0.75] (-3, 0.6);

	\draw [thick](0,3) to [in=-150,out=0,looseness=0.75] (3, 3+0.6);
	\draw (0,3) to [in=-30,out=180,looseness=0.75] (-3, 3+0.6);

	\node  at (-1.5,-0.6) {\small $\A_-(x,t)$};
	\node  at (1.5,3-0.2) {\small $\A_+(x,t)$};
	\node  at (-1.5,0.6) {\small $\A_+(x,t)(\II-2\nn\nn)$};

	\draw (-2.2, -0.29) to [in=-100,out=100,looseness=0.75] (-2.2, 0.29);
	\node  at (-3.3,0) {\small $\Phi(z,x,t)$};
\draw[->](-2.65,0)--(-2.25,0);

    \end{tikzpicture}
	\caption{Illustration of $\Phi(z,x,t)$ connecting $\A_-(x,t)$ and $\A_+(x,t)(\II-2\nn(x,t)\nn(x,t))$ }
	\label{fig:orbit}
\end{figure}

Let $\overline{\Phi}(x,t; \tau)(0\le \tau\le 1)$ be a geodesic on ${O}^-(n)$ connecting $\Phi_-(x,t)$ and $\Phi_+(x,t)$:
\begin{align*}
  \overline{\Phi}(x,t;0)=\Phi_-(x,t),\quad  \overline{\Phi}(x,t;1)=\Phi_+(x,t),\quad  \|\partial_\tau\overline{\Phi}(x,t;\tau)\| = \text{const. for }0\le\tau\le 1.
\end{align*}
Then we reparameterize the geodesic as
\begin{align}\label{def:Phi}
   \Phi(z, x,t)= \overline{\Phi}(x,t;\bar\eta(z)),
\end{align}
where $\bar\eta(z)$ is a monotonic increasing function which tends to $0$ (or $1$) exponentially
fast as $z\to -\infty$ (or $+\infty$). In particular, we can choose $\bar\eta(z)=s(z)$.
Apparently, one has that for $k\ge0$
\begin{align}\label{decay:Phi}
\|\partial_z^k(\Phi(z, x,t)-\Phi_\pm(x,t))\|& =O(e^{-\alpha_0|z|}d_0^2(x,t)), \text{ as }z\to \pm\infty,\,\, d_0\to 0.
\end{align}

We define
\begin{align}\label{def:quasi-mini}
\BTe(z,x,t)= \Phi(z,x,t)\PP_0(z,x,t)\quad\text{with }\quad \PP_0= (\II-2s(z)\nn(x,t)\nn(x,t)).
\end{align}
Thus, one has
\begin{align*}
f(\BTe)=\BTe\BTe^{\mathrm{T}}\BTe-\BTe=4s(s-1)(1-2s)\Phi(z,x,t)\nn\nn,
\end{align*}
which gives
 \begin{align*}
  \partial_z^2\BTe-f(\BTe)=\partial_z^2\Phi(z,x,t)\PP_0+2\partial_z\Phi(z,x,t)\partial_z\PP_0.
\end{align*}
Moreover, the linearized operator  around $\BTe$ can be written as
\begin{align*}
  \mathcal{L}_{\BTe}\A=\Phi(z,x,t)\mathcal{L}_{\PP_0}\PP+\partial_z^2\Phi(z,x,t)\PP+2\partial_z\Phi(z,x,t)\partial_z\PP
\end{align*}
for $\A=\Phi\PP$.

Due to the construction of $\Phi$,  $\partial_z^2\Phi(z,x,t), \partial_z\Phi(z,x,t)$ are of order $O(e^{-\alpha_0|z|}d_0^2(x,t))$.
Therefore, $\BTe$  satisfies \eqref{eq:leading} up to some terms which decay exponentially fast in $z$-variable and quadratically vanish on $\Gamma$. More importantly, as $\Phi\in{O}^-(n)$ is invertible and $\mathcal{L}_{\PP_0}$ is solvable (see Section \ref{subsec:pre:diagonal} below),  $\mathcal{L}_\BTe$ is also solvable up to some small remainders. These crucial properties  enable us to modify the original equation and solve the expanding systems exactly.\smallskip

The profile $\BTe$ defined in \eqref{def:quasi-mini} is called a {\it quasi-minimal connecting orbit}, and we will use it as the leading order approximation in the inner region.

\medskip

\section{Diagonalization of the linearized operator}\label{subsec:pre:diagonal}

To solve the linearized ODE system
\begin{align}\nonumber
 \mathcal{L}_{\PP_0}\PP:= -\partial_z^2\PP+\PP_0\PP_0^{\mathrm{T}}\PP+\PP_0\PP^{\mathrm{T}}\PP_0+\PP\PP_0^{\mathrm{T}}\PP_0-\PP=\FF.
\end{align}
with $\PP_0(z)=\II-2s(z)\nn\nn$ and $\nn\in S^2$,
we need to make a diagonalization to $\mathcal{L}_{\PP_0}$.
Here and in what follows, we simply write  $ \mathcal{L}_{\PP_0}$ as $\mathcal{L}$ when no ambiguity will be caused.

\subsection{An orthogonal  decomposition  of $\mathbb{M}_n$}

We introduce
\begin{equation}\label{def:subspace}
  \begin{split}
&\BV_1=\big\{\lambda \nn\nn\big|~ \lambda\in\mathbb{R}\big\},\\
& \BV_2=\big\{(\nn\bl+\bl\nn)\big|~\bl\cdot\nn=0\big\},\quad
\BV_3=\big\{(\nn\bl-\bl\nn)\big|~\bl\cdot\nn=0\big\},
\\&\BV_4=\mathrm{span}~\big\{(\bl\mm-\mm\bl)\big|~\bl\cdot\mm=\bl\cdot\nn=\mm\cdot\nn=0\big\},
\\&\BV_5=\mathrm{span}~\big\{\bl\bl,(\bl\mm+\mm\bl)\big|~\bl\cdot\mm=\bl\cdot\nn=\mm\cdot\nn=0\big\}.
 \end{split}
\end{equation}
Clearly, $\BM_n=\oplus_{i=1}^5\BV_i$, $\BA_n=\BV_3\oplus\BV_4$, $\BS_n=\BV_1\oplus\BV_2\oplus \BV_5$  and
$$\dim (\BV_1,~\BV_2,~\BV_3,~\BV_4,~\BV_5)=\Big(1,~ n-1,~ n-1, ~\frac12(n-1)(n-2),~ \frac12n(n-1)\Big).$$
Moreover, if $(\A_-,\A_+)$ is a minimal pair with $\A_-=\A_+(\II-2\nn\nn)$, then
\begin{align*}
\A_+\BV_i=\A_-\BV_i(i=1,4,5),\quad \A_+\BV_2=\A_-\BV_3,\quad \A_+\BV_3=\A_-\BV_2.
\end{align*}

Let  $\mathcal{P}_i: \mathbb{M}_n\to \BV_i$ ($1\le i\le 5$) be the projection operators. Then one has
\begin{align}\label{def:projection}
  \begin{split}
&\mathcal{P}_1\A=\nn\nn\A\nn\nn=\nn\nn(\A:\nn\nn), \\
&\mathcal{P}_2\A=\frac12\big[\nn\nn(\A+\A^\mathrm{T})(\II-\nn\nn)+(\II-\nn\nn)(\A+\A^\mathrm{T})\nn\nn\big],\\ &\mathcal{P}_3\A=\frac12\big[\nn\nn(\A-\A^\mathrm{T})(\II-\nn\nn)+(\II-\nn\nn)(\A-\A^\mathrm{T})\nn\nn\big],\\
&\mathcal{P}_4\A=\frac12(\II-\nn\nn)\big(\A-\A^\mathrm{T}\big)(\II-\nn\nn),\\
&\mathcal{P}_5\A=\frac12(\II-\nn\nn)\big(\A+\A^\mathrm{T}\big)(\II-\nn\nn).
\end{split}
\end{align}
One can directly check from \eqref{def:projection} that
\begin{align}\label{Append:comm-P2}
\begin{split}
(\II-2\nn\nn)  \mathcal{P}_2\A=&-\mathcal{P}_3\big((\II-2\nn\nn)\A\big),\\
(\II-2\nn\nn)  \mathcal{P}_3\A=&-\mathcal{P}_2\big((\II-2\nn\nn)\A\big),\\
 \mathcal{P}_4((\II-2\nn\nn)\A)=&~ (\II-2\nn\nn)  \mathcal{P}_4\A=\mathcal{P}_4\A.
 \end{split}
\end{align}
These projection operators play important roles throughout this paper. We remark that for our later use, $\nn$  and the corresponding decompositions (\ref{def:subspace}) may depend on $(x,t)$.

\subsection{Diagonalization}
Now we solve the ODE system
\begin{align}\label{eq:ODEsys}
 \mathcal{L}\PP(z)=\mathbf{F}(z).
\end{align}
{A crucial observation} is that the system \eqref{eq:ODEsys} can be  diagonalized into several scalar ODEs via the above orthogonal decomposition of  $\mathbb{M}_n$.

We denote
\begin{align*}
\UU_i(z) =\mathcal{P}_i\PP(z),\quad \VV_i(z) =\mathcal{P}_i\FF(z)\quad\text{ for } 1\le i\le 5.
\end{align*}
From the fact
\begin{align*}
  \PP_0\PP_0^{\mathrm{T}}=  \PP_0^{\mathrm{T}}\PP_0=\II-4s(1-s)\nn\nn
\end{align*}
and \eqref{def:projection}, we deduce that
\begin{align*}
  \PP_0\PP_0^{\mathrm{T}}\PP+\PP\PP_0^{\mathrm{T}}\PP_0&=2\PP-4s(1-s)(\nn\nn\PP+\PP\nn\nn)\\
 & =2\PP-8s(1-s)\nn\nn(\PP:\nn\nn)-4s(1-s)((\II-\nn\nn)\PP\nn\nn+\nn\nn\PP(\II-\nn\nn))\\
 &= 2\PP-8s(1-s)\UU_1-4s(1-s)(\UU_2+\UU_3),\\
  \PP_0\PP^{\mathrm{T}}\PP_0&=(\II-2s\nn\nn)\PP^{\mathrm{T}}(\II-2s\nn\nn)\\
 &= \PP^{\mathrm{T}}+4s^2\UU_1-4s\UU_1-2s(\UU_2-\UU_3).
\end{align*}
Using the fact that $\UU_i^{\mathrm{T}}=\UU_i(i=1,2,5)$ and $\UU_j^{\mathrm{T}}=-\UU_j(j=3,4)$,  we find
\begin{align*}
\mathcal{P}_i \mathcal{L}\PP=\mathcal{L}_i   \mathcal{P}_i \PP\quad\text{ for } 1\le i\le 5,
\end{align*}
where
\begin{align}\label{pre:decomp-Op-L}
\begin{split}
\mathcal{L}_iu =&-\partial_z^2u+\kappa_i(s)u,\\
 \text{ with }&\kappa_1(s)=2(1-6s+6s^2),\quad \kappa_2(s)=2(1-s)(1-2s),\\
\qquad & \kappa_3(s)=2s(2s-1),\quad \kappa_4(s)=0,\quad \kappa_5(s)=2.
\end{split}
\end{align}
Thus the system (\ref{eq:ODEsys}) can be reduced to
\begin{align}\label{eq:ODEsy-decompose}
 \mathcal{L}_i\UU_i=\VV_i \quad\text{ for } 1\le i\le 5.
\end{align}

From \eqref{eq:s}, it is easy to see that
\begin{align*}
  2(1-6s+6s^2)=\frac{s'''}{s'},\quad 2-6s+4s^2=\frac{s''}{s},\qquad 4s^2-2s=\frac{(1-s)''}{1-s}.
\end{align*}
Thus, we obtain
\begin{align*}
 &\mathcal{L}_1u=-\frac{1}{s'}\partial_z \Big((s')^2\partial_z\Big(\frac{u}{s'}\Big)\Big),\\
 &\mathcal{L}_2u=-\frac{1}{s}\partial_z \Big(s^2\partial_z\Big(\frac{u}{s}\Big)\Big),\\
  &\mathcal{L}_3u=-\frac{1}{1-s}\partial_z \Big((1-s)^2\partial_z\Big(\frac{u}{1-s}\Big)\Big).
\end{align*}

We define
\begin{align*}
\mathrm{Null}~\mathcal{L}=\mathrm{span}\big\{  s'\EE_1,~s\EE_2,~(1-s)\EE_3,~ \EE_4: \text{for }\forall \text{ constant } \EE_i\in\mathbb{V}_i \big\}.
\end{align*}
Then we have
\begin{Lemma}
$\mathcal{L}\Bp=0$ for $\forall~\Bp\in\mathrm{Null}~\mathcal{L}$.
\end{Lemma}

The following lemma on the solvability of $\mathcal{L}$ will be repeatedly used in the inner expansion.
\begin{Lemma}\label{lem:ODEsys}
Assume that $\FF_\pm=\lim_{z\to\pm\infty}\FF(z)$ is bounded and $\FF$ satisfies the following boundary conditions with $\alpha_0>0$:
\begin{align*}
&\text{(B1)}: ~\FF-\FF_\pm=O(e^{-\alpha_0|z|}),\qquad \text{as }z\to \pm\infty;\\
&\text{(B2)}: ~\mathcal{P}_2\FF_+=0;\quad\text{(B3)}: \mathcal{P}_3\FF_-=0;\quad \text{(B4)}: \mathcal{P}_4\FF_\pm=0,
\end{align*}
and orthogonal conditions:
\begin{align*}
&\text{(O1)}: \int_\mathbb{R} s'\FF:\EE_1\ud z =0,~\forall~ \EE_1\in \mathbb{V}_1;\quad
\text{(O2)}: \int_\mathbb{R} s\FF:\EE_2\ud z =0,~\forall~ \EE_2\in \mathbb{V}_2;\\
&\text{(O3)}:  \int_\mathbb{R} (1-s)\FF:\EE_3\ud z=0,~\forall ~\EE_3 \in \mathbb{V}_3;\quad
\text{(O4)}:   \int_\mathbb{R} \FF:\EE_4\ud z =0,~\forall~ \EE_4\in\mathbb{V}_4.
\end{align*}
 Then there exists a unique $\overline{\QQ}\in\mathbb{V}_4$,  such that \eqref{eq:ODEsys} has a unique solution $\PP^*(z)$ satisfying
\begin{align*}
&\text{(B1)}: ~\PP^*-\PP_\pm=O(e^{-\alpha_0 |z|}),\qquad \text{as }z\to \pm\infty,\\
&\text{(B2)}: ~\mathcal{P}_2\PP_+=0,\quad\text{(B3)}: \mathcal{P}_3\PP_-=0,\quad \text{(B4)}: \mathcal{P}_4\PP_-=\mathcal{P}_4\PP_+-\overline{\QQ}=0.
\end{align*}
\end{Lemma}
\begin{Remark}
  The conditions (B1)-(B4) ensure that the integrals in (O1)-(O4) are finite.
\end{Remark}
\begin{proof}
The results can be deduced from Lemmas \ref{lem:s1}-\ref{lem:s5}.
\end{proof}
\begin{Remark}
As a corollary, we deduce that for any given
 \begin{align*}
 (\QQ_2,\QQ_3,\QQ_4)\in \mathbb{V}_2\times \mathbb{V}_3\times \mathbb{V}_4,
 \end{align*}
the ODE system \eqref{eq:ODEsys} has a unique solution  $\PP$ with properties:
\begin{align*}
&\text{(B1)}: ~\PP-\PP_\pm=O(e^{-\alpha_0 |z|}),\qquad \text{as }z\to \pm\infty,\\
&\text{(B2)}: ~\mathcal{P}_2\PP_+=\QQ_2,\quad\text{(B3)}: \mathcal{P}_3\PP_-=\QQ_3,\quad \text{(B4)}: \mathcal{P}_4\PP_-=\mathcal{P}_4\PP_+-\overline{\QQ}=\QQ_4.
\end{align*}
\end{Remark}

\subsection{Cubic-null cancellation}{}

For $\A_1, \A_2, \A_3\in\mathbb{M}_n$, we define the trilinear form
\begin{align}\label{def:BB}
\TT_f(\A_1,\A_2,\A_3)=(\A_1\A_2^{\mathrm{T}}+\A_2\A_1^{\mathrm{T}})\A_3
+\A_3(\A_1^{\mathrm{T}}\A_2+\A_2^{\mathrm{T}}\A_1)+(\A_1\A_3^{\mathrm{T}}\A_2+\A_2\A_3^{\mathrm{T}}\A_1).
\end{align}
Then $\TT_f(\A_1,\A_2,\A_3):\A_4$ keeps the same if we exchange any $\A_i$ and $\A_j$ ($1\le i, j\le 4$). The following cancellation relation plays an important role in closing the expansion system of each order.

\begin{Lemma}\label{lem:inner-null}
For $\QQ_1, \QQ_2, \QQ_3\in \mathrm{Null}~\mathcal{L}$, we have
\begin{align*}
\int_\BR\TT_f(\PP_0,\QQ_1,\QQ_2):\QQ_3\ud z=0,
\end{align*}
where the integral is understood as $\lim_{R\to +\infty}\int_{-R}^R(\cdot)\ud z$ if necessary.
\end{Lemma}

\begin{proof}
 For the convenience, the left hand side  is denoted by $\mathfrak{I}(\PP_0,\QQ_1,\QQ_2,\QQ_3)$. It suffices to prove the case of $\QQ_2=\QQ_3=\BB$, since we have by \eqref{eq:ABC}  that
$$2\mathfrak{I}(\PP_0,\QQ_1,\QQ_2,\QQ_3)=\mathfrak{I}(\PP_0,\QQ_1,\QQ_2+\QQ_3,\QQ_2+\QQ_3)-\mathfrak{I}(\PP_0,\QQ_1,\QQ_2,\QQ_2)-\mathfrak{I}(\PP_0,\QQ_1,\QQ_3,\QQ_3).$$

Assuming $\BB(z)=\sum_{i=1}^4\BB_i(z)$ with $\BB_i(z)\in\mathbb{V}_i$.  When $\QQ_1=\lambda s'\nn\nn$, we have
\begin{align*}
\mathfrak{I}(\PP_0,\QQ_1,\BB,\BB)&=\lambda\int_\BR\Big(2s'(1-2s)(|\BB\nn|^2+|\nn\BB|^2)+2s'|\BB^2:\nn\nn|^2-4ss'|\BB:\nn\nn|^2 \Big)\ud z.
\end{align*}
It is easy to check that $(\BB_i\BB_j):\nn\nn=0$ for different $i,j$, and $\BB_4\nn=\nn\BB_4=0$. Thus,
\begin{align*}
\mathfrak{I}(\PP_0,\QQ_1,\BB,\BB)
&=\lambda\sum_{i=1}^3\int_\BR\Big(2s'(1-2s)\big[|\BB_i\nn|^2+|\nn\BB_i|^2\big]+2s'|\BB_i^2:\nn\nn|^2-4ss'|\BB_i:\nn\nn|^2 \Big)\ud z\\
&=\lambda\sum_{i=1}^3\mathfrak{I}(\PP_0,\QQ_1,\BB_i,\BB_i).
\end{align*}
For $\BB_2=s(\nn\bl+\bl\nn)$ with $\bl\bot\nn$, we have
\begin{align*}
\mathfrak{I}(\PP_0,\QQ_1,\BB_2,\BB_2)=\lambda\int_\BR\Big(4s'(1-2s)s^2+2s's^2\Big)|\bl|^2\ud z=\lambda\int_\BR (s^3(1-s))'|\bl|^2\ud z=0.
\end{align*}
Similarly, we can prove that $\mathfrak{I}(\PP_0,\QQ_1,\BB_3,\BB_3)=0$ for $\BB_3=(1-s)(\nn\bl-\bl\nn)$ with $\bl\bot\nn$.
For $\BB_1=\mu s'\nn\nn$, we have
\begin{align*}
\mathfrak{I}(\PP_0,\QQ_1,\BB_1,\BB_1)=\lambda\mu^2\int_\BR 6(s')^3(1-2s)\ud z=0.
\end{align*}
Therefore, $\mathfrak{I}(\PP_0,\QQ_1,\BB,\BB)=0$ for $\QQ_1=\lambda s'\nn\nn$ and $\BB\in\mathrm{Null}~\mathcal{L}$.

For $\QQ_1=s\EE_2+(1-s)\EE_3+\EE_4$ with $\EE_i\in\mathbb{V}_i$, by direct calculations (see Lemma \ref{lem:bilinear}), we get
\begin{align}\nonumber
\mathfrak{I}(\PP_0,\QQ_1,\BB,\BB) = \int_{\BR}\Big\{&~2s\EE_2:\big[(2s-1)(\BB_3\BB_4+\BB_4\BB_3)+(3-4s)(\BB_1\BB_2+\BB_2\BB_1)\big]\\ \nonumber
&+2(1-s)\EE_3:\big[(1-4s)(\BB_1\BB_3+\BB_3\BB_1)+(1-2s)(\BB_2\BB_4+\BB_4\BB_2)\big]\\
&+2(1-2s)\EE_4:(\BB_2\BB_3+\BB_3\BB_2) \Big\}\ud z.\nonumber
\end{align}
By taking $\BB_1=\lambda_1 s'\EE_1, \BB_2=\lambda_2 s \EE_2, \BB_3=\lambda_3 (1-s)\EE_3, \BB_4=\lambda_4\EE_4$, we can show that all the integrals vanish due to the following fact:
\begin{align*}
\int_\BR s(2s-1)(1-s)\ud z=
\int_\BR 2s^2(3-4s)s'\ud z=
\int_\BR 2(1-s)^2(1-4s)s'\ud z=0.
\end{align*}

The proof is completed.
\end{proof}

\medskip

 \section{Inner expansion}
\label{sec:inner-expan}

\subsection{Formal inner expansion}

Formally, we write the inner expansion as
\begin{align*}
\A_{I}^{\ve}(z,x,t)=\Phi(z,x,t)\PP^\ve(z,x,t)=\Phi(z,x,t)\Big(\PP_0(z,x,t)+\ve \PP_1(z,x,t)+\ve^2 \PP_2(z,x,t)+\cdots \Big),
\end{align*}
where $\PP_0(z,x,t)=\II-2s(z)\nn(x,t)\nn(x,t)$ and $\Phi$ is given by \eqref{def:Phi}. We should keep in mind that $\PP_k$ has to satisfy the matching conditions for $(x,t)\in\Gamma(\delta)$:
\begin{align}\label{bc-Pk}
\Phi_\pm\PP_k(\pm\infty, x,t)=\A_\pm^{(k)},\quad \text{or equivalently}\quad \PP_k(\pm\infty, x,t)=\PP_0^\pm\UU_\pm^{(k)}.
\end{align}

Let $\widetilde\PP^\ve(x,t) =\PP^\ve(d^\ve/\ve,x,t)$ and $\widetilde\Phi(x,t) =\Phi(d^\ve/\ve,x,t)$. Then we have
\begin{align*}
\partial_t\widetilde\A_{I}^{\ve}-\Delta\widetilde\A_{I}^{\ve}+\ve^{-2}f(\widetilde\A_{I}^{\ve})
=\widetilde\Phi(\partial_t\widetilde\PP^\ve-\Delta \widetilde\PP^\ve+\ve^{-2}f(\widetilde\PP^{\ve}))+(\partial_t-\Delta)\widetilde\Phi\widetilde\PP^\ve-2\nabla\widetilde\Phi\nabla\widetilde\PP^\ve.
\end{align*}
A direct calculation shows that
\begin{multline}\label{eq:fullP}
\Big\{\ve^{-2}\big[-\partial_z^2\PP^\ve+f(\PP^{\ve})-\Phi^{\mathrm{T}}\partial_z^2 \Phi\PP^\ve-2\Phi^{\mathrm{T}}\partial_z\Phi\partial_z\PP^\ve\big]
\\
+\ve^{-1}\big[(\partial_td^\ve-\Delta d^\ve)(\partial_z\PP^\ve+\Phi^{\mathrm{T}}\partial_z\Phi)-2\Phi^{\mathrm{T}}\nabla d^\ve\nabla\partial_z(\Phi\PP^\ve)\big]\\
+\Phi^{\mathrm{T}}(\partial_t-\Delta)(\Phi\PP^\ve)\Big\}\Big|_{z=d^{\ve}/\ve}=0.
\end{multline}
As in \cite{ABC}, we will regard $z$ as an independent variable and $(x,t)\in\Gamma(\delta)$ as parameters. Then we will solve a series of ODE systems
with respect to $z$ for $\PP_k(z,x,t)$ $(k=0, 1,2,\cdots)$.
Note that we can add any terms vanishing on $\{d^\ve =\ve z\}$  on the left hand side, which does not change the equation (\ref{eq:fullP}).

Let $\widetilde\GG^\ve(z,x,t)$ and $\HH^{\pm,\ve}(x,t)$ be matrix-valued functions to be determined later. We choose a fixed smooth and nonnegative function
$\eta(z)$ satisfying: $\eta(z)=0$ if $z\le -1$, $\eta(z)=1$ if $z\ge 1$, $\eta'(z)\ge 0$. Let
\begin{align}\label{def:constant-a}
a_0=\int_\BR (s')^2\ud z,\qquad (a_1, a_2, a_3)=\int_{\BR}\eta'(s', s, 1-s)\ud z.
\end{align}
Consider the following modified system
\begin{multline}\label{eq:modified}
\ve^{-2}\Big\{-\partial_z^2\PP^\ve+f(\PP^{\ve})-\Phi^{\mathrm{T}}\partial_z^2 \Phi\PP^\ve-2\Phi^{\mathrm{T}}\partial_z\Phi\partial_z\PP^\ve+(d^\ve-\ve z)\widetilde\GG^\ve\Big\}
\\
+\ve^{-1}\Big\{(\partial_td^\ve-\Delta d^\ve)(\partial_z\PP^\ve+\Phi^{\mathrm{T}}\partial_z\Phi)-2\Phi^{\mathrm{T}}\big(\nabla d^\ve\nabla\partial_z(\Phi\PP^\ve))\Big\}
\\+\Phi^{\mathrm{T}}(\partial_t-\Delta)(\Phi\PP^\ve)+\HH^{+,\ve}(x,t)\eta_M^+(z)+\HH^{-,\ve}(x,t)\eta_M^-(z)=0
\end{multline}
for any $(z,x,t)\in \mathbb{R}\times\Gamma(\delta)$.  Following the idea in \cite{ABC}, we choose $\eta^\pm_M(z)=\eta(-M\pm z)$ with
$$M=|d_1|_{C^0(\Gamma_t)}+2.$$ Then one has $\eta(-M\pm d^\ve(x,t)/\ve)=0$ for $(x,t)\in\Gamma(\delta)\cap {Q}_\mp$.
$\HH^{\pm,\ve}(x,t)$ will be chosen as in (\ref{def:Hve}) with \eqref{vanish:Hk+} and \eqref{vanish:Hk-}, which imply  $\HH^{\pm,\ve}(x,t)=0$ for $(x,t)\in \Gamma(\delta)\cap {Q}_\pm$. Thus,  one has
\begin{align*}
  \HH^{+,\ve}(x,t)\eta_M^+(z)+\HH^{-,\ve}(x,t)\eta_M^-(z)\Big|_{z=d^\ve(x,t)/\ve}=0,\quad \text{ for any }(x,t)\in\Gamma(\delta).
\end{align*}
As a result, all the modified terms vanish on $\{d^\ve =\ve z\}$, and will not change the system (\ref{eq:fullP}).

From the definition of $\Phi$, we can deduce that $\partial_z(\Phi^{\mathrm{T}}\partial_\nu\Phi)=0$ on $\Gamma$, which gives
\begin{align}\label{inner:proper-PhidPhi}
\Phi^{\mathrm{T}}\partial_\nu\Phi=\Phi_-^{\mathrm{T}}\partial_\nu\Phi_-\triangleq\overline\WW\in\BV_4.
\end{align}
Indeed, as $\Phi_+-\Phi_-=O(d_0^2)$, we have $\partial_z\Phi=O(d_0^2)$ and  $\partial_z(\Phi^{\mathrm{T}}\partial_\nu\Phi)=\Phi^{\mathrm{T}}\partial_z\partial_\nu\Phi=\partial_\nu(\Phi^{\mathrm{T}}\partial_z\Phi)=0$ on $\Gamma$.
See Appendix \ref{App:BC} for the proof of $\overline\WW\in\BV_4$.

Moreover, since it holds $ \Phi^{\mathrm{T}}\partial_z^2\Phi =0$ and $\Phi^{\mathrm{T}}\partial_z\Phi =0$ on $\Gamma$,
we can assume that
\begin{align}\label{inner:def-Phi12}
  \Phi^{\mathrm{T}}\partial_z^2\Phi =d_0(x,t)\Phi_1(z,x,t),\quad 2\Phi^{\mathrm{T}}\partial_z\Phi=d_0(x,t)\Phi_2(z,x,t),
\end{align}
where on $\Gamma$ the definitions are interpreted as
\begin{align*}
  \Phi_1(z,x,t)=\lim_{d_0\to 0}d_0^{-1}\Phi^{\mathrm{T}}\partial_z^2\Phi=\partial_\nu(\Phi^{\mathrm{T}}\partial_z^2\Phi),\quad
  \Phi_2(z,x,t)=\lim_{d_0\to 0}2d_0^{-1}\Phi^{\mathrm{T}}\partial_z\Phi=2\partial_\nu(\Phi^{\mathrm{T}}\partial_z\Phi).
\end{align*}
It is direct to check that
\begin{align}\label{decay:Phi12}
  \Phi_1(z,x,t),\,\,\Phi_2(z,x,t)=O(e^{-\alpha_0|z|})d_0(x,t).
\end{align}
Before going to next steps,  let us recall from \eqref{decay:Phi} and \eqref{decay:Phi12} that
\begin{align}\label{property:vanish}
  \partial_z\Phi=O(d_0^2), \quad\,\, \Phi_1,~\Phi_2, ~ \nabla d_0\cdot\nabla\partial_z\Phi,~\nabla d_0\cdot\nabla\nn=O(d_0)\quad \text{ for }(x,t)\in\Gamma(\delta),
\end{align}
which will be repeatedly used in the sequel.

Now we take
\begin{align}
  \widetilde{\GG}^\ve(z,x,t) &= (\Phi_1\PP^\ve+\Phi_2\partial_z\PP^\ve)+ \GG^\ve(z,x,t),\\
& \text{with } \GG(z,x,t)=\sum_{k\ge 1}\ve^k\Big(\GG_{k}(x,t)\eta'(z)+\LL_k(x,t)\eta''(z)\Big),\\
 \HH^{\pm,\ve}(x,t)&=\sum_{k\ge 0}\ve^k\HH^\pm_{k}(x,t). \label{def:Hve}
\end{align}
At each $(x,t)\in\Gamma(\delta)$, we will choose
\begin{align*}
  \GG_{k}(x,t)\in\mathbb{V}_1\oplus\mathbb{V}_2\oplus\mathbb{V}_3\oplus\mathbb{V}_4,\quad  \LL_{k}(x,t)\in\mathbb{V}_4,\quad
   \HH_k^+\in\mathbb{V}_2\oplus\mathbb{V}_4,\quad
   \HH^{-}_k\in\mathbb{V}_3\oplus\mathbb{V}_4,
\end{align*}
which will be precisely defined later.
\begin{Remark}
The role of  $(\Phi_1\PP^\ve+\Phi_2\partial_z\PP^\ve)$ in $\widetilde\GG^\ve$ is to leave the small error terms into next orders.
Otherwise, the obtained ODE systems are too complicate to solve.
The term $\GG_k$ is used to ensure the orthogonal conditions (O1)-(O4) of $\FF_k$ on $\Gamma_t(\delta)\setminus\Gamma_t$,
while $\HH_k^\pm$ is used to ensure the boundary conditions
(B2)-(B4), and $\LL_k$ is introduced to characterize the variation of ``normal derivative" of $\mathcal{P}_4\PP_k$.
  \end{Remark}

Now we substitute the expansion into (\ref{eq:modified}) and collect the terms of the same order.\smallskip

\no$\bullet$ The $O(\ve^{-2})$ system takes the form
\begin{align}\label{expand:P0}
-\partial_z^2\PP_0+\PP_0\PP_0^{\mathrm{T}}\PP_0-\PP_0=0,
\end{align}
which is satisfied by
\begin{align}\label{solution:P0}
\PP_0(z,x,t)= \II-2s(z)\nn(x,t)\nn(x,t).
\end{align}

\no$\bullet$ The $O(\ve^{-1})$ system reads as
\begin{align}\nonumber
&-\partial_z^2\PP_1+\PP_0\PP_0^{\mathrm{T}}\PP_1+\PP_0\PP_1^{\mathrm{T}}\PP_0+\PP_1\PP_0^{\mathrm{T}}\PP_0-\PP_1\\ \nonumber
&= -(\partial_td_0-\Delta d_0)(\Phi^{\mathrm{T}}\partial_z\Phi\PP_0+\partial_z\PP_0)
+2\Phi^{\mathrm{T}}\nabla d_0\cdot\nabla_x\partial_z(\Phi\PP_0)\\ \nonumber
&\quad\,-(\Phi_1\PP_0+\Phi_2\partial_z\PP_0)(d_1-z)+\GG_1\eta' d_0+\LL_1\eta'' d_0\\\label{expand:P1}
&\triangleq  \FF_1+\GG_1\eta' d_0+\LL_1\eta'' d_0.
\end{align}

\no$\bullet$ The $O(1)$ system can be written as
\begin{align}\nonumber
&-\partial_z^2\PP_2+\PP_0\PP_0^{\mathrm{T}}\PP_2+\PP_0\PP_2^{\mathrm{T}}\PP_0+\PP_2\PP_0^{\mathrm{T}}\PP_0-\PP_2\\\nonumber
=&-(\partial_td_0-\Delta d_0)(\Phi^{\mathrm{T}}\partial_z\Phi\PP_1+\partial_z\PP_1)
+2\Phi^{\mathrm{T}}\nabla d_0\cdot\nabla_x\partial_z(\Phi\PP_1)\\\nonumber
&-(\partial_td_1-\Delta d_1)(\Phi^{\mathrm{T}}\partial_z\Phi\PP_0+\partial_z\PP_0)+2\Phi^{\mathrm{T}}\nabla d_1\cdot\nabla_x\partial_z(\Phi\PP_0)\\\nonumber
&-\Phi^{\mathrm{T}}(\partial_t-\Delta)(\Phi\PP_{0})-(\PP_0\PP_1^{\mathrm{T}}\PP_1
+\PP_1\PP_0^{\mathrm{T}}\PP_1+\PP_1\PP_1^{\mathrm{T}}\PP_0)+\HH_0^+\eta_M^++\HH_0^-\eta_M^-\\  \nonumber
&-(\Phi_1\PP_1+\Phi_2\partial_z\PP_1)(d_1-z)-(\Phi_1\PP_0+\Phi_2\partial_z\PP_0)d_2\\ \nonumber
&+(\GG_1\eta'+\LL_1\eta'')(d_1-z)+(\GG_2\eta'+\LL_2\eta'')d_0\\\label{expand:P2}
\triangleq &\FF_2+(\GG_2\eta'+\LL_2\eta'')d_0.
\end{align}

\no$\bullet$ The $O(\ve^{k-2})$ ($k\ge3$) system takes the form
\begin{align}\nonumber
&-\partial_z^2\PP_k+\PP_0\PP_0^{\mathrm{T}}\PP_k+\PP_0\PP_k^{\mathrm{T}}\PP_0+\PP_k\PP_0^{\mathrm{T}}\PP_0-\PP_k\\ \nonumber
=& -(\partial_td_0-\Delta d_0)(\Phi^{\mathrm{T}}\partial_z\Phi\PP_{k-1}+\partial_z\PP_{k-1})
+2\Phi^{\mathrm{T}}\nabla d_0\cdot\nabla_x\partial_z(\Phi\PP_{k-1})\\ \nonumber
&-(\partial_td_{k-1}-\Delta d_{k-1})(\Phi^{\mathrm{T}}\partial_z\Phi\PP_{0}+\partial_z\PP_{0})
+2\Phi^{\mathrm{T}}\nabla d_{k-1}\cdot\nabla_x\partial_z(\Phi\PP_0)\\ \nonumber
&-\sum_{\substack{i+j=k-1\\ 1\le i,j\le k-2}}(\partial_td_i-\Delta d_i)(\Phi^{\mathrm{T}}\partial_z\Phi\PP_j+\partial_z\PP_j)+2\Phi^{\mathrm{T}}\sum_{\substack{i+j=k-1\\ 1\le i,j\le k-2}}
\nabla d_i\cdot\nabla_x\partial_z(\Phi\PP_j)\\ \nonumber
&-\Phi^{\mathrm{T}}(\partial_t-\Delta)(\Phi\PP_{k-2})-\sum_{\substack{i+j+l=k\\ 0\le i,j,l\le k-1}}\PP_i\PP_j^{\mathrm{T}}\PP_l+\HH_{k-2}^+\eta_M^++\HH_{k-2}^-\eta_M^-\\ \nonumber
&-\sum_{\substack{i+j=k\\ 1\le i,j\le k-1}}(\Phi_1\PP_i+\Phi_2\partial_z\PP_i)(d_j-\delta_j^1z)+(\Phi_1\PP_0+\Phi_2\partial_z\PP_0)d_k\\ \nonumber
&+(\GG_1\eta'+\LL_1\eta'')d_{k-1}
+\sum_{\substack{i+j=k\\ 1\le i,j\le k-1}}(\GG_i\eta'+\LL_i\eta'')(d_j-\delta_j^1z)+(\GG_{k}\eta'+\LL_{k}\eta'')d_0\\\label{expand:Pk}
\triangleq &~\FF_k+(\GG_{k}\eta'+\LL_{k}\eta'')d_0.
\end{align}

\subsection{Determination of  $\HH_k^\pm$ and $\GG_k$}
The equation \eqref{expand:P1}-(\ref{expand:Pk}) can be written as
\begin{align}
\mathcal{L}( \PP_k+\LL_{k}d_0\eta)=\FF_k+\GG_k\eta'd_0.
\end{align}
Now we determine $\HH_k^\pm$ and $\GG_k$ in order to ensure that the system for the inner expansion is solvable.
Let $\PP_{k}^\pm(x,t)=\PP_k(\pm\infty,x,t)$. Firstly, we choose
\begin{align}\label{solution:Hk+}
   \HH_{k-2}^+=-(\mathcal{P}_2+\mathcal{P}_4)\Big(\Phi_+^{\mathrm{T}}(\partial_t-\Delta)(\Phi_+\PP_{k-2}^+)
   +\sum_{\substack{i+j+l=k\\ 0\le i,j,l\le k-1}}\PP_i^+(\PP_j^{+})^{\mathrm{T}}\PP_l^+\Big),\\
   \HH_{k-2}^-=-(\mathcal{P}_3+\mathcal{P}_4)\Big(\Phi_-^{\mathrm{T}}(\partial_t-\Delta)(\Phi_-\PP_{k-2}^-)
   +\sum_{\substack{i+j+l=k\\ 0\le i,j,l\le k-1}}\PP_i^-(\PP_j^{-})^{\mathrm{T}}\PP_l^-\Big),\label{solution:Hk-}
\end{align}
for all $k\ge2$.
Thus, we have
$$\mathcal{P}_2\lim_{z\to +\infty}\FF_k=\mathcal{P}_3\lim_{z\to -\infty}\FF_k=\mathcal{P}_4\lim_{z\to \pm\infty}\FF_k =0,$$
which are necessary for solvability of $\PP_k$; see (B2)-(B4) in Lemma \ref{lem:ODEsys}.

From \eqref{bc-Pk}, one has $\Phi_+\PP_{k}^+=\A_+^{(k)}=\A_+\UU_{+}^{(k)}$ and $\Phi_+^{\mathrm{T}}\A_+=\II-2\nn\nn$. Thus, we get
\begin{align}\nonumber
   \HH_{k-2}^+
=&-(\mathcal{P}_2+\mathcal{P}_4)\Big\{(\II-2\nn\nn)\Big(\A_+^{\mathrm{T}}(\partial_t-\Delta)(\A_+\UU_{+}^{(k-2)})
   +\sum_{\substack{i+j+l=k\\ 0\le i,j,l\le k-1}}\UU_{+}^{(i)}(\UU_{+}^{(j)})^{\mathrm{T}}\UU_{+}^{(l)}\Big)\Big\},\\ \nonumber
=&~(\II-2\nn\nn)(\mathcal{P}_3+\mathcal{P}_4)\Big(\mathcal{J}_+\UU_{+}^{(k-2)}+\BB_+^{(k-1)}+{\CC}_+^{(k-2)}\Big)\\\label{vanish:Hk+}
=&~0,\qquad \text{for } x\in{Q}_+\cap\Gamma_t(\delta).
 \end{align}
Similarly, one has
\begin{align}\label{vanish:Hk-}
   \HH_{k-2}^-=&~0,\qquad \text{for } x\in{Q}_-\cap\Gamma_t(\delta).
 \end{align}
 In particular, we have
\begin{align}\label{cancel:Hk-on-interface}
  \HH_{k-2}^\pm=0\text{ on }\Gamma_t,\quad \text{ for }k\ge 2.
\end{align}

From (\ref{expand:P1})-(\ref{expand:Pk}), we know that
\begin{align}\label{inner:def-Fk}
\FF_k=\widetilde{\FF}_k(\Phi, \{d_i, \PP_i, \GG_i,\LL_i:{i\le k-1}\})+(\Phi_1\PP_0+\Phi_2\partial_z\PP_0)d_k,
\end{align}
where $\widetilde{\FF}_k$ is an explicitly given function, and the last term vanishes on $\Gamma$.

Once we have determined $\FF_k$ on $\Gamma_t(\delta)$ by \eqref{inner:def-Fk} with $\FF_k|_{\Gamma_t}$ satisfying (O1)-(O4),
 $\GG_{k}$ can be uniquely defined as
\begin{eqnarray}
\GG_k=\left \{
\begin {array}{ll}\label{solution:Gk}\vspace{2mm}
-d_0^{-1}\int_\BR \big(a_1^{-1}s'\mathcal{P}_1+a_2^{-1}s\mathcal{P}_2+a_3^{-1}(1-s)\mathcal{P}_3+\mathcal{P}_4\big)\FF_k \ud z, \quad (x,t)\in\Gamma(\delta)\setminus\Gamma,\\
-\partial_\nu\int_\BR \big(a_1^{-1}s'\mathcal{P}_1+a_2^{-1}s\mathcal{P}_2+a_3^{-1}(1-s)\mathcal{P}_3+\mathcal{P}_4\big)\FF_k \ud z,\quad (x,t)\in\Gamma.
\end{array}
\right.
\end{eqnarray}
Here $a_i(i=1,2,3)$ are constants defined in \eqref{def:constant-a}.
One can directly check  that $\FF_k+\GG_k\eta'd_0$ satisfies (B1)-(B4) and (O1)-(O4) in $\Gamma_t(\delta)$.
Moreover, we have
\begin{Lemma}\label{lem:Gknn} $(\GG_k:\nn\nn)|_{\Gamma_t} $ is independent of $d_k$.
\end{Lemma}
\begin{proof} From \eqref{inner:def-Fk} and \eqref{solution:Gk}, it suffices to prove that
\begin{align}\label{cancel:Gknn}
\lim_{d_0\to 0}  d_0^{-1}(\Phi_1\PP_0+\Phi_2\partial_z\PP_0):\nn\nn=0.
\end{align}
We recall the definition \eqref{inner:def-Phi12} for $\Phi_1$ and $\Phi_2$. Obviously, $\Phi^{\mathrm{T}}\partial_z\Phi$ is antisymmetric in $\Gamma_t(\delta)$,
and then so do $\Phi_2$ and $\lim_{d_0\to0}d_0^{-1}\Phi_2$. Therefore,
\begin{align*}
  \Phi_2\partial_z\PP_0:\nn\nn= -2s'\Phi_2:\nn\nn=0\qquad \text{ in }\Gamma_t(\delta).
\end{align*}
Moreover, we have
\begin{align*}
\Phi^{\mathrm{T}}\partial_z^2\Phi=\partial_z(\Phi^{\mathrm{T}}\partial_z\Phi)-\partial_z\Phi^{\mathrm{T}}\partial_z\Phi.
\end{align*}
The first term is antisymmetric, and the second term is of order  $O(d_0^4)$. This proves \eqref{cancel:Gknn}.
\end{proof}

We can also obtain the explicit expression for $\GG_k:\nn\nn$ on $\Gamma$ from \eqref{solution:Gk}:
\begin{align}\label{inner:eq-Gknn}
\GG_k:\nn\nn =&- a_1^{-1}\partial_\nu\int_\BR s'\widetilde\FF_k:\nn\nn \ud z,
\end{align}
which only depends on $\Phi$ and $\{d_i, \PP_i, \GG_i,\LL_i:{i\le k-1}\}$. In particular, $(\GG_1:\nn\nn)|_{\Gamma_t}$ only relies on $\Phi, d_0,\PP_0$.
We will use these facts for the  derivation of  the equations of $d_1$ and $d_k(k\ge 2)$.

\medskip

\section{Solving the systems for outer/inner expansion}\label{sec:construct}

In this section, we first solve the systems for the outer/inner expansions, and then construct the approximate solutions by the gluing method.

Recall that $\PP_k$ for $k\ge 1$ solves
\begin{align}\label{inner:eq-Pk}
\mathcal{L}( \PP_k+\LL_{k}d_0\eta)=\FF_k+\GG_k\eta'd_0,
\end{align}
together with the boundary condition
\begin{align}\label{innerbc-Pk}
\PP_k(\pm\infty, x,t)=\PP_0^\pm\UU_\pm^{(k)}.
\end{align}
The solution $\PP_k$ of \eqref{inner:eq-Pk}  can be written as
\begin{align}\label{solution:Pk}
  \PP_k=s(z)\PP_{k,2}(x,t)+(1-s(z))\PP_{k,3}(x,t)+\PP_{k,4}(x,t)-\LL_{k}(x,t)d_0(x,t)\eta(z)+\PP_k^*(z,x,t),
\end{align}
where  $\PP_{k,i}(x,t)\in\mathbb{V}_i$ and  $\PP_k^*(z,x,t)$ is uniquely solved from
\begin{align}\label{solution:Pkstar}
  \begin{split}
&\mathcal{L}\PP_k^*(z,x,t)=\FF_k+\GG_k\eta'd_0,\\
&  \mathcal{P}_2\PP_k^*(+\infty, x,t)= \mathcal{P}_3\PP_k^*(-\infty, x,t)= \mathcal{P}_4\PP_k^*(-\infty, x,t)=0
 \end{split}
 \end{align}
with the help of Lemma \ref{lem:ODEsys}.

Due to the matching condition (\ref{innerbc-Pk}), one has for $(x,t)\in\Gamma(\delta)$ that
 \begin{align}\label{jump:Pk4-00}
\PP_{k,4}(x,t)&=\mathcal{P}_4 \PP_k(-\infty, x,t) =\mathcal{P}_4\Big( \PP_k(+\infty, x,t)-\PP_k^*(+\infty, x,t)\Big)+\LL_{k}d_0,
\end{align}
and
\begin{align}\label{jump:Pk2}
\PP_{k,2}(x,t)&=\mathcal{P}_2\PP_k(+\infty,x,t) =\mathcal{P}_2 (\II-2\nn\nn)\UU_+^{(k)}=-(\II-2\nn\nn)\mathcal{P}_3 \UU_+^{(k)}=-(\II-2\nn\nn)\mathcal{P}_3 \VV_+^{(k)}, \\
 \PP_{k,3}(x,t)&=\mathcal{P}_3\PP_k(-\infty,x,t) =\mathcal{P}_3 \UU_-^{(k)}=\mathcal{P}_3 \VV_-^{(k)}.\label{jump:Pk3}
 \end{align}
 Here we have used the relation $ \mathcal{P}_2\big( (\II-2\nn\nn)\A\big)=-(\II-2\nn\nn)\mathcal{P}_3 \A$  in \eqref{Append:comm-P2}.
Moreover, as
\begin{align*}
  \mathcal{P}_4 \PP_k(-\infty)=&~\mathcal{P}_4 \UU_-^{(k)}=\mathcal{P}_4\VV_-^{(k)},\\
\mathcal{P}_4 \PP_k(+\infty)=&~\mathcal{P}_4 \Big((\II-2\nn\nn)\UU_+^{(k)}\Big)=\mathcal{P}_4\UU_+^{(k)}=\mathcal{P}_4\VV_+^{(k)},
\end{align*}
we infer from \eqref{jump:Pk4-00} that
\begin{align}\label{jump:Pk4}
\mathcal{P}_4 \VV_-^{(k)}&=\mathcal{P}_4\Big(\VV_+^{(k)}-\PP_k^*(+\infty, x,t)\Big),\quad \text{for }(x,t)\in\Gamma,\\
\partial_\nu\mathcal{P}_4 \VV_-^{(k)}&=\partial_\nu\mathcal{P}_4\Big(\VV_+^{(k)}-\PP_k^*(+\infty, x,t)\Big)+\LL_k, \quad \text{for }(x,t)\in\Gamma.\label{jump:Pk4-d}
\end{align}
This gives two boundary conditions on $\Gamma$ for $\VV_-^{(k)}$ and $\VV_+^{(k)}$. The determination of $\LL_k$
and other boundary conditions for $\VV_-^{(k)}$ and $\VV_+^{(k)}$ will be derived from conditions (O2)-(O4) for $\FF_{k+1}$,
which will be explained in Subsection \ref{sec:deriv-innerBC-k}.

\subsection{Solving the equation of $\PP_1$}\label{subsec:p1}
Since $\partial_z\Phi, \partial_z\PP_0, \Phi_1, \Phi_2, \eta',\eta''$ decay to 0 exponentially as $z\to\pm\infty$,
the boundary condition (B1)-(B4) for $\FF_1$ are automatically satisfied.

Firstly, one has $\partial_z\Phi,~\Phi_1,~\Phi_2, ~ \partial_\nu\partial_z\Phi,~\partial_\nu\nn=0$ on $\Gamma$.
In addition, it holds that
\begin{align*}
  2\Phi^{\mathrm{T}}\nabla d_0\cdot\nabla_x\partial_z(\Phi\PP_0)=2\Phi^{\mathrm{T}}\partial_\nu\Phi\partial_z\PP_0=-4s'\overline\WW\nn\nn=0,
\end{align*}
as $\overline\WW\in\mathbb{V}_4$. Thus, we get
\begin{align}
  \FF_1=-(\partial_td_0-\Delta d_0)\partial_z\PP_0\qquad \text{ on }\Gamma_t.
\end{align}
Therefore, the orthogonal condition (O1) for $\FF_1$ on $\Gamma$ gives the equation of $d_0$:
\begin{align}\label{inner:eq-d0}
 a_0(\partial_td_0-\Delta d_0)=0,
\end{align}
which means that {\bf $\Gamma_t$ evolves according to the mean curvature flow}. This in turn gives that $\FF_1=0$ on $\Gamma$.
Then the orthogonal conditions (O2)-(O4) for $\FF_1$ on $\Gamma$ are automatically satisfied.

It is worthy to notice that, here we verified the conditions (O1)-(O4) for $\FF_1|_{\Gamma_t}$ from the boundary conditions (\ref{SharpInterfaceSys:MiniPair})-(\ref{SharpInterfaceSys:MCF}) of the sharp limit system. On the contrary,
the boundary conditions can be derived from (O1)-(O4) for $\FF_1|_{\Gamma_t}$
which is necessary for the solvability of $\PP_1\big|_{\Gamma_t}$. See Appendix \ref{sec:deriv-BC} for the proof.

Now we have $\FF_1=0$ on $\Gamma$. So we can define $\FF_1$ and $\GG_1$ in $\Gamma(\delta)$
by (\ref{expand:P1}) and (\ref{solution:Gk}) respectively. Note that $d_1$ is not determined yet.
However, a remarkable consequence of Lemma \ref{lem:Gknn} is that,
as we will see in (\ref{inner:eq-d1}), $d_1$ satisfies an equation only depending on the zeroth order terms $(\Phi, d_0, \nn)$.
Therefore, we can write
\begin{align}\label{solution:P1}
  \PP_1=s(z)\PP_{1,2}(x,t)+(1-s(z))\PP_{1,3}(x,t)+\PP_{1,4}(x,t)-\LL_1(x,t)d_0(x,t)\eta(z)+\PP_1^*(z,x,t),
\end{align}
with $\PP_1^*(z,x,t)$ uniquely determined by \eqref{solution:Pkstar}. Note that $\PP_1^*(z,x,t)$ depends only on the zeroth order terms $(\Phi, d_0, \nn)$. Moreover, $\PP_1^*=0$ on $\Gamma$.

\subsection{Solving the equation of  $\PP_2$}
We can directly check that the right hand side exponentially tends to its values at $\pm\infty$.
Again, we yield from (\ref{property:vanish}) and (\ref{cancel:Hk-on-interface}) that on $\Gamma$,
\begin{align}\nonumber
 \FF_2&=2\Phi^{\mathrm{T}}\nabla d_0\cdot\nabla_x\partial_z(\Phi\PP_1)-(\partial_td_1-\Delta d_1)\partial_z\PP_0
 +2\Phi^{\mathrm{T}}\nabla d_1\cdot\nabla_x\partial_z(\Phi\PP_0)\\
&\quad -\Phi^{\mathrm{T}}(\partial_t-\Delta_x)(\Phi\PP_{0})-(\PP_0\PP_1^{\mathrm{T}}\PP_1+\PP_1\PP_0^{\mathrm{T}}\PP_1+\PP_1\PP_1^{\mathrm{T}}\PP_0)
+(\GG_1\eta'+\LL_1\eta'')(d_1-z).
\end{align}
\subsubsection{The equation of $d_1$}
The orthogonal condition (O1) on $\Gamma$ gives us that
\begin{align*}
0=  \int_\mathbb{R} \FF_2:\partial_z\PP_0 \ud z.
\end{align*}
The above equation gives an evolution equation for $d_1$ on $\Gamma$. Surprisingly, it is independent of the choice of $\PP_1$.
As we mentioned above, this in turn gives a closed solution (\ref{solution:Gk}) ($k=1$) for $\GG_1$.

Firstly, due to \eqref{property:vanish}, we have $\partial_z\Phi=0, \partial_\nu\partial_z\Phi=0 $ on $\Gamma$, and thus
\begin{align*}
  2\Phi^{\mathrm{T}}\nabla d_0\cdot\nabla_x\partial_z(\Phi\PP_1)=2\Phi^{\mathrm{T}}\partial_\nu\Phi\partial_z\PP_1+2\partial_z(\partial_\nu\PP_1).
\end{align*}
Again, as $\Phi^{\mathrm{T}}\partial_\nu\Phi=\overline\WW\in\mathbb{V}_4$, one has
$\Phi^{\mathrm{T}}\partial_\nu\Phi\partial_z\PP_1:\partial_z\PP_0=0$. From (\ref{solution:P1}), one has $\partial_z\partial_\nu\PP_1:\nn\nn=0$ and thus $\partial_z(\partial_\nu\PP_1):\partial_z\PP_0=0$. Therefore, we obtain
\begin{align}\label{cancel:d1-1}
  2\Phi^{\mathrm{T}}\nabla d_0\cdot\nabla_x\partial_z(\Phi\PP_1):\partial_z\PP_0=0.
\end{align}
By Lemma \ref{lem:inner-null}, we have
\begin{align}\label{cancel:d1-2}
\int_\mathbb{R}  (\PP_0\PP_1^{\mathrm{T}}\PP_1+\PP_1\PP_0^{\mathrm{T}}\PP_1+\PP_1\PP_1^{\mathrm{T}}\PP_0):\partial_z\PP_0\ud z=0.
\end{align}
Since $\partial_z\Phi$ vanishes on $\Gamma$ quadratically in $d_0$ and $\Phi^{\mathrm{T}}\nabla d_1\cdot\nabla_x \Phi$ is antisymmetric, we have
\begin{align}\nonumber
2\Phi^{\mathrm{T}}\nabla d_1\cdot\nabla_x\partial_z(\Phi\PP_0) & =2\Phi^{\mathrm{T}}\nabla d_1\cdot\nabla_x \Phi\partial_z\PP_0+2\nabla d_1\cdot\nabla_x \partial_z\PP_0\\
&=2s'\Phi^{\mathrm{T}}\nabla d_1\cdot\nabla_x \Phi\nn\nn+2s'\nabla d_1\cdot\nabla_x(\nn\nn)\label{cancel:d1-3}
\end{align}
which is orthogonal to $\nn\nn$.
Therefore, combining the above conclusions  we obtain that
\begin{align}\label{inner:eq-d1}
  4a_0(\partial_t d_1-\Delta d_1)-\int_\BR (d_1-z)s'\GG_1:\nn\nn \ud z= \mathfrak{F}(\Phi, d_0, \nn)\qquad \text{ on }\Gamma_t,
\end{align}
which is linear due to (\ref{inner:eq-Gknn}) with $k=1$. Then we can extend $d_1$ from $\Gamma$ to $\Gamma_t(\delta)$ uniquely by $\nabla d_0\cdot\nabla d_1=0$.

\subsubsection{Determining $\FF_1, \GG_1, \PP_1^*$ in $\Gamma(\delta)$}
After $d_1$ is determined in $\Gamma(\delta)$, we define $\FF_1, \GG_1$ in $\Gamma(\delta)$ as in (\ref{expand:P1}) and (\ref{solution:Gk}).
Then $\PP_1^*$ can be determined from \eqref{solution:Pkstar}.

\subsubsection{The equations for $\PP_{1,i}(2\le i\le 4), \LL_1$ on $\Gamma$}
The conditions (O2)-(O4) for  $\FF_2|_{\Gamma}$ give us that
\begin{align}\label{inner:O-P1}
\int_\BR  \Big(2\Phi^{\mathrm{T}}\partial_\nu\Phi\partial_z\PP_1+2\partial_z(\partial_\nu\PP_1)-\Phi^{\mathrm{T}}(\partial_t-\Delta_x)(\Phi\PP_{0})
+(\GG_1\eta'+\LL_1\eta'')(d_1-z)\Big):\QQ(z,x,t) \ud z=0,
\end{align}
for $\QQ(z,x,t)=s\EE_2, (1-s)\EE_3, \EE_4$.
Here we have used the fact that $\PP_1|_{\Gamma_t}\in~\mathrm{Null}~\mathcal{L}$ and the cubic cancellation relation Lemma \ref{lem:inner-null}.

By \eqref{inner:proper-PhidPhi},  $\Phi^{\mathrm{T}}\partial_\nu\Phi$ is independent of $z$. Thus, we have
\begin{align}\label{cancel:dzP1:E4}
  \int_\BR 2\Phi^{\mathrm{T}}\partial_\nu\Phi\partial_z\PP_1:\EE_4 \ud z=2\Phi^{\mathrm{T}}\partial_\nu\Phi\int_\BR \partial_z\PP_1:\EE_4 \ud z=0.
\end{align}
Noting that $\partial_\nu\nn=0$ on $\Gamma$, we have $\partial_\nu \PP_{1,2}(x,t), \partial_\nu \PP_{1,3}(x,t)\bot\EE_4$. Thus
\begin{align}\label{cancel:P1:E4}
 \int_\BR  \Big(2\partial_z(\partial_\nu\PP_1)+\LL_1\eta''(d_1-z)\Big):\EE_4 \ud z
= \int_\BR  \Big(\LL_1\eta''(d_1-z)-2\LL_1\eta'\Big):\EE_4 \ud z
=-\LL_1:\EE_4.
\end{align}
Therefore, \eqref{inner:O-P1} with $\QQ(z,x,t)=\EE_4$ gives us that
\begin{align}\label{jump:P1-L1}
\LL_1=-\int_\BR \mathcal{P}_4 \Big(\Phi^{\mathrm{T}}(\partial_t-\Delta_x)(\Phi\PP_{0})
-\GG_1\eta'(d_1-z)\Big)\ud z.
\end{align}
Recalling \eqref{equation:Apm0}, we get
$\lim_{z\to\pm\infty}\Phi^{\mathrm{T}}(\partial_t-\Delta_x)(\Phi\PP_{0})=\A_-^{\mathrm{T}}(\partial_t-\Delta_x)\A_\pm \bot \mathbb{V}_4$, which gives
\begin{align}\label{decay-P4-1}
  \mathcal{P}_4 \Big(\Phi^{\mathrm{T}}(\partial_t-\Delta_x)(\Phi\PP_{0}) \Big)\to 0 ~~(\text{exponentially} ),\text{ as }z\to\pm\infty.
\end{align}
Thus, the integral in right side of \eqref{jump:P1-L1} is well defined.

Since
\begin{align*}
  \partial_z\partial_\nu\PP_1:s\EE_2=s'\partial_\nu\PP_{1,2}:s\EE_2,\quad \Phi^{\mathrm{T}}\partial_\nu\Phi\PP_{1,2}:\EE_2=0,
\end{align*}
we have
\begin{align} \nonumber
0= & \int_\BR  \Big(2\Phi^{\mathrm{T}}\partial_\nu\Phi\partial_z\PP_1+2\partial_z(\partial_\nu\PP_1)-\Phi^{\mathrm{T}}(\partial_t-\Delta_x)(\Phi\PP_{0})
+(\GG_1\eta'+\LL_1\eta'')(d_1-z)\Big):s\EE_2 \ud z\\\nonumber
=&\int_\BR  \Big(-\overline\WW\PP_{1,3}(s^2)'+(s^2)'\partial_\nu\PP_{1,2}-s\Phi^{\mathrm{T}}(\partial_t-\Delta_x)(\Phi\PP_{0})
+\GG_1\eta'(d_1-z)s\Big):\EE_2 \ud z,
\end{align}
which gives
\begin{align}
\label{jump:P12}
\partial_\nu\PP_{1,2}-\mathcal{P}_2(\overline\WW\PP_{1,3})
+\int_\BR  s\mathcal{P}_2\Big(-\Phi^{\mathrm{T}}(\partial_t-\Delta_x)(\Phi\PP_{0})+\GG_1\eta'(d_1-z)\Big) \ud z=0.
\end{align}
Similarly, we have on $\Gamma$ that
\begin{align}\label{jump:P13}
\partial_\nu\PP_{1,3}-\mathcal{P}_3(\overline\WW\PP_{1,2})
+\int_\BR  (1-s)\mathcal{P}_3\Big(-\Phi^{\mathrm{T}}(\partial_t-\Delta_x)(\Phi\PP_{0})+\GG_1\eta'(d_1-z)\Big) \ud z=0.
\end{align}
Similar to (\ref{decay-P4-1}), we have
\begin{align}\nonumber
 & \mathcal{P}_2 \Big(\Phi^{\mathrm{T}}(\partial_t-\Delta_x)(\Phi\PP_{0}) \Big)\to 0 ~~(\text{exponentially} ),\text{ as }z\to\infty,\\
 & \mathcal{P}_3 \Big(\Phi^{\mathrm{T}}(\partial_t-\Delta_x)(\Phi\PP_{0}) \Big)\to 0 ~~(\text{exponentially} ),\text{ as }z\to-\infty, \nonumber
\end{align}
which imply that the integrals in right side of (\ref{jump:P12}) and (\ref{jump:P13}) are well defined.

\subsubsection{Solving $\VV_\pm^{(1)}$ in ${Q}_\pm$} Using $[\partial_\nu, \II-2\nn\nn]=0$,
the equations \eqref{jump:Pk2}-\eqref{jump:Pk4-d}, \eqref{jump:P1-L1}, (\ref{jump:P12})-\eqref{jump:P13} together shows that for $(x,t)\in\Gamma$
\begin{align}\label{inner:BC-U1pm}
  \begin{split}
  \mathcal{P}_4 \VV_-^{(1)}&=\mathcal{P}_4\VV_+^{(1)},\\
\partial_\nu(\mathcal{P}_4 \VV_-^{(1)})&=\partial_\nu\mathcal{P}_4\VV_+^{(1)}+\LL_1,\\
\partial_\nu(\mathcal{P}_3 \VV_+^{(1)})-\mathcal{P}_3\big(\overline\WW\mathcal{P}_3 \VV_-^{(1)}\big)
&=(\II-2\nn\nn)\int_\BR  s\mathcal{P}_2\Big(-\Phi^{\mathrm{T}}(\partial_t-\Delta_x)(\Phi\PP_{0})+\GG_1\eta'(d_1-z)\Big) \ud z.\\
\partial_\nu(\mathcal{P}_3 \VV_-^{(1)})+\mathcal{P}_3\big(\overline\WW\mathcal{P}_3 \VV_+^{(1)}\big)
&=\int_\BR  (1-s)\mathcal{P}_3\Big(\Phi^{\mathrm{T}}(\partial_t-\Delta_x)(\Phi\PP_{0})-\GG_1\eta'(d_1-z)\Big) \ud z,
 \end{split}
 \end{align}
which provide boundary conditions for $\VV_+^{(1)}$ and $\VV_-^{(1)}$ on $\Gamma$. As $\VV_\pm^{(1)}=(\mathcal{P}_3+\mathcal{P}_4)\VV_\pm^{(1)}\in \mathbb{A}_n$,
these conditions are complete. So $\VV_+^{(1)}|_{{Q}_+}$ and $\VV_-^{(1)}|_{{Q}_-}$ can be solved from  a linear parabolic system.

\subsubsection{Determining $\PP_1$ in $\Gamma(\delta)$}\label{subsec:determine-P1}
We extend $\VV_+^{(1)}$ and $\VV_-^{(1)}$ to be smooth antisymmetric matrix-valued functions on $\{(x,t)|(x,t)\in\Gamma(\delta)\}$,
and let
\begin{align}\label{inner:def-P12}
&\PP_{1,2}(x,t)=-(\II-2\nn\nn)\mathcal{P}_3 \VV_+^{(1)}, \\ \label{inner:def-P13}
 &\PP_{1,3}(x,t)=\mathcal{P}_3 \VV_-^{(1)},\\\label{inner:def-L1}
 &\LL_1(x,t)=d_0^{-1}\mathcal{P}_4\big( \VV_-^{(1)}-\VV_+^{(1)}+\PP_1^*(+\infty, x,t)\big),\\\label{inner:def-P14}
& \PP_{1,4}(x,t)=\mathcal{P}_4 \VV_-^{(1)} =\mathcal{P}_4 \big(\VV_+^{(1)}-\PP_1^*(+\infty, x,t)\big)+\LL_{1}d_0.
 \end{align}
In addition, we define $\PP_1(z,x,t)$ as in \eqref{solution:P1}. Apparently, $\PP_1(z,x,t)$ is smooth in $(x,t)$
and satisfies the matching conditions \eqref{bc-Pk}.

\subsubsection{Determining $\FF_2, \GG_2, \PP_2^*$ in $\Gamma(\delta)$}
$\FF_2$ is determined by \eqref{inner:def-Fk} with $k=2$, and $\GG_2$ is determined by \eqref{solution:Gk}
which is derived from the orthogonal condition (O1)-(O4) for $\FF_2$.
Thus, we can solve $\PP_2^*(z,x,t)$ from \eqref{solution:Pkstar} and
write $\PP_2(z,x,t)$ as in \eqref{solution:Pk} with $k=2$. We remark that
$\PP_2^*(z,x,t)$ depends only on $(\Phi, d_0, \nn, \PP_1, d_1)$.

\begin{figure}\small
\begin{tikzpicture}[man/.style={rectangle,draw,fill=blue!20},
  woman/.style={rectangle,draw,fill=red!20,rounded corners=.8ex},
 VIP/.style={rectangle,very thick, draw}, fill=white]
\draw [dotted](-3,0.25) to (-3,-1*5.62);
\draw [dotted](-3,-7.1) to (-3,-1*13.62);
\draw [dotted](-3,-15.1) to (-3,-1*17.3);

\node (Out) at (0,0) {Inner expansion};
\node (Inn) at (-6,0) {Outer expansion};

\node[man] (A0) at (-0.4,-1) {$\Phi(z), \nn, \PP_0(z), d_0 $ in ${\Gamma(\delta)}$};
\node[man] (A0-out) at (-6,-1) {$\A_\pm^{(0)} $ in ${Q}_\pm$};

\node[man] (F1) at (-0.4,-1*2)  {$\FF_1(z), \PP_1^*(z), \GG_1:\nn\nn$ on $\Gamma$};
\node[man] (d11) at (-0.4,-1*3)  {$d_1$ on $\Gamma$};
\node[man] (d12) at (3,-1*3)  {$d_1$ in $\Gamma(\delta)$};
\node[man] (F12) at (-0.4,-1*4)  {$\FF_1(z)$, $\GG_1$ in $\Gamma(\delta)$};
\node[man] (P1star) at (-0.4,-1*5)  {$\PP_1^*(z)$ in $\Gamma(\delta)$};

\node[VIP] (BC1) at (1,-1*6-0.3)  {Boundary condition for $\VV_\pm^{(1)}$ on $\Gamma$: \eqref{inner:BC-U1pm}};
\node[man] (P1i) at (-0.2,-1*7-0.6)   {$\PP_{1,i}(z), \LL_1, \PP_1(z)$ in $\Gamma(\delta)$};
\node[woman] at (3.9,-1*7-0.6)  {\eqref{inner:def-P12}-\eqref{inner:def-P14},\eqref{solution:P1}};

\node[woman] at (4.1,-1*2) {(\ref{inner:def-Fk}), \eqref{solution:Pkstar}, (\ref{inner:eq-Gknn}) on $\Gamma$};
\node[woman] at (-2,-1*3)  {(\ref{inner:eq-d1})};
\node[woman] at (5,-1*3)  {(\ref{inner:dist-01})};
\node[woman] at (3,-1*4)  {(\ref{inner:def-Fk}) and (\ref{solution:Gk})};
\node[woman] at (1.8,-1*5)  {\eqref{solution:Pkstar}};

\node[man] (M1) at (-7,-1*5)  {$\MM_\pm^{(1)}=0$ in ${Q}_\pm$};
\node[man] (V1) at (-4.8,-1*7-0.6)  {$\VV_\pm^{(1)}$ in ${Q}_\pm$};
\node[woman] at (-4.8,-1*5)  {(\ref{out:sol-M1})};
\node[VIP] (Eq1) at (-6,-1*6-0.3)  {Equations for $\VV_\pm^{(1)}|_{{Q}_\pm}$: \eqref{eq:V1}};
\node  (plus1) at (-3,-1*6-0.3)  {+};

\node (stepk) at (2.9,-1*8-1)  {\underline{$\quad$ $\{d_i,~\PP_i, ~\LL_i, ~\GG_i| i\le k-1 \}$ are solved}};

\node (stepk) at (-7.2,-1*8-1)  {\underline{$\{\MM_\pm^{(i)},\VV_\pm^{(i)}|i\le k-1 \}$ are solved$\quad$ }};

\node[man] (F2) at (-0.4,-1*8-2)  {$\FF_k(z), \PP_k^*(z), \GG_k:\nn\nn$ on $\Gamma$};
\node[man] (d21) at (-0.4,-1*9-2)  {$d_k$ on $\Gamma$};
\node[man] (d22) at (3,-1*9-2)  {$d_k$ in $\Gamma(\delta)$};
\node[man] (F22) at (-0.4,-1*10-2)  {$\FF_k(z)$, $\GG_k$ in $\Gamma(\delta)$};
\node[man] (P2star) at (-0.4,-1*11-2)  {$\PP_k^*(z)$ in $\Gamma(\delta)$};
\node[VIP] (BCk) at (1,-1*12-2.3)  {Boundary conditions for $\VV_\pm^{(k)}$ on $\Gamma$: \eqref{inner:BC-Ukpm}};
\node[man] (P2i) at (-0.2,-1*13-2.6)  {$\PP_{k,i}(z), \LL_k, \PP_k(z)$ in $\Gamma(\delta)$};
\node[woman] at (3.9,-1*13-2.6)  {\eqref{inner:def-Pk2}-\eqref{inner:def-Pk4},\eqref{solution:Pk}};

\node[woman] at (3,-1*10-2)  {(\ref{inner:def-Fk}) and (\ref{solution:Gk})};
\node[woman] at (-2,-1*9-2)  {(\ref{solution:dk})};
\node[woman] at (5,-1*9-2)  {(\ref{solution:dk:extend})};
\node[woman] at (1.8,-1*11-2)  {\eqref{solution:Pkstar}};
\node[woman] at (4.1,-1*8-2)  {(\ref{inner:def-Fk}), \eqref{solution:Pkstar}, (\ref{inner:eq-Gknn}) on $\Gamma$};

\node[man] (Mk) at (-7,-1*11-2)  {$\MM_\pm^{(k)}$ in ${Q}_\pm$};
\node[man] (Vk) at (-4.8,-1*13-2.6)  {$\VV_\pm^{(k)}$ in ${Q}_\pm$};
\node[VIP] (Eqk) at (-6,-1*13-1.3)  {Equations for $\VV_\pm^{(k)}|_{{Q}_\pm}$: \eqref{eq:Vk}};
\node (plusk) at (-3,-1*13-1.3)  {+};
\node[woman] at (-5,-1*11-2)  {(\ref{eq:out-Ak})};

\node (stepk) at (3,-1*13-3.5)  {\underline{$\qquad$ $d_k,~\PP_k, ~\LL_k, ~\GG_k$ are solved}};
\node (stepk1) at (-7,-1*13-3.5)    {\underline{$\MM_\pm^{(k)},\VV_\pm^{(k)}$ are solved$\quad$}};

\draw [->](A0) to  (F1);
\draw [->](F1) to (d11);
\draw [->](d11) to (d12);
\draw [->](d12.south west) to (F12);
\draw [->](F12) to (P1star);
\draw [->](V1) to (P1i);

\draw [->](A0-out) to  (M1);

\draw [->](M1) to (Eq1);
\draw [->](P1star) to (BC1);
\draw [->](V1) to (P1i);
\draw [->](plus1) to [in=0,out=-90,looseness=0.75] (V1.north east);

\draw [->](F2) to (d21);
\draw [->](d21) to (d22);
\draw [->](d22.south west) to (F22);
\draw [->](F22) to (P2star);

\draw [->] (-6,-1*10) to (Mk);
\draw [->](Mk) to (Eqk);
\draw [->](P2star) to (BCk);
\draw [->](Vk) to (P2i);
\draw [->](plusk) to [in=0,out=-90,looseness=0.75] (Vk.north east);
\draw [very thick, dotted](0, -1*8.2) to (0, -8.82);
\draw [very thick, dotted](-4.8, -1*8.2) to (-4.8, -8.82);

\draw [very thick, dotted](0, -1*16.7) to (0, -1*17.3);
\draw [very thick, dotted](-4.8, -1*16.7) to (-4.8, -1*17.3);
\end{tikzpicture}
\caption{The whole procedure to solve the outer and inner expansion systems}
\label{fig:expansion}
\end{figure}

\subsection{Solving the equation of  $\PP_k$ ($k\ge 3$)} Repeating the above procedure, we can solve $(\PP_k, d_k)$ step by step.
Assume that in $\Gamma(\delta)$
\begin{align*}
\{d_i,~\PP_i, ~\LL_i, ~\GG_i|0\le i\le k-1 \},
\end{align*}
has been known. Then we  solve $d_{k}, \PP_{k}, \LL_{k}, \GG_{k}$ in $\Gamma(\delta)$.

Firstly, we can obtain $\FF_k|_{\Gamma_t}$ and $\PP_{{k}}^*|_{\Gamma_t}$ from \eqref{inner:def-Fk} and \eqref{solution:Pkstar} respectively.
Note that, we can also determine $(\GG_{{k}}:\nn)|_{\Gamma_t}$ from \eqref{inner:eq-Gknn}. Note that all these terms are independent of $d_k$.

On $\Gamma$, from \eqref{expand:Pk}, we can write
\begin{align*}
\FF_{k+1}=&-(\partial_td_{{k}}-\Delta d_{{k}})\partial_z\PP_{0}+2\Phi^{\mathrm{T}}\nabla d_0\cdot\nabla_x\partial_z(\Phi\PP_{k})
+2\Phi^{\mathrm{T}}\nabla d_{{k}}\cdot\nabla_x\partial_z(\Phi\PP_0)\\ \nonumber
&-\sum_{\substack{i+j={k}\\ 1\le i,j\le {k-1}}}(\partial_td_i-\Delta d_i)\partial_z\PP_j+2\Phi^{\mathrm{T}}\sum_{\substack{i+j={k}\\ 1\le i,j\le {k-1}}}
\nabla d_i\cdot\nabla_x\partial_z(\Phi\PP_j)\\ \nonumber
&-\Phi^{\mathrm{T}}(\partial_t-\Delta)(\Phi\PP_{{k-1}})-\sum_{\substack{i+j+l=k+1\\ 0\le i,j,l\le {k}}}\PP_i\PP_j^{\mathrm{T}}\PP_l+2\Phi^{\mathrm{T}}\nabla d_0\cdot\nabla_x\partial_z(\Phi\PP_{{k}})\\ \nonumber
&+(\GG_1\eta'+\LL_1\eta'')d_{{k}}+(\GG_{{k}}\eta'+\LL_{{k}}\eta'')(d_1-z)
+\sum_{i=2}^{{k-1}}(\GG_{i}\eta'+\LL_{i}\eta'')d_{{k+1}-i}.
\end{align*}
\subsubsection{The equation for $d_{{k}}$} We use the condition (O1) for $\FF_{k+1}$.
Let $\PP_{{k}}^\top=\PP_{{k}}-\PP_{{k}}^*$.
Similar to the derivation of \eqref{cancel:d1-1}, we can obtain
\begin{align*}
  2\Phi^{\mathrm{T}}\nabla d_0\cdot\nabla_x\partial_z(\Phi\PP_{{k}}^\top):\partial_z\PP_0=0.
\end{align*}
Moreover, similar to \eqref{cancel:d1-2} and \eqref{cancel:d1-3}, we can obtain
\begin{align*}
&\int_\mathbb{R}  \Big\{\big(\PP_0\PP_1^{\mathrm{T}}+\PP_1\PP_0^{\mathrm{T}}\big)\PP_{{k}}^\top+\PP_{{k}}^\top\big(\PP_0^{\mathrm{T}}\PP_1+\PP_1^{\mathrm{T}}\PP_0\big)
+\PP_0\PP_{{k}}^{\top,\mathrm{T}}\PP_1+\PP_1\PP_{{k}}^{\top,\mathrm{T}}\PP_0\Big\}:\partial_z\PP_0\ud z=0,\\
&2\Phi^{\mathrm{T}}\nabla d_{{k}}\cdot\nabla_x\partial_z(\Phi\PP_0) :\nn\nn
=2s'\big(\Phi^{\mathrm{T}}\nabla d_{{k}}\cdot\nabla_x \Phi\nn\nn+\nabla d_{{k}}\cdot\nabla_x(\nn\nn)\big):\nn\nn=0.
\end{align*}
Therefore, the equation for $d_{{k}}$ on $\Gamma$ reads as
\begin{align}\nonumber
&4a_0(\partial_td_{{k}}-\Delta d_{{k}})+d_{{k}}\int_\BR\GG_1:\nn\nn\eta's'\ud z=\int_\BR\Big\{2\Phi^{\mathrm{T}}\nabla d_0\cdot\nabla_x\partial_z(\Phi\PP^*_{{k}})\\ \nonumber
&-\sum_{\substack{i+j={k}\\ 1\le i,j\le {k-1}}}(\partial_td_i-\Delta d_i)\partial_z\PP_j+2\Phi^{\mathrm{T}}\sum_{\substack{i+j={k}\\ 1\le i,j\le {k-1}}}
\nabla d_i\cdot\nabla_x\partial_z(\Phi\PP_j)-\Phi^{\mathrm{T}}(\partial_t-\Delta)(\Phi\PP_{{k-1}})\\ \label{solution:dk}
&-\sum_{\substack{i+j+l=k+1\\ 0\le i,j,l\le {k}}}\PP_i\PP_j^{\mathrm{T}}\PP_l+\GG_{{k}}\eta'(d_1-z)+\sum_{i=2}^{{k-1}}s'\GG_{i}\eta'd_{{k}+1-i}\Big\}:\nn\nn s'\ud z.
\end{align}
Note that the right hand side is independent of $\PP_k$, and from Lemma \ref{lem:Gknn} or (\ref{inner:eq-Gknn}), it is also independent of $d_{{k}}$.

After $d_{{k}}|_{\Gamma}$ is determined,  $d_{{k}}$ can be extended to $\Gamma(\delta)$ by the following ODE
\begin{align} \label{solution:dk:extend}
 2 \nabla d_0\cdot\nabla d_{{k}}+\sum_{i=1}^{{k-1}}\nabla d_i\cdot\nabla d_{{k}-i}=0.
\end{align}

\subsubsection{Determining $\FF_k, \GG_k, \PP_k^*$ in $\Gamma(\delta)$} $\FF_{{k}}$, $\GG_{{k}}$ and $\PP_{{k}}^*$ in $\Gamma(\delta)$ can be determined by (\ref{inner:def-Fk}), \eqref{solution:Gk} and \eqref{solution:Pkstar} accordingly.

\subsubsection{The equation for $\PP_{{k},i}(2\le i\le 4), \LL_{{k}}$ on $\Gamma$} \label{sec:deriv-innerBC-k}
We use the condition (O2)-(O4) for $\FF_{k+1}$ on $\Gamma$. For $\QQ(z,x,t)=s\EE_2, (1-s)\EE_3, \EE_4$ with $\EE_i\in\mathbb{V}_i$, we have
\begin{align*}
&\int_\BR\Big\{2\Phi^{\mathrm{T}}\partial_\nu\Phi\partial_z\PP_{{k}}^\top+2\partial_z(\partial_\nu\PP_{{k}}^\top)
+\LL_{{k}}\eta''(d_1-z)\Big\}:\QQ(z,x,t)\ud z \\
&\quad+\int_\BR\Big\{-\sum_{i+j={k}, i\ge 1}(\partial_td_i-\Delta d_i)\partial_z\PP_j+2\Phi^{\mathrm{T}}\sum_{i+j={k}, i\ge 1}
\nabla d_i\cdot\nabla_x\partial_z(\Phi\PP_j)\\ \nonumber
&\qquad+2\Phi^{\mathrm{T}}\nabla d_0\cdot\nabla_x\partial_z(\Phi\PP^*_{{k}})-\Phi^{\mathrm{T}}(\partial_t-\Delta)(\Phi\PP_{{k-1}})
-\sum_{\substack{i+j+l=k+1\\ 0\le i,j,l\le {k}}}\PP_i\PP_j^{\mathrm{T}}\PP_l
\\ \nonumber
&\qquad+(\GG_1\eta'+\LL_1\eta'')d_{{k}}+\GG_{{k}}\eta'(d_1-z)
+\sum_{i=2}^{{k-1}}(\GG_{i}\eta'+\LL_{i}\eta'')d_{{k}+1-i}\Big\}:\QQ(z,x,t) \ud z=0.
\end{align*}
We denote the last three lines by $\int_\BR \TT_{{k}}(z,x,t):\QQ(z,x,t) \ud z$, in which all terms are known functions.

Taking $\QQ=\EE_4$ and using the same argument as in (\ref{cancel:dzP1:E4})-(\ref{cancel:P1:E4}), one has
\begin{align*}
&\int_\BR\Big\{2\Phi^{\mathrm{T}}\partial_\nu\Phi\partial_z\PP_{{k}}^\top+2\partial_z(\partial_\nu\PP_{{k}}^\top)
+\LL_{{k}}\eta''(d_1-z)\Big\}:\EE_4\ud z =-\LL_{{k}}:\EE_4.
\end{align*}
Thus, we have
\begin{align}\label{jump:Pk:4}
\partial_\nu\mathcal{P}_4\PP_{{k}}^+-\partial_\nu\mathcal{P}_4\PP_{{k}}^- =- \LL_{{k}}=\mathcal{P}_4\int_\BR \TT_{{k}}(z,x,t)\ud z.
\end{align}
Similarly, we have
\begin{align} \label{jump:Pk:2}
&\partial_\nu\PP_{{k},2}-\mathcal{P}_2(\bar\WW\PP_{{k},3})
+\mathcal{P}_2\int_\BR  s(z)\TT_{{k}}(z,x,t) \ud z=0,\\
\label{jump:Pk:3}
&\partial_\nu\PP_{{k},3}-\mathcal{P}_3(\bar\WW\PP_{{k},2})
+\mathcal{P}_3\int_\BR  (1-s(z))\TT_{{k}}(z,x,t) \ud z=0.
\end{align}

\subsubsection{Solving $\VV_+^{(k)}$ and $\VV_-^{(k)}$ in ${Q}_\pm$}
The equations \eqref{jump:Pk2}-\eqref{jump:Pk4-d}, \eqref{jump:Pk:4}-\eqref{jump:Pk:3} together give that
\begin{align}\label{inner:BC-Ukpm}
  \begin{split}
\mathcal{P}_4\VV_+^{(k)}-  \mathcal{P}_4 \VV_-^{(k)}&=\mathcal{P}_4\PP_k^*(+\infty, x,t),\\
\partial_\nu(\mathcal{P}_4\VV_+^{(k)})-\partial_\nu(\mathcal{P}_4 \VV_-^{(k)})&=\mathcal{P}_4\int_\BR \TT_{{k}}(z,x,t)\ud z, \\
\partial_\nu(\mathcal{P}_3 \VV_+^{(k)})-\mathcal{P}_3\Big(\bar\WW\mathcal{P}_3 \VV_-^{(k)}\Big)
&=(\II-2\nn\nn)\int_\BR s\mathcal{P}_2\TT_{{k}}(z,x,t) \ud z.\\
\partial_\nu(\mathcal{P}_3 \VV_-^{(k)})+\mathcal{P}_3\Big(\bar\WW\mathcal{P}_3 \VV_+^{(k)}\Big)
&=-\int_\BR  (1-s)\mathcal{P}_3\TT_{{k}}(z,x,t) \ud z,
 \end{split}
 \end{align}
which offer complete boundary conditions for $\VV_+^{(k)}$ and $\VV_-^{(k)}$ on $\Gamma$. Thus combining \eqref{eq:Vk} we can solve
$\VV_+^{(k)}|_{{Q}_+}$ and $\VV_-^{(k)}|_{{Q}_-}$.

\subsubsection{Determining $\PP_{k}$ in $\Gamma(\delta)$}
Afterwards, as in Subsection \ref{subsec:determine-P1}, we extend $\VV_+^{(k)}$ and $\VV_-^{(k)}$ to be smooth antisymmetric matrix-valued functions in $\Gamma(\delta)$,
and let
\begin{align}\label{inner:def-Pk2}
&\PP_{k,2}(x,t)=-(\II-2\nn\nn)\mathcal{P}_3 \VV_+^{(k)}, \\ \label{inner:def-Pk3}
 &\PP_{k,3}(x,t)=\mathcal{P}_3 \VV_-^{(k)},\\\label{inner:def-Lk}
& \LL_k(x,t)=d_0^{-1}(\mathcal{P}_4 \VV_-^{(k)}-\mathcal{P}_4 \VV_+^{(k)}),\\\label{inner:def-Pk4}
 &\PP_{k,4}(x,t)=\mathcal{P}_4 \VV_-^{(k)} =\mathcal{P}_4 \Big(\VV_+^{(k)}-\PP_k^*(+\infty, x,t)\Big)+\LL_{k}d_0.
 \end{align}
Moreover, now we can define $\PP_k(z,x,t)$ as in \eqref{solution:P1}. Obviously, $\PP_k(z,x,t)$ is smooth in $(x,t)$
and the matching conditions \eqref{bc-Pk} are satisfied.

As a result, we have solved $d_{k}, \PP_{k}, \LL_{k}, \GG_{k}$ in $\Gamma(\delta)$.
Therefore, by repeating the above steps, the expansion system can be solved from an induction argument.
Note that, in each step, we only need to solve a linear system whose well-posedness can be shown directly. The whole procedure
is illustrated in Figure \ref{fig:expansion}.

\begin{Remark}\label{Remark:lift}
  The minimal paired condition (\ref{SharpInterfaceSys:MiniPair}) and Lemma \ref{lem:mini-pair} give us that there exists a smooth map
$\NN(x,t):\Gamma\to\{\nn\nn: \nn\in S^{n-1}\}$ such that $\A_-(x,t)=\A_+(x,t)(\II-2\NN(x,t))$ for all $(x,t)\in\Gamma$.
For given $(x_0,t_0)\in \Gamma$ and any neighbourhood $U\subset \Gamma$ of $(x_0,t_0)$,
there exists a smooth vector field $\nn(x,t):U\to S^{n-1}$ such that $\NN(x,t)=(\nn\nn)(x,t)$.
We have assumed that such a lifting map, which keeps the regularity, exists
globally in $\Gamma$. This assumption is made just for simplicity and clarity of presentations, as our previous analysis does not rely on the particular choice
of $\nn$ or $-\nn$. For examples, the condition $\partial_\nu\nn=0$ can be replaced by $\partial_\nu\NN=0$ without causing any obstacles in our arguments, and the decomposition and projections in \eqref{def:subspace} and \eqref{def:projection} indeed depend only on $\NN=\nn\nn$.

\end{Remark}

\subsection{Proof of Theorem \ref{thm:main1}}
Now we glue the outer expansion and inner expansion to obtain the approximate solutions in the whole region $\Omega$.
Let
\begin{align*}
 \A_O^K(x,t)&=\sum_{k=0}^{K} \ve^k \big(\A_+^{(k)}\chi_{{Q}_+}+\A_-^{(k)}\chi_{{Q}_-}\big)\,\, \text{ for } (x,t)\in{Q}_\pm.
\end{align*}
Then it holds that for $(x,t)\in Q_\pm$,
\begin{align*}
(\partial_t-\Delta) \A_O^K-\ve^{-2}f(\A_O^K)=
&\sum_{k=0}^{K-2}\ve^{k}\Big\{(\partial_t-\Delta)\A_\pm^{(k)}-\sum_{i+j+l=k+2}
\A_\pm^{(i)}(\A_\pm^{(j)})^{\mathrm{T}}\A_\pm^{(l)}\Big\}+O(\ve^{K-1})\\
=&O(\ve^{K-1}).
\end{align*}

For $(x,t)\in\Gamma(\delta)$, we define:
\begin{align*}
 & d^K(x,t)=\sum_{k=0}^{K} \ve^k d_k(x,t),\\
  &\A_I^K(x,t)=\Phi(\ve^{-1}d^K,x,t)\sum_{k=0}^{K} \ve^k \PP_k(\ve^{-1}d^K,x,t),\\
  &\GG^K=\sum_{k=1}^{K}\ve^k \GG_k(x,t),\quad
  \LL^K=\sum_{k=1}^{K}\ve^k \LL_k(x,t),\quad
  \HH^{\pm,K}=\sum_{k=0}^{K-2}\ve^k \HH^{\pm}_k(x,t),
\end{align*}
with $d_k, \Phi, \PP_k, \GG_k, \LL_k,\HH^{\pm}_k$ defined in Sections \ref{sec:inner-expan}-\ref{sec:construct}. Then
\begin{align*}
  |\nabla d^K|^2=1+\sum_{1\le i, j\le K, i+j\ge K+1} \ve^{i+j}\nabla d_i\nabla d_j=1+O(\ve^{K+1}),
\end{align*}
and for $(x,t)\in\Gamma(\delta)$, we have
\begin{align*}
&(\partial_t-\Delta) \A_I^K-\ve^{-2}f(\A_I^K)\\
&=\Big\{\ve^{-2}\Phi\big(-\partial_z^2\PP^K+f(\PP^{K})-\Phi^{\mathrm{T}}\partial_z^2 \Phi\PP^K
-2\Phi^{\mathrm{T}}\partial_z\Phi\partial_z\PP^K\big)\\
&\quad+\ve^{-1}\Phi\big[(\partial_td^K-\Delta d^K)(\partial_z\PP^K+\Phi^{\mathrm{T}}\partial_z\Phi)-2\Phi^{\mathrm{T}}\nabla d^K\nabla\partial_z(\Phi\PP^K)\big]\\
&\quad+(\partial_t-\Delta)(\Phi\PP^K)\Big\}\Big|_{z=\ve^{-1}d^{K}}+O(\ve^{K-1})\\
&=\Big\{\ve^{-2}\Phi\big[-\partial_z^2\PP^K+f(\PP^{K})+(d^K-\ve z-d_0)(\Phi_1\PP^K+\Phi_2\partial_z\PP^K)+(d^K-\ve z)(\GG^K\eta'+\LL^K\eta'')\big]\\
&\quad+\ve^{-1}\Phi\big[(\partial_td^K-\Delta d^K)(\partial_z\PP^K+\Phi^{\mathrm{T}}\partial_z\Phi)-2\Phi^{\mathrm{T}}\nabla d^K\nabla\partial_z(\Phi\PP^K)\big]\\
&\quad+(\partial_t-\Delta)(\Phi\PP^K)+\Phi\HH^{+,K}(x,t)\eta_M^+(z)+\Phi\HH^{-,K}(x,t)\eta_M^-(z)\Big\}\Big|_{z=\ve^{-1}d^{K}}+O(\ve^{K-1})\\
&=O(\ve^{K-1}).
\end{align*}
Here in the last equality, we used the expansion systems (\ref{expand:P0})-(\ref{expand:Pk}), which imply that all $O(\ve^{k})(k\le K-2)$ terms are cancelled by each others.
Moreover, for $(x,t)\in \Gamma(\delta)$, due to the matching condition (\ref{matching condition}), we have
\begin{align}\label{match}
|\partial_t^i\partial^j_x\partial^l_z  (\A_I^K- \A_O^K)|\le C e^{-\alpha_0|d^K(x,t)|/\ve } \le C e^{-\alpha_0|d_0(x,t)|/\ve }.
\end{align}

Therefore, if we define
\begin{align*}
    \A^K(x,t)&= \big\{1-\tilde{\chi}\big(d_0(x,t)\delta^{-1}\big)\big\}\A_O^K(x,t)+\tilde{\chi}\big(d_0(x,t)\delta^{-1}\big) \A_I^K(x,t),
\end{align*}
where $\tilde{\chi}$ be a smooth nonnegative function satisfying $\mathrm{supp}~\tilde{\chi}\subset(-1,1)$ and $\tilde{\chi}(z)=1$ for $|z|\le 1/2$, then
it holds  that
\begin{align*}
(\partial_t-\Delta) \A^K-\ve^{-2}f(\A^K)=\mathfrak{R}^{K-1}\sim O(\ve^{K-1}),
\end{align*}
in the whole domain $\Omega$. Moreover, we have $\partial_x^i\mathfrak{R}^{K-1}=O(\ve^{K-1-i}).$
This finishes the proof of Theorem \ref{thm:main1}.

\section{Spectral lower bound estimate for the linearized operator}\label{sec:spectral-estimate}
This section is devoted to proving Theorem \ref{thm:main2}.
The proof is accomplished by five  steps, which are done in Subsections \ref{subsec:decompose}-\ref{subsec:estimate-correction} respectively:

\smallskip

{\it Step 1: reduce to 1-D interval}. Firstly, we restrict the inequality into the region near the interface(see (\ref{spectral inequality-inner})), as it holds on regions away from the interface by direct estimation.
  Then we introduce two transformations, one for coordinates and the other for matrix fields,  to reduce the problem into an inequality on an one dimensional interval, which can be divided into two estimates for cross terms and correction terms respectively; see (\ref{ineq:1-dim-cross}) and (\ref{ineq:1-dim-nextorder}).
\smallskip

{\it Step 2: decompose into scalar inequalities}. Motivated by the diagolization of $\mathcal{L}_{\PP_0}$, we use a bases decomposition  to reduce a matrix-valued problem into several estimates related to the scalar linearized operators $\mathcal{L}_i$ defined in \eqref{pre:decomp-Op-L}; see (\ref{ineq:1-dim-cross-final-2}) and (\ref{ineq:nextorder-decom}).
\smallskip

{\it Step 3: coercive estimates and endpoints $L^\infty$-control}. We develop the coercive estimates for the scalar linearized operators $\mathcal{L}_i$(Lemmas \ref{lem:lower-bound-23}, \ref{lem:lower-bound-1} and \ref{lem:lower-bound-1-out}), $L^\infty$-control at endpoints (Lemmas \ref{lem:Linf} and \ref{lem:Linf-endpoint}).


\smallskip

{\it Step 4: estimate the cross terms}. The cross terms involving $\mathcal{L}_4$ or $\mathcal{L}_5$ can be controlled directly as we have strong coercive estimates for these two operators; see Proposition \ref{prop:cross-45}. However, the coercive estimates for $\mathcal{L}_1$-$\mathcal{L}_3$ are relatively weak, thus the same method could not apply to control the corresponding cross terms, which becomes somewhat technical.

Motivated by \cite{FWZZ}, we observe that the weights in all cross terms
involving $\mathcal{L}_1$ are indeed small by using the homogeneous Neumann boundary condition of $\nn$.
This is the key that it enables us to remove the singularity of cross terms involving $\mathcal{L}_1$; see Lemma \ref{lem:E1-smallness}
and Proposition \ref{prop:cross-1}.

The estimates for the cross terms involving $\mathcal{L}_2$ and $\mathcal{L}_3$ are much more
involved, since neither we have strong coercive  estimates, nor the weights are small.
We accomplish it by a  product estimate (see Proposition \ref{prop:cross-23}), which is proved by applying a symmetric structure for the eigenfunctions of $\mathcal{L}_2$ and $\mathcal{L}_3$ (see Lemma \ref{lem:cross-estimate-23}).
\smallskip

{\it Step 5: estimate the singular correction terms}. We explicitly decompose the   singular correction terms. Then
the inequality is reduced to  some new  product estimates similar to Lemma  \ref{lem:cross-estimate-23}. The proof of these product estimate
also rely on the important symmetric structures between the first eigenfunctions
of $\mathcal{L}_i(1\le i\le 4)$; see Lemmas \ref{lem:cross-estimate-34}-\ref{lem:cross-estimate-13}. \bigskip

We remark again that, for estimates in Steps 3-5, we have repeatedly applied elementary decompositions based on the eigenfunctions of scalar linearized operators $\{\mathcal{L}_i(1\le i\le5)\}$. This method provides new and elementary proofs for the spectral estimates of the operators $\mathcal{L}_i$.

\subsection{Reduction to inequalities on an interval}

First of all,  we choose $\ve$ small enough such that
\begin{align*}
 \Gamma_t(\delta/8) \subset \Gamma_t^K(\delta/4):=\{x: |d^K(x,t)|<\delta/4 \} \subset \Gamma_t(\delta/2),\quad \text{for }t\in[0, T].
\end{align*}
For $x\in\Omega_t^\pm$, $\A^{(0)}\in \mathbb{O}_n$ and $\A^{(1)}\in \A_+\mathbb{A}_n $ or $\A_-\mathbb{A}_n$. We decompose
\begin{align*}
\A=\KK+\JJ\qquad\text{with }\KK\in\A_\pm \mathbb{A}_n, ~\JJ\in\A_\pm\mathbb{S}_n .
\end{align*}
Then we have
\begin{align*}
\mathcal{H}_{\A^{(0)}}\A :\A&~=\mathcal{H}_{\A^{(0)}}\JJ:\JJ =2|\JJ|^2,\\
\TT_f(\A^{(0)},\A^{(1)},\A):\A&~=\TT_f(\A^{(0)},\A^{(1)},\JJ):\JJ+2\TT_f(\A^{(0)},\A^{(1)},\JJ):\KK,
\end{align*}
where $\TT_f$ is defined in \eqref{def:BB}.
Thus, for $x\in\Omega_t^\pm$, one can find a constant $C$ such that
\begin{align*}
\ve^{-2}\mathcal{H}_{\A^{(0)}}\A:\A +\ve^{-1}\TT_f(\A^{(0)},\A^{(1)},\A) :\A \ge -C|\KK|^2\ge -C|\A|^2.
\end{align*}
In addition, in $\Omega\setminus\Gamma_t^K(\delta/4)$, due to the matching condition \eqref{match}, we have
\begin{align*}
  \A^K-\A^{(0)}-\ve \A^{(1)}=\tilde{\chi}\big(d(x,t)\delta^{-1}\big)(\A_I^K-\A_O^K)+O(\ve^2)=O(e^{-\frac{\alpha_0\delta}{4\ve}})+O(\ve^2)=O(\ve^2).
\end{align*}
Thus, in $\Omega\setminus\Gamma_t^K(\delta/4)$, one has that
\begin{align*}
 \ve^{-2}\mathcal{H}_{\A^K}\A:\A = \ve^{-2}\mathcal{H}_{\A^{(0)}}\A:\A +\ve^{-1}\TT_f(\A^{(0)},\A^{(1)},\A) :\A+O(\ve^2 )\|\A\|^2 \ge -C\|\A\|^2.
\end{align*}

For $x\in \Gamma_t^K(\delta/4)$, we let
\begin{align*}
\A_0=\big(\Phi\PP_0\big)(\ve^{-1}d^K(x,t),x,t),\qquad \A_1=\big(\Phi\PP_1\big)(\ve^{-1}d^K(x,t),x,t).
\end{align*}
Then
\begin{align*}
  \A^K= \A_I^K=\A_0+\ve \A_1+O(\ve^2).
\end{align*}
Therefore, it suffices to prove that
\begin{align}
\int_{\Gamma_t^K(\delta/4)}\Big(\|\nabla\A\|^2 +\ve^{-2}\big(\mathcal{H}_{\A_0}\A:\A\big)
+\ve^{-1}\big(\TT_f(\A_0,\A_1,\A):\A\big)\Big)\ud x \geq-C\int_{\Gamma_t^K(\delta/4)}\|\A\|^2\ud x.
\label{spectral inequality-inner}
\end{align}

For each $t\in[0,T]$, we introduce the curved coordinate $(\sigma, r)$ in $\Gamma_t^K(\delta/4)$ with $\{x(\sigma, 0)\} =\Gamma_t^K$ and
\begin{align*}
\partial_rx(\sigma,r)=\frac{\nabla d^K}{|\nabla d^K|^2}\circ(x(\sigma, r), t)
\end{align*}
Then $\frac{\ud}{\ud r}( d^K(x(\sigma, r))-r)=0$. As $d^K(x(\sigma, 0))=0$, we have
\begin{align*}
  d^K(x(\sigma, r), t)=r.
\end{align*}
The coordinate transformation $x\mapsto (\sigma, r)$ is a diffeomorphism from $\Gamma_t^K(\delta/4)$ to $\Gamma_t^K\times [-\delta/4, \delta/4]$.
Let $J(\sigma, r)=\det(\frac{\partial x(\sigma, r)}{\partial(\sigma, r)})$ be the Jacobian of the transformation. Then
\begin{align*}
  J|_{r=0} =1,~ J(\sigma, r)=1+O(r),~\ud x=J\ud\sigma\ud r,~ \text{ and } \partial_rf=\partial_rx\cdot\nabla f=\frac{\nabla d^K\cdot \nabla f}{|\nabla d^K|^2}.
\end{align*}
Therefore, as $|\nabla d^K|^2=1+O(\ve^{K+1})$, we have
\begin{align}\label{project-deriv}
  |\nabla f|^2\ge \Big|\frac{\nabla d^K}{|\nabla d^K|}\cdot \nabla f\Big|^2\ge (\partial_r f)^2+O(\ve^{K+1})|\nabla f|^2.
\end{align}

The inequality (\ref{spectral inequality-inner}) is equivalent to
\begin{align}\nonumber\label{ineq:coordi-trans}
&\int_{\Gamma_t^K}\ud\sigma\int_{-\frac{\delta}{4}}^{\frac{\delta}{4}}\Big(\|\nabla\A\|^2 +\ve^{-2}\big(\mathcal{H}_{\A_0}\A:\A\big)
+\ve^{-1}\big(\TT_f(\A_0,\A_1,\A):\A\big)\Big)J(\sigma, r)\ud r \\
&\geq-C\int_{\Gamma_t^K}\ud\sigma\int_{-\frac{\delta}{4}}^{\frac{\delta}{4}} \|\A\|^2J(\sigma, r)\ud r.
\end{align}
Using \eqref{project-deriv}, it suffices to prove that for each $\sigma\in \Gamma_t^K$,
\begin{align}\label{ineq:one-dim}
\int_{-\frac{\delta}{4}}^{\frac{\delta}{4}}\Big(\|\partial_r\A\|^2 +\ve^{-2}\big(\mathcal{H}_{\A_0}\A:\A\big)
+\ve^{-1}\big(\TT_f(\A_0,\A_1,\A):\A\big)\Big)J\ud r
\geq-C\int_{-\frac{\delta}{4}}^{\frac{\delta}{4}} \|\A\|^2J\ud r.
\end{align}

Let $\BB(x,t)=\Phi^{\mathrm{T}}\big(\ve^{-1}d^K(x,t),x,t\big)\A(x,t)$ or $\A(x,t)=\Phi\big(\ve^{-1}d^K(x,t),x,t\big)\BB(x,t)$. Then
\begin{align*}
&\|\partial_r \A\|^2=\|\partial_r \BB\|^2+\|\partial_r \Phi \BB\|^2+2\Phi\partial_r \BB : \partial_r \Phi\BB,\\
&\mathcal{H}_{\A_0}\A:\A=\mathcal{H}_{\PP_0}\BB:\BB,\\
&\TT_f(\A_0,\A_1,\A):\A=\TT_f(\PP_0,\PP_1,\BB):\BB
\end{align*}
Therefore, (\ref{ineq:one-dim}) is reduced to
\begin{multline}\label{ineq:one-dim-2}
\int_{-\frac{\delta}{4}}^{\frac{\delta}{4}}\Big(\|\partial_r \BB\|^2 +\ve^{-2}\mathcal{H}_{\PP_0}\BB:\BB
+\ve^{-1}\TT_f(\PP_0,\PP_1,\BB):\BB+2\Phi\partial_r \BB : \partial_r \Phi\BB \Big)J\ud r\\
\geq-C\int_{-\frac{\delta}{4}}^{\frac{\delta}{4}} \|\BB\|^2J\ud r,
\end{multline}
which can be concluded from the following two inequalities:
\begin{align}\label{ineq:1-dim-cross}
&\int_{-\frac{\delta}{4}}^{\frac{\delta}{4}} \Phi\partial_r \BB : \partial_r \Phi\BB J\ud r
\le
\frac{1}{4}\int_{-\frac{\delta}{4}}^{\frac{\delta}{4}}\Big(\|\partial_r \BB\|^2 +\ve^{-2}\mathcal{H}_{\PP_0}\BB:\BB\Big)J\ud r+C\int_{-\frac{\delta}{4}}^{\frac{\delta}{4}} \|\BB\|^2J\ud r,\\ \label{ineq:1-dim-nextorder}
&\int_{-\frac{\delta}{4}}^{\frac{\delta}{4}}\ve^{-1}\TT_f(\PP_0,\PP_1,\BB):\BB J\ud r
\le
\frac{1}{4}\int_{-\frac{\delta}{4}}^{\frac{\delta}{4}}\Big(\|\partial_r \BB\|^2 +\ve^{-2}\mathcal{H}_{\PP_0}\BB:\BB\Big)J\ud r+C\int_{-\frac{\delta}{4}}^{\frac{\delta}{4}} \|\BB\|^2J\ud r.
\end{align}
The left hand sides of (\ref{ineq:1-dim-cross}) and (\ref{ineq:1-dim-nextorder}) are called cross terms and
correction terms of next order respectively, which will
be estimated in the following two subsections separately.

In the sequel, without loss of generality, we will assume $\delta/4=1$ to simplify the notations.

\subsection{Reduction to the inequalities for scalar functions} \label{subsec:decompose}
Recall that $\BV_i$ is a finite dimensional space which only  depends on $\nn(x,t)$. So for given $\sigma\in\Gamma_t^K$, we can choose
$\{\EE_{\alpha}:~\alpha\in\Lambda_i\}$ be a set of complete orthogonal bases of $\BV_i$ which are smooth in $r$.
Let $\Lambda=\cup_{i=1}^5\Lambda_i$. Then we can write
\begin{align*}
  \BB=\sum_{\alpha\in\Lambda} p_\alpha \EE_\alpha.
\end{align*}
As $\PP_0=(\II-2s_\ve(r)\nn\nn)$ with $s_\ve(\cdot)=s((\cdot)/\ve)$, a direct calculation(see \eqref{pre:decomp-Op-L}) leads to
\begin{align}\label{ident:bulk-decomp}
  \mathcal{H}_{\PP_0}\BB:\BB=\sum_{i=1}^5 \sum_{\alpha\in \Lambda_i} \kappa_i(s_\ve(r)) p_\alpha^2,
\end{align}
where $\kappa_i$ are defined in \eqref{pre:decomp-Op-L}.
Moreover,
\begin{align}\nonumber
  \|\partial_r\BB\|^2&=\|\sum_{\alpha\in\Lambda} \partial_r p_\alpha \EE_\alpha+p_\alpha\partial_r\EE_\alpha\|^2\\ \nonumber
  &\ge \sum_{\alpha} |\partial_r p_\alpha |^2+2\sum_{\alpha\neq \beta}\partial_r p_\alpha p_\beta\EE_\alpha:\partial_r\EE_\beta\\
  &= \sum_{\alpha} |\partial_r p_\alpha |^2+\sum_{\alpha\neq \beta}(\partial_r p_\alpha p_\beta-p_{\alpha}\partial_r p_\beta)\EE_\alpha:\partial_r\EE_\beta.   \label{cross-1}
\end{align}
Then we obtain
\begin{align}\nonumber
&\int_{-\frac{\delta}{4}}^{\frac{\delta}{4}}\Big(\|\partial_r \BB\|^2 +\ve^{-2}\mathcal{H}_{\PP_0}\BB:\BB\Big)J\ud r+C\int_{-\frac{\delta}{4}}^{\frac{\delta}{4}} \|\BB\|^2J\ud r\\
&\ge \int_{-1}^{1}\sum_{i=1}^5 \sum_{\alpha\in \Lambda_i}\Big(|\partial_r p_\alpha|^2 + \frac{1}{\ve^2}\kappa_i(s_\ve(r)) p_\alpha^2\Big)J\ud r+C\int_{-1}^{1}\sum_{\alpha }|p_\alpha|^2J\ud r\nonumber\\
&\quad +\int_{-1}^{1}\sum_{\alpha\neq\beta}
(\partial_r p_\alpha p_\beta-p_\alpha\partial_r p_\beta )\EE_\alpha:\partial_r\EE_\beta J\ud r.\label{ineq:decompL}
\end{align}
Since for $\alpha\in \Lambda_i(i=1,2,5), \beta\in\Lambda_j(j=3,4)$, one has $\EE_\alpha\in\BS_n,\partial_r\EE_\beta\in\BA_n$ which yields that
$\EE_\alpha:\partial_r\EE_\beta=0$. Thus, we only have to consider the case
$\alpha,\beta\in \Lambda_1\cup\Lambda_2\cup\Lambda_5$ or $\alpha,\beta\in \Lambda_3\cup\Lambda_4$.\smallskip

Now we remove $J$ via the endpoint estimates established in Lemmas \ref{lem:Linf} and \ref{lem:Linf-endpoint}. Let $q_\alpha=J^{\frac{1}{2}}p_\alpha$. Then
\begin{align*}
 \int_{-1}^{1} |\partial_r p_\alpha|^2 J\ud r
 &=\int_{-1}^{1} |\partial_r(J^{-\frac{1}{2}} q_\alpha)|^2 J\ud r
 =\int_{-1}^{1} |\partial_rq_\alpha|^2 +|\partial_r(J^{-\frac{1}{2}})|^2J q_\alpha^2+2
 \partial_r(J^{-\frac{1}{2}}) q_\alpha J^{\frac{1}{2}} \partial_rq_\alpha\ud r\\
 &=\int_{-1}^{1} |\partial_rq_\alpha|^2
 +\Big(|\partial_r(J^{-\frac{1}{2}})|^2J-\partial_r(\partial_rJ^{-\frac{1}{2}} J^{\frac{1}{2}})\Big) q_\alpha^2\ud r
 + \partial_rJ^{-\frac{1}{2}} J^{\frac{1}{2}} q_\alpha^2\Big|_{-1}^{1}\\
&\ge \int_{-1}^{1} |\partial_rq_\alpha|^2\ud r
-\nu_0 \int_{-1}^{1} \Big(|\partial_rq_\alpha|^2+\frac{1}{\ve^2}\kappa_i(s_\ve)q_\alpha^2 \Big)\ud r
-C\int_{-1}^{1} q_\alpha^2\ud r.
\end{align*}
Here we have used  Lemmas \ref{lem:Linf} and \ref{lem:Linf-endpoint} in the last inequality to control $|q_\alpha(\pm 1)|$.
Thus, we get
\begin{align}\nonumber
&\int_{-1}^{1}\sum_{i=1}^5 \sum_{\alpha\in \Lambda_i}\Big(|\partial_r p_\alpha|^2 + \frac{1}{\ve^2}\kappa_i(s_\ve(r)) p_\alpha^2\Big)J\ud r+C\int_{-1}^{1}\sum_{\alpha }|p_\alpha|^2J\ud r\\ \label{ineq:removeJ}
&\ge (1-\nu_0)\int_{-1}^{1}\sum_{i=1}^5 \sum_{\alpha\in \Lambda_i}\Big(|\partial_r q_\alpha|^2 + \frac{1}{\ve^2}\kappa_i(s_\ve(r)) q_\alpha^2\Big)\ud r+C\int_{-1}^{1}\sum_{\alpha }|q_\alpha|^2\ud r.
\end{align}

Let
\begin{align}\label{def:W}
\WW=  \Phi^{\mathrm{T}}\partial_r \Phi
\end{align}
which is antisymmetric. Notice that
\begin{multline}\label{cross-2}
    \Phi\partial_r \BB : \partial_r \Phi\BB J=\partial_r \BB:\WW \BB  J=\sum_{\alpha,\beta}  J p_\beta(\partial_r p_\alpha\EE_\alpha+p_\alpha\partial_r \EE_\alpha):\WW \EE_\beta\\
  =\frac12\sum_{\alpha\neq\beta}  J(\partial_r p_\alpha p_\beta-p_\alpha\partial_r p_\beta )\EE_\alpha:\WW \EE_\beta+
  \sum_{\alpha,\beta} J p_\alpha p_\beta\partial_r \EE_\alpha:\WW\EE_\beta.
\end{multline}
as well as that
\begin{align*}
& \int_{-1}^{1}  \sum_{\alpha,\beta}p_\alpha p_\beta\partial_r \EE_\alpha:\WW\EE_\beta J\ud r
 \le C \int_{-1}^{1}  \sum_{\alpha }|p_\alpha|^2 J\ud r=C\int_{-1}^{1} \sum_{\alpha }|q_\alpha|^2\ud r.
\end{align*}
Then combining (\ref{ineq:decompL}), (\ref{ineq:removeJ}), (\ref{cross-2}) and
\begin{align*}
 (\partial_r p_\alpha p_\beta-p_\alpha\partial_r p_\beta )J=\partial_r (q_\alpha J^{-\frac12}) q_\beta J^{\frac12}-q_\alpha\partial_r (q_\beta J^{-\frac12})J^{\frac12}
  =  \partial_r q_\alpha q_\beta-q_\alpha\partial_r q_\beta,
\end{align*}
the inequality (\ref{ineq:1-dim-cross})  can be deduced from
\begin{align}\nonumber
&\int_{-1}^{1}\sum_{\alpha\neq\beta}
(\partial_r q_\alpha q_\beta-q_\alpha\partial_r q_\beta )a(r)\ud r
 \\
&\qquad\le
\frac{1}{4}\int_{-1}^{1}\sum_{i=1}^5 \sum_{\alpha\in \Lambda_i}\Big(|\partial_r q_\alpha|^2 + \frac{1}{\ve^2}\kappa_i(s_\ve(r)) q_\alpha^2\Big)\ud r+C\int_{-1}^{1}\sum_{\alpha }|q_\alpha|^2\ud r,\label{ineq:1-dim-cross-final-2}
\end{align}
for $a(r)= \EE_\alpha:\partial_r\EE_\beta$ or $\EE_\alpha:\WW \EE_\beta$.

To obtain (\ref{ineq:1-dim-nextorder}), we introduce
\begin{align}
  \tilde{\BB}=J^{1/2}\BB=\sum_{\alpha\in\Lambda} q_\alpha \EE_\alpha. \label{esti:decomp-new}
\end{align}
 Then (\ref{ineq:1-dim-nextorder}) can be deduced from  (\ref{ineq:1-dim-cross-final-2})  and
\begin{align} \nonumber
&\int_{-\frac{\delta}{4}}^{\frac{\delta}{4}}\ve^{-1}\TT_f(\PP_0,\PP_1,\tilde\BB):\tilde\BB \ud r
 \\
&\qquad\le
\frac{1}{4}\int_{-1}^{1}\sum_{i=1}^5 \sum_{\alpha\in \Lambda_i}\Big(|\partial_r q_\alpha|^2 + \frac{1}{\ve^2}\kappa_i(s_\ve(r)) q_\alpha^2\Big)\ud r+C\int_{-1}^{1}\sum_{\alpha }|q_\alpha|^2\ud r.\label{ineq:nextorder-decom}
\end{align}

\subsection{Coercive estimates for $\mathcal{L}_i$ ($1\le i\le 3$)}\label{subsec:spec:CoerEsti}

\subsubsection{Spectral lower bound estimates}

Let us introduce some lemmas on the coercivity of $\mathcal{L}_i$ ($1\le i\le 3$).
\begin{Lemma}\label{lem:lower-bound-23}
Let $q_2=s_\ve \bar{q}_2$ and $q_3=(1-s_\ve)\bar{q}_3$. Then for any $\nu_0>0$, there exists $C_0(\nu_0)>0$ such that
\begin{align}\label{ineq:lower-q2}\nonumber
&\int_{-1}^1\Big((\partial_rq_2)^2+\varepsilon^{-2}\kappa_2(s_\ve)q_2^2\Big)\ud r\\
&\quad\ge\Big(\frac12+O\big(\ve^{-1}e^{-\frac{\sqrt{2}}{\ve}}\big)-\nu_0\Big)\int_{-1}^1s_\ve^2
\big(\partial_r\bar{q}_2\big)^2\ud r-O\big(e^{-\frac{\sqrt{2}}{\ve}}\big)C_0(\nu_0)\int_{-1}^1s_\ve^2\bar{q}_2^2\ud r,\\\nonumber
&\int_{-1}^1\Big((\partial_rq_3)^2+\varepsilon^{-2}\kappa_3(s_\ve)q_3^2\Big)\ud r\\
&\quad\ge
\Big(\frac12+O\big(\ve^{-1}e^{-\frac{\sqrt{2}}{\ve}}\big)-\nu_0\Big)\int_{-1}^1(1-s_\ve)^2
\big(\partial_r\bar{q}_3\big)^2\ud r-O\big(e^{-\frac{\sqrt{2}}{\ve}}\big)C_0(\nu_0)\int_{-1}^1(1-s_\ve)^2\bar{q}_3^2\ud r.\label{ineq:lower-q3}
\end{align}
\end{Lemma}
\begin{Remark}
  The constant $1/2$ on the right hand side is optimal.
\end{Remark}

\begin{proof}

Using the fact that $s''=s\kappa_2(s)$,  we arrive at
\begin{align}
&\int_{-1}^1\bigg((\partial_rq_2)^2+\varepsilon^{-2}\kappa_2(s_\ve)\big(q_2\big)^2\bigg)\ud r\nonumber\\\nonumber
&=\int_{-1}^1\Big(\big[\partial_rs_\ve\bar{q}_2+s_\ve\partial_r\bar{q}_2\big]^2+\varepsilon^{-2}\kappa_2(s_\ve)s_\ve^2\bar{q}_2^2\Big)\ud r\\\nonumber
&=s_\ve\partial_rs_\ve\bar{q}_2^2\Big|_{-1}^1+\int_{-1}^1s_\ve^2(\partial_r\bar{q}_2)^2\ud r\\
&\ge -(s_\ve\partial_rs_\ve\bar{q}_2^2)(-1)+\int_{-1}^1s_\ve^2(\partial_r\bar{q}_2)^2\ud r.\nonumber
\end{align}
As $s_\ve(r)=1/(1+e^{-\sqrt{2}r/\ve})$, one can directly get
\begin{align*}
(s_\ve \partial_rs_\ve)(-1)&=\ve^{-1}\sqrt{2}s_\ve^2(1-s_\ve)(-1)=\frac{\sqrt{2}}{\ve}e^{-2\sqrt{2}/\ve}\Big( 1+O\big(e^{-{\sqrt{2}}/{\ve}}\big)\Big),\\
\int_{-1}^0s_\ve^{-2}\ud r&=\int_{-1}^0(1+e^{-\sqrt{2}r/\ve})^2\ud r
=\frac{\ve e^{2\sqrt{2}/\ve} }{2\sqrt{2}}\Big( 1+O\big(\ve^{-1}e^{-{\sqrt{2}}/{\ve}}\big)\Big).
\end{align*}
Moreover, from the Gagliardo-Nirenberg inequality, we have
\begin{align*}
\bar{q}_2(0)^2\leq~&\frac{\nu_1^2}{2} \int_0^1\big(\partial_r\bar{q}_2\big)^2\ud r
+C(\nu_1)\int_0^1\bar{q}_2^2\ud r\\
\leq ~&2{\nu_1^2} \int_0^1s_\ve^2\big(\partial_r\bar{q}_2\big)^2\ud r+C(\nu_1)\int_0^1s_\ve^2\bar{q}_2^2\ud r.
\end{align*}
Thus choosing $\nu_1=\nu_0/2$ and $\ve$ sufficiently small, we obtain
\begin{align*}
&(s_\ve\partial_rs_\ve)(-1)|\bar{q}_2(-1)|^2\\
&\le  (s_\ve\partial_rs_\ve)(-1)\Big( |\bar{q}_2(0)|+\Big|\int_{-1}^0\partial_r\bar{q}_2\ud r\Big|\Big)^2 \\
&\le  (s_\ve\partial_rs_\ve)(-1)\Big( (1+\nu_1^{-1}) |\bar{q}_2(0)|^2+(1+\nu_1)\Big|\int_{-1}^0(s_\ve\partial_r\bar{q}_2)^2\ud r\Big|\Big|\int_{-1}^0s_\ve^{-2}\ud r\Big|\Big)\\
&\le O\big(e^{-\frac{\sqrt{2}}{\ve}}\big)\Big(3\nu_1 \int_0^1s_\ve^2\big(\partial_r\bar{q}_2\big)^2\ud r+C(\nu_1)\int_0^1s_\ve^2\bar{q}_2^2\ud r\Big)
+\Big(\frac{1+\nu_1}{2}+O\big(\ve^{-1}e^{-\frac{\sqrt{2}}{\ve}}\big)\Big)\int_{-1}^0(s_\ve\partial_r\bar{q}_2)^2\ud r\\
&\le \Big(\frac12+\nu_0+O\big(\ve^{-1}e^{-\frac{\sqrt{2}}{\ve}}\big)\Big)\int_{-1}^1s_\ve^2
\big(\partial_r\bar{q}_2\big)^2\ud r+O\big(e^{-\frac{\sqrt{2}}{\ve}}\big)C_1(\nu_0)\int_{-1}^1s_\ve^2\bar{q}_2^2\ud r.
\end{align*}
Then (\ref{ineq:lower-q2}) follows immediately. The inequality (\ref{ineq:lower-q3}) can be proved in a similar way.
\end{proof}

We define
\begin{align*}
  \theta_\ve(r)=(s')\Big(\frac{r}{\ve}\Big)=\sqrt{2}\big(s(1-s)\big)\Big(\frac{r}{\ve}\Big)=\ve\partial_rs_\ve.
\end{align*}
\begin{Lemma}\label{lem:lower-bound-1}
Let  $q_1=\theta_\ve\bar{q}_1$.
For any $\nu_0>0$, there exists $C_0(\nu_0)>0$ such that for $\ve$ sufficiently small
\begin{align}\nonumber
&\int_{-1}^1\Big((\partial_rq_1)^2+\varepsilon^{-2}\kappa_1(s_\ve)q_1^2\Big)\ud r\\
&\ge
\Big(\frac12+O\big(\ve^{-1}e^{-\frac{\sqrt{2}}{\ve}}\big)-\nu_0\Big)\int_{-1}^1\theta_\ve^2
\big(\partial_r\bar{q}_1\big)^2\ud r-O\big(\ve^{-2}e^{-\frac{2\sqrt{2}}{\ve}}\big)
C_0(\nu_0)\int_{-1}^1\theta_\ve^2\bar{q}_1^2\ud r.\label{ineq:lower-q1}
\end{align}
\end{Lemma}
\begin{Remark}
The inequality (\ref{ineq:lower-q1}) gives another version of the first eigenvalue estimate for $\mathcal{L}_1$, which has been proved in \cite{Chen}.
Note that here no boundary condition is needed.
\end{Remark}
\begin{proof}
Direct calculations yield that
\begin{align}
&\int_{-1}^1\bigg((\partial_rq_1)^2+\varepsilon^{-2}\kappa_1(s_\ve)\big(q_1\big)^2\bigg)\ud r\nonumber\\\nonumber
&=\int_{-1}^1\Big(\big[\partial_r\theta_\ve\bar{q}_1+\theta_\ve\partial_r\bar{q}_1\big]^2+\varepsilon^{-2}\kappa_1(s_\ve)\theta_\ve^2\bar{q}_1^2\Big)\ud r\\\nonumber
&=\theta_\ve\partial_r\theta_\ve\bar{q}_1^2\Big|_{-1}^1+\int_{-1}^1\theta_\ve^2(\partial_r\bar{q}_1)^2\ud r.
\end{align}
One can directly get
\begin{align*}
&|(\theta_\ve \partial_r\theta_\ve)(\pm 1)|=\ve^{-1}2\sqrt{2}s_\ve^2(1-s_\ve)^2|1-2s_\ve|
=\frac{2\sqrt{2}}{\ve}e^{-2\sqrt{2}/\ve}\Big( 1+O\big(e^{-{\sqrt{2}}/{\ve}}\big)\Big),\\
&\int_{0}^1\theta_\ve^{-2}\ud r=\frac{1}{2}\int_{0}^1e^{2\sqrt{2}r/\ve}\ud r\Big(1+O\big(e^{-{\sqrt{2}}/{\ve}}\big)\Big)
=\frac{\ve e^{2\sqrt{2}/\ve} }{4\sqrt{2}}\Big(1+O\big(\ve^{-1}e^{-{\sqrt{2}}/{\ve}}\big)\Big).
\end{align*}
Moreover, from the Gagliardo-Nirenberg inequality and a scaling argument, we have
\begin{align*}
|\bar{q}_1(0)|^2\leq~&\frac{\nu_1^2}{100}\ve\int_{-\ve}^{\ve}\big(\partial_r\bar{q}_1\big)^2\ud r
+C(\nu_1) \frac{1}{\ve}\int_{-\ve}^{\ve}\bar{q}_1^2\ud r\\
\leq ~&2{\nu_1^2} \ve\int_{-\ve}^{\ve}\theta_\ve^2\big(\partial_r\bar{q}_1\big)^2\ud r+\ve^{-1}C(\nu_1)\int_{-\ve}^{\ve}\theta_\ve^2\bar{q}_1^2\ud r.
\end{align*}
Thus choosing $\nu_1=\nu_0/2$ and $\ve$ sufficiently small, we obtain
\begin{align}\nonumber
&|(\theta_\ve\partial_r\theta_\ve)(1)||\bar{q}_1(1)|^2\\ \nonumber
&\le  |(\theta_\ve\partial_r\theta_\ve)(1)|\Big( |\bar{q}_1(0)|+\Big|\int_{0}^1\partial_r\bar{q}_1\ud r\Big|\Big)^2 \\ \nonumber
&\le  |(\theta_\ve\partial_r\theta_\ve)(1)|\Big( (1+\nu_1^{-1}) |\bar{q}_1(0)|^2+(1+\nu_1)\Big[\int_0^1(\theta_\ve\partial_r\bar{q}_1)^2\ud r\Big]\int_0^1\theta_\ve^{-2}\ud r\Big)\\ \nonumber
&\le O\big({\ve^{-1}}e^{-2\sqrt{2}/\ve}\big)\Big(3\nu_1 \ve \int_{-\ve}^{\ve}\theta_\ve^2\big(\partial_r\bar{q}_1\big)^2\ud r
+\ve^{-1}C(\nu_1)\int_{-\ve}^{\ve}\theta_\ve^2\bar{q}_1^2\ud r\Big)\\ \nonumber
&\qquad+\Big(\frac{1+\nu_1}{2}+O\big(\ve^{-1}e^{-\frac{\sqrt{2}}{\ve}}\big)\Big)\int_0^1(\theta_\ve\partial_r\bar{q}_1)^2\ud r\\ \label{ineq:endpoint-1}
&\le \Big(\frac12+\nu_0+O\big(\ve^{-1}e^{-\frac{\sqrt{2}}{\ve}}\big)\Big)\int_{-1}^1\theta_\ve^2
\big(\partial_r\bar{q}_1\big)^2\ud r+O\big(\ve^{-2}e^{-\frac{2\sqrt{2}}{\ve}}\big)C_1(\nu_0)\int_{-1}^1\theta_\ve^2\bar{q}_1^2\ud r.
\end{align}
The same estimate holds for the value at $-1$. Thus, (\ref{ineq:lower-q1}) follows.
\end{proof}

\smallskip

The following lemma gives a new proof for the second eigenvalue estimate of $\mathcal{L}_1$.
\begin{Lemma}\label{lem:lower-bound-1-out}
  Let  $q_1=\mu\theta_\ve+\hat{q}_1$ with $\int_{-1}^1\theta_\ve\hat{q}_1\ud r=0$.
  Then for any $\nu_0>0$, there exists $C_0(\nu_0)>0$ such that for $\ve$ sufficiently small
\begin{align}
&\int_{-1}^1\Big((\partial_rq_1)^2+\varepsilon^{-2}\kappa_1(s_\ve)q_1^2\Big)\ud r+O\big(\ve^{-2}e^{-\frac{2\sqrt{2}}{\ve}}\big)
\int_{-1}^1{q}_1^2\ud r\ge \frac{C}{\ve^2}\int_{-1}^1\hat{q}_1^2\ud r.
\label{ineq:lower-q1-out}
\end{align}
\end{Lemma}
\begin{proof}
Let $\bar{q}_1=q_1/\theta_\ve-\mu=\hat{q}_1/\theta_\ve$. Then it follows from (\ref{ineq:lower-q1}) that
\begin{align*}\nonumber
&\int_{-1}^1\Big((\partial_rq_1)^2+\varepsilon^{-2}\kappa_1(s_\ve)q_1^2\Big)\ud r+O\big(\ve^{-2}e^{-\frac{2\sqrt{2}}{\ve}}\big)
\int_{-1}^1{q}_1^2\ud r\\
&\ge \Big(\frac14+O\big(\ve^{-1}e^{-\frac{\sqrt{2}}{\ve}}\big)\Big)\int_{-1}^1\theta_\ve^2
\big(\partial_r\bar{q}_1\big)^2\ud r.
\end{align*}
Now we prove that for $\int_{-1}^1\theta_\ve^2
\bar{q}_1\ud r=0$,
\begin{align*}
  \int_{-1}^1\theta_\ve^2\big(\partial_r\bar{q}_1\big)^2\ud r
  \ge\frac{C}{\ve^2} \int_{-1}^1\theta_\ve^2 \bar{q}_1^2\ud r= \frac{C}{\ve^2} \int_{-1}^1 \hat{q}_1^2\ud r.
\end{align*}
Let
\begin{align*}
A_1=\int_{0}^1\theta_\ve^2(\tau)\bar{q}_1^2(\tau)\ud \tau ,\quad B_1= \int_{0}^1\theta_\ve^2 \big(\partial_r\bar{q}_1\big)^2\ud r, \quad C_1=\int_{0}^1\theta_\ve^2(\tau)\bar{q}_1(\tau)\ud \tau,\\
A_2=\int_{-1}^0\theta_\ve^2(\tau)\bar{q}_1^2(\tau)\ud \tau ,\quad B_2= \int_{-1}^0\theta_\ve^2 \big(\partial_r\bar{q}_1\big)^2\ud r, \quad C_2=\int_{-1}^0\theta_\ve^2(\tau)\bar{q}_1(\tau)\ud \tau.
\end{align*}
Assume that $A_1+A_2>0$. We have
\begin{align}\nonumber
\int_{0}^1\theta_\ve^2(\tau)\bar{q}_1^2(\tau)\ud \tau & =\int_{0}^1 \theta_\ve^2(\tau)\bar{q}_1(\tau)\Big( \bar{q}_1(0)+\int_{0}^\tau \partial_r\bar{q}_1\ud r\Big)\ud\tau \\ \nonumber
 & =\bar{q}_1(0)C_1+\int_{0}^1 \theta_\ve^2(\tau)\bar{q}_1(\tau)\Big(\int_{0}^\tau \partial_r\bar{q}_1\ud r\Big)\ud\tau \\ \nonumber
 & \le\bar{q}_1(0)C_1 +\Big(\int_{0}^1 \theta_\ve^2(\tau)\bar{q}^2_1(\tau)\ud\tau\Big)^{1/2}\Big(\int_{0}^1\theta_\ve^2(\tau)\Big(\int_{0}^\tau \partial_r\bar{q}_1\ud r\Big)^{2}\ud\tau\Big)^{1/2}.
\end{align}
On the other hand, we have
\begin{align*}
  \int_{0}^1\theta_\ve^2(\tau)\Big(\int_{0}^\tau \partial_r\bar{q}_1\ud r\Big)^2\ud\tau
&\le \int_{0}^1\theta_\ve^2(\tau)\Big(\int_{0}^\tau \theta_\ve (\partial_r\bar{q}_1)^2\ud r\Big)
\Big(\int_{0}^\tau \theta_\ve^{-1}(r)\ud r\Big)\ud\tau\\
&= \int_{0}^1 \theta_\ve (\partial_r\bar{q}_1)^2 I(r) \ud r
\end{align*}
with
\begin{align*}
I(r)= \int_{r}^1 \theta_\ve^2(\tau) \Big(\int_{0}^\tau \theta_\ve^{-1}(y)\ud y\Big)\ud\tau=\ve^2\int_{\frac{r}{\ve}}^{\frac{1}{\ve}} \theta^2(z) \Big(\int_{0}^z \theta^{-1}(w)\ud w\Big)\ud z\le C \ve^2\theta_\ve(r).
\end{align*}
Thus
\begin{align*}
  A_1\le\bar{q}_1(0) C_1+C\ve A_1^{1/2}B_1^{1/2}.
\end{align*}
Similarly, we have
\begin{align*}
  A_2\le\bar{q}_1(0) C_2+C\ve  A_2^{1/2}B_2^{1/2}.
\end{align*}
As $C_1+C_2=0$, we get
\begin{align*}
 A_1+ A_2\le C\ve A_1^{1/2}B_1^{1/2}+C\ve A_2^{1/2}B_2^{1/2}\le C\ve ( A_1+ A_2)^{1/2}( B_1+ B_2)^{1/2}
\end{align*}
which concludes our lemma.
\end{proof}\smallskip

As $|\partial_r\theta_\ve|\le \frac{C}{\ve}\theta_\ve$, and $|\partial_r\hat{q}_1|\le|\partial_r\theta_\ve \bar{q}_1|+|\theta_\ve\partial_r\bar{q}_1| $, we have the following
\begin{Corollary}\label{lem:lower-bound-1-deriv-out}
  Let  $q_1=\mu\theta_\ve+\hat{q}_1$ with $\int_{-1}^1\theta_\ve\hat{q}_1\ud r=0$.
  Then  there exists $C>0$ such that for $\ve$ sufficiently small
\begin{align}
&\int_{-1}^1\Big((\partial_rq_1)^2+\varepsilon^{-2}\kappa_1(s_\ve)q_1^2\Big)\ud r+O\big(\ve^{-2}e^{-\frac{2\sqrt{2}}{\ve}}\big)
\int_{-1}^1{q}_1^2\ud r\ge C\int_{-1}^1(\partial_r\hat{q}_1)^2\ud r.
\label{ineq:lower-bound-1-deriv-out}
\end{align}

\end{Corollary}

\subsubsection{Endpoints $L^\infty$ estimates}

\begin{Lemma}[$L^\infty$ control] \label{lem:Linf}
For $i=2,3,4,5$, and any $\nu_0>0$, there exists $C(\nu_0)$ such that
\begin{align*}
  \|q\|_{L^\infty([-1,1])}^2\le \nu_0  \int_{-1}^1\Big((\partial_rq)^2
  +\varepsilon^{-2}\kappa_i(s_\ve)q^2\Big)\ud r+C(\nu_0)\int_{-1}^1q^2\ud r.
  \end{align*}
\end{Lemma}
\begin{proof}
The inequalities for $i=4,5$ are obvious, since $\kappa_4=0$ and $\kappa_5\ge 0$. For $i=2$, let $q=s_\ve \bar{q}$. Then we have
 \begin{align*}
 \|\bar{q}\|_{L^\infty([0,1])}^2\le \nu_0 \int_{0}^1s_\ve^2(\partial_r\bar{q})^2\ud r+C(\nu_0)\int_{0}^1s_\ve^2\bar{q}^2\ud r.
 \end{align*}
 For $r\in[-1,0]$, we have
 \begin{align*}
 |q(r)|&=s_\ve(r)|\bar{q}(r)|\le s_\ve(r)(|\bar{q}(0)|+\int_{0}^r|\partial_r\bar{q}|\ud r)\\
 &\le \frac12 |\bar{q}(0)|+s_\ve(r)\Big(\int_{0}^rs_\ve^2|\partial_r\bar{q}|^2\ud r\Big)^{1/2}\Big(\int_{0}^rs_\ve^{-2}\ud r\Big)^{1/2}\\
 &\le  \frac12 |\bar{q}(0)|+O(\sqrt{\ve}) \Big(\int_{0}^rs_\ve^2|\partial_r\bar{q}|^2\ud r\Big)^{1/2}.
 \end{align*}
Thus using Lemma \ref{lem:lower-bound-23}, we obtain the claim for $i=2$. The proof for $i=3$ is similar.
\end{proof}

The above lemma is not true for $i=1$. However, as a corollary of (\ref{ineq:endpoint-1}) and Lemma \ref{lem:lower-bound-1}, we have the following endpoint control.
\begin{Lemma}[Endpoints control] \label{lem:Linf-endpoint}
There exists $C>0$ such that
\begin{align*}
  |q(\pm1)|^2\le C\ve \Big\{ \int_{-1}^1\Big((\partial_rq)^2
  +\varepsilon^{-2}\kappa_1(s_\ve)q^2\Big)\ud r+\int_{-1}^1q^2\ud r\Big\}.
  \end{align*}
\end{Lemma}
\begin{proof}
Let $q=\theta_\ve\bar{q}$. Note that $\theta_\ve^2(\pm 1)\le C \ve |(\theta_\ve\partial_r\theta_\ve)(\pm 1)|$, then we can get the result.
\end{proof}

\subsection{Estimate for cross terms}\label{subsec:estimate-cross}

Now the inequality \eqref{ineq:1-dim-cross-final-2} is a consequence of the following Propositions \ref{prop:cross-45}, \ref{prop:cross-1}
and \ref{prop:cross-23} by letting $a(r)=\EE_\alpha:\partial_r\EE_\beta$ or $\EE_\alpha:\WW \EE_\beta$.

\begin{Proposition}\label{prop:cross-45}Assume $i$ or $j$ $\in \{4, 5\}$.
Then for any $\nu_0>0$ there exists $C_0=C_0(\nu_0, \|a\|_{W^{1,\infty}})>0$ such that
\begin{align}
&\int_{-1}^1\big(\partial_rq_\alpha q_\beta-q_\alpha\partial_rq_\beta\big)a(r)\ud r
\nonumber\\&\leq\nu_0\int_{-1}^1\Big((\partial_rq_\alpha )^2+\varepsilon^{-2}\kappa_i (s_\ve)q_\alpha ^2
+(\partial_rq_\beta )^2+\varepsilon^{-2}\kappa_j (s_\ve)q_\beta ^2\Big)\ud r
+C_0\int_{-1}^1\big(q_\alpha ^2+q_\beta ^2\big)\ud r.
\end{align}
\end{Proposition}

\begin{proof}
  Assume $j=4$ or $5$. Then $\kappa_j\ge 0$. Thus
  \begin{align*}
 \int_{-1}^1q_\alpha\partial_rq_\beta a(r)\ud r\le &~ \nu_0\int (\partial_rq_\beta)^2 \ud r+C(\nu_0, |a|_{L^\infty}) \int q_\alpha^2 \ud r,\\
 \int_{-1}^1\partial_rq_\alpha q_\beta a(r)\ud r\le & - \int_{-1}^1(q_\alpha \partial_r q_\beta a(r)+ q_\alpha  q_\beta\partial_r a(r))\ud r
+ (q_\alpha q_\beta a)\Big|_{-1}^1\\
\le & ~\nu_0\int_{-1}^1 (\partial_rq_\beta)^2 \ud r+C(\nu_0, \|a\|_{W^{1,\infty}})
\int_{-1}^1 (q_\alpha^2 + q_\beta^2)\ud r+ (q_\alpha q_\beta a)\Big|_{-1}^1.
  \end{align*}
Then the result follows from Lemmas \ref{lem:Linf} and \ref{lem:Linf-endpoint}.
\end{proof}
\smallskip

Now we turn to crossing terms  involving elements in $\mathbb{V}_1$. Assume $\EE_\beta=\EE_1\in\BV_1$.
The following lemma shows that the variation of $\EE_1$ and $\Phi$ along the normal direction $\nabla d^K$ is very small, which is key to bound the crossing terms involving elements in $\mathbb{V}_1$.
\begin{Lemma}\label{lem:E1-smallness}
There exists constant $C_1$ such that
  For $(x,t)\in\Gamma^K(\delta)$, one has that
  \begin{align*}
  |\partial_r \EE_1| \le C_1(d^K(x,t)+\ve),\qquad  |\Phi^{\mathrm{T}}\partial_r\Phi \EE_1| \le C_1(d^K(x,t)+\ve).
  \end{align*}
\end{Lemma}
\begin{proof}
Recalling
$\partial_r=(1+O(\ve^{K+1}))\nabla d^K\cdot\nabla$, we can replace $\partial_r$ by $\nabla d^K\cdot\nabla$. On the other hand,
as $\|d^K-d_0\|_{C^1(\Gamma^K(\delta))}\le C \ve $, one has
  \begin{align*}
  |(\nabla d^K\cdot\nabla) \EE_1|& \le |(\nabla d_0\cdot\nabla) \EE_1| + C_1\ve\\
 & \le  |(\nabla d_0\cdot\nabla) \EE_1|\Big|_{d_0=0}+ C_1(d_0+\ve)=C_1(d_0+\ve).
  \end{align*}
In the last equality, we have used the fact that $\partial_\nu\nn|_{d_0=0}=0$. Similarly, we have
\begin{align*}
|\Phi^{\mathrm{T}}( \nabla d^K\cdot\nabla)\Phi \EE_1|&= |\Phi^{\mathrm{T}}( \nabla d_0\cdot\nabla)\Phi \EE_1| + C_1\ve\\
 & \le  |\Phi^{\mathrm{T}}( \nabla d_0\cdot\nabla)\Phi \EE_1|\Big|_{d_0=0}+ C_1(d_0+\ve)=C_1(d_0+\ve),
\end{align*}
as it holds on $\Gamma_0=\{d_0(x,t)=0\}$ that
\begin{align*}
  (\Phi^{\mathrm{T}}( \nabla d_0\cdot\nabla)\Phi \EE_1)\Big|_{d_0=0}=(\A_-^\mathrm{T}\partial_\nu\A_- \nn\nn)\Big|_{d_0=0}=0.
\end{align*}
due to the boundary condition (\ref{SharpInterfaceSys:Neumann}).
The proof is finished.
\end{proof}

\begin{Proposition}\label{prop:cross-1}Assume that $i$ or $j=1$ and $|a(r)|\le C_{a}(r+\ve)$.
Then for any $\nu_0>0$ there exists $C_0=C_0(\nu_0, C_a, \|a\|_{W^{1,\infty}})>0$ such that
\begin{align}
&\int_{-1}^1\big(\partial_rq_\alpha q_\beta-q_\alpha\partial_rq_\beta\big)a(r)\ud r
\nonumber\\&\leq\nu_0\int_{-1}^1\Big((\partial_rq_\alpha )^2+\varepsilon^{-2}\kappa_i (s_\ve)q_\alpha ^2
+(\partial_rq_\beta )^2+\varepsilon^{-2}\kappa_j (s_\ve)q_\beta ^2\Big)\ud r
+C_0\int_{-1}^1\big(q_\alpha ^2+q_\beta ^2\big)\ud r.\label{ineq:cross-1}
\end{align}
\end{Proposition}

\begin{proof}
Assume $j=1$. Firstly, we have
  \begin{align*}
 \int_{-1}^1\big(\partial_rq_\alpha q_\beta-q_\alpha\partial_rq_\beta\big) a(r)\ud r\le &~
 q_\alpha q_\beta a\Big|_{-1}^1- \int_{-1}^1\big(2q_\alpha\partial_rq_\beta a(r)+q_\alpha q_\beta \partial_ra(r))\ud r.
 \end{align*}
From Lemmas \ref{lem:Linf} and \ref{lem:Linf-endpoint}, it suffices to estimate $\int_{-1}^1q_\alpha\partial_rq_\beta a(r)\ud r$.
Let
\begin{align*}
q_\beta =\mu_0\theta_\ve+\hat{q}_\beta,\text{ with } \int_{-1}^{1}\theta_\ve\hat{q}_\beta \ud r=0.
\end{align*}
Then using the fact that $\partial_r\theta_\ve =\frac{\sqrt{2}}{\ve}\theta_\ve(1-2s_\ve)$, we get
 \begin{align*}
\int_{-1}^1 q_\alpha\partial_rq_\beta a(r)\ud r=&~\mu_0 \int_{-1}^1 q_\alpha\partial_r\theta_\ve a(r)\ud r +\int_{-1}^1q_\alpha\partial_r\hat{q}_\beta a(r)\ud r\\
= &~\mu_0 \sqrt{2}\int_{-1}^1 q_\alpha\theta_\ve(1-2s_\ve) \frac{a(r)}{\ve}\ud r +\int_{-1}^1q_\alpha\partial_r\hat{q}_\beta a(r)\ud r\\
\le &~C\|q_\alpha\|_{L^2{}}\Big(\mu_0\|\ve^{-1}\theta_\ve(1-2s_\ve)a\|_{L^2{}} +\|\partial_r\hat{q}_\beta\|_{L^2{}} \Big).
  \end{align*}
As $a(r)\le C_a(r+\ve)$, we have
\begin{align*}
\mu_0\|\ve^{-1}\theta_\ve(1-2s_\ve)a\|_{L^2{}}\le C_a\mu_0\big(\|\theta_\ve \|_{L^2{}}+\|\theta_\ve r/\ve\|_{L^2{}}\big)\le C(C_a)\mu_0\|\theta_\ve \|_{L^2{}}\le C(C_a) \|q_\beta\|_{L^2{}}.
\end{align*}
Moreover, due to Corollary \ref{lem:lower-bound-1-deriv-out}, we can control $\|\partial_r\hat{q}_\beta\|^2_{L^2{}}$  by the right hand side of \eqref{ineq:cross-1}.
\end{proof}

The estimate of crossing terms for $\mathbb{V}_2$ and $\mathbb{V}_3$ is more subtle.
\begin{Proposition}\label{prop:cross-23}
For any $\nu_0>0$ there exists $C_0=C_0(\nu_0, |a|_{L^\infty})>0$ such that
\begin{align}
&\int_{-1}^1\big(\partial_rq_2q_3-q_2\partial_rq_3\big)a(r)\ud r
\nonumber\\&\le \nu_0\int_{-1}^1\Big((\partial_rq_2)^2+\varepsilon^{-2}\kappa_2(s_\ve)q_2^2
+(\partial_rq_3)^2+\varepsilon^{-2}\kappa_3(s_\ve)q_3^2\Big)\ud r
+C_0\int_{-1}^1\big(q_2^2+q_3^2\big)\ud r.\label{esti:cross-23}
\end{align}
\end{Proposition}
\begin{proof}
We use the decompositions:
\begin{align}
q_2(r)= s_\ve(r)\bar{q}_2(r),\quad q_3(r)=(1- s_\ve(r))\bar{q}_3(r). \label{2Ddecomposing}
\end{align}
Then we have
\begin{align*}
\partial_rq_2q_3-q_2\partial_rq_3=&~\big(\partial_rs_\ve\bar{q}_2+s_\ve \partial_r\bar{q}_2\big) (1- s_\ve)\bar{q}_3
  -s_\ve\bar{q}_2\big[-\partial_rs_\ve\bar{q}_3+(1-s_\ve)\partial_r\bar{q}_2\big]\\
=&~\partial_rs_\ve\bar{q}_2\bar{q}_3+s_\ve (1- s_\ve)\big(\partial_r\bar{q}_2\bar{q}_3-\bar{q}_3\partial_r\bar{q}_2\big).
\end{align*}
Using the next Lemma \ref{lem:cross-estimate-23}, we have
\begin{align*}
\Big|\int_{-1}^1\partial_rs_\ve\bar{q}_2\bar{q}_3 a(r)\ud r\Big|
\le &~
\frac{\nu_0}{2}\int_{-1}^1\Big(s_\ve^2
(\partial_r\bar{q}_2)^2+(1-s_\ve)^2(\partial_r\bar{q}_3)^2\Big)\ud r\\
&+C_0(\nu_0, |a|_{L^\infty})\int_{-1}^1\Big(s_\ve^2\bar{q}_2^2+(1-s_\ve)^2\bar{q}_3^2\Big)\ud r.
\end{align*}
Using Cauchy-Schwartz inequality, one get
\begin{align*}
\Big|\int_{-1}^1s_\ve (1- s_\ve)\big(\partial_r\bar{q}_2\bar{q}_3-\bar{q}_3\partial_r\bar{q}_2\big) a(r)\ud r\Big|
\le&~
\frac{\nu_0}{2}\int_{-1}^1\Big(s_\ve^2
(\partial_r\bar{q}_2)^2+(1-s_\ve)^2(\partial_r\bar{q}_3)^2\Big)\ud r\\
&+C_0(\nu_0, |a|_{L^\infty})\int_{-1}^1\Big(s_\ve^2\bar{q}_2^2+(1-s_\ve)^2\bar{q}_3^2\Big)\ud r.
\end{align*}
Therefore, we obtain
\begin{align*}
  \int_{-1}^1\big(\partial_rq_2q_3-q_2\partial_rq_3\big)a(r)\ud r
  \le &~\nu_0\int_{-1}^1\Big(s_\ve^2
(\partial_r\bar{q}_2)^2+(1-s_\ve)^2(\partial_r\bar{q}_3)^2\Big)\ud r\\
&+C_0(\nu_0, |a|_{L^\infty})\int_{-1}^1\Big(s_\ve^2\bar{q}_2^2+(1-s_\ve)^2\bar{q}_3^2\Big)\ud r.
\end{align*}
Then the claim follows immediately from Lemma \ref{lem:lower-bound-23}.
\end{proof}

\begin{Lemma} \label{lem:cross-estimate-23}
Let $a(r)\in L^\infty([-1,1])$. Then for any $\nu_0>0$ there exists $C_0=C_0(\nu_0, |a|_{L^\infty})>0$ such that
\begin{align}
\bigg|\int_{-1}^1\partial_rs_\ve\varrho_2\varrho_3 a(r) \ud r\bigg|
\leq &C_0\int_{-1}^1\big(s_\ve^2\varrho_2^2+(1-s_\ve)^2\varrho_3^2\big)\ud r
\nonumber\\&+\nu_0\int_{-1}^1\big(s_\ve^2
(\partial_r\varrho_2)^2+(1-s_\ve)^2(\partial_r\varrho_3)^2\big)\ud r.\label{esti:produc-23}
\end{align}
\end{Lemma}
\begin{proof} Let
\begin{align*}
I_1=\int_{-1}^1\big(s_\ve^2\varrho_2^2+(1-s_\ve)^2\varrho_3^2\big)\ud r,\quad
I_2=\int_{-1}^1\big(s_\ve^2
(\partial_r\varrho_2)^2+(1-s_\ve)^2(\partial_r\varrho_3)^2\big)\ud r.
\end{align*}
By the Gagliardo-Nirenberg inequality, one has for any $\nu_1>0$ there exists $C(\nu_1)>0$ such that
\begin{align*}
\|\varrho_2\|_{L^\infty([0,1])}\leq \frac{\nu_1}{2} \bigg(\int_0^1\big(\partial_r\varrho_2\big)^2\ud r\bigg)^{\frac{1}{2}}
+C(\nu_1)\bigg(\int_0^1\varrho_2^2\ud r\bigg)^{\frac{1}{2}}
\leq  \nu_1 I_2^{\frac{1}{2}}+C(\nu_1)I_1^{\frac{1}{2}}.
\end{align*}
Similarly, there holds
\begin{align*}
\|\varrho_3\|_{L^\infty([-1,0])}
\leq  \nu_1 I_2^{\frac{1}{2}}+C(\nu_1)I_1^{\frac{1}{2}}.
\end{align*}
Then
\begin{align}
&\bigg|\int_{-1}^0\partial_rs_\ve\varrho_2\varrho_3a(r)\ud r\bigg|
\nonumber\\&\leq C\Big(\int_{-1}^0\partial_rs_\ve|\varrho_2|\ud r\Big)\big(\nu_1 I_2^{\frac{1}{2}}+C(\nu_1)I_1^{\frac{1}{2}}\big)
\nonumber\\&\leq C\bigg(\int_{-1}^0s_\ve|\partial_r\varrho_2|\ud r+|s_\ve(0)\varrho_2(0)|+|s_\ve(-1)\varrho_2(-1)|\bigg)\big(\nu_1 I_2^{\frac{1}{2}}+C(\nu_1)I_1^{\frac{1}{2}}\big)
\nonumber\\&\leq C\Big(I_2^{\frac{1}{2}}+\nu_1 I_2^{\frac{1}{2}}+C(\nu_1)I_1^{\frac{1}{2}}+|s_\ve(-1)\varrho_2(-1)|\Big)\Big(\nu_1 I_2^{\frac{1}{2}}+C(\nu_1)I_1^{\frac{1}{2}}\Big).\label{esti:produc-23-1}
\end{align}
On the other hand, we have
\begin{align}
|s_\ve(-1)\varrho_2(-1)|&\leq s_\ve(-1)\Big(|\varrho_2(0)|+\int_{-1}^0|\partial_r\varrho_2|\ud r\Big)\nonumber\\
&\leq s_\ve(-1)|\varrho_2(0)|+\int_{-1}^0s_\ve|\partial_r\varrho_2|\ud r
\nonumber\\
&\leq \nu_1 I_2^{\frac{1}{2}}+C(\nu_1)I_1^{\frac{1}{2}}+ I_2^{\frac{1}{2}}.\label{esti:produc-23-2}
\end{align}

With the help of \eqref{esti:produc-23-1}-\eqref{esti:produc-23-2},  one has for any $\nu_0>0$ there exists $C_0>0$ such that
\begin{align*}
\bigg|\int_{-1}^0\partial_rs_\ve\varrho_2\varrho_3 a(r)\ud r\bigg|\leq
\nu_0 I_2+C_0I_1.
\end{align*}
By a similar argument, we get
\begin{align*}
\bigg|\int_{0}^1\partial_rs_\ve\varrho_2\varrho_3 a(r)\ud r\bigg|\leq
\nu_0 I_2+C_0I_1.
\end{align*}
The proof is completed.
\end{proof}

\subsection{Estimate for correction terms}\label{subsec:estimate-correction}

Recall from \eqref{solution:P1} that
\begin{align*}
\PP_1=s_\ve(r)\PP_{1,2}(x,t)+(1-s_\ve)\PP_{1,3}(x,t)+\PP_{1,4}(x,t)-\LL_1(x,t)d_0(x,t)\eta(r/\ve)+\PP_1^*({r}/{\ve},x,t).
\end{align*}
Using the  exponential decay in $z$ of $\PP_1^*(z,x,t)$, and  the fact that $|\PP_1^*(z,x,t)|_{d_0(x,t)=0}=0 $, we have
\begin{align*}
\Big|\frac1{\ve}\PP_1^*\Big(\frac{r}{\ve},x,t\Big)\Big|\le C\frac{|d_0|}{\ve}e^{-\frac{\alpha_0|r|}{\ve}} \le C.
\end{align*}
Thus, the terms containing $\PP_1^*$ can be controlled by $\int |\BB|^2 J\ud r$.
Note that  $\LL_1(x,t)d_0(x,t)\eta(\frac r \ve)\in \mathbb{V}_4$ and $\pa_r\big(\LL_1(x,t)d_0(x,t)\eta(\frac r \ve)\big)$ is bounded.
Thus, without loss of generality, we only need to consider
\begin{align*}
\PP_1(r)=s_\ve(r)\EE_2(r)+(1-s_\ve)\EE_3(r)+\EE_4(r)\quad \text{with}\quad\EE_i\in\mathbb{V}_i(i=2,3,4).
\end{align*}

Now we calculate $\TT_f(\PP_0,\PP_1,\tilde\BB):\tilde\BB$ in (\ref{ineq:nextorder-decom}). First of all, any term in $\int \ve^{-1}\TT_f(\PP_0,\PP_1,\tilde\BB):\tilde\BB Jdr$ containing
$\mathcal{P}_5\tilde\BB$ can be bounded by  the right hand side of  \eqref{ineq:1-dim-nextorder}.
So, we only need to consider the terms in $\mathbb{V}_i(1\le i\le 4)$. Consider $\tilde\BB=\sum_{i=1}^4\BB_i$ with $\BB_i\in\mathbb{V}_i(1\le i\le 4)$.
By Lemma \ref{lem:bilinear}, the summation of all terms containing $\EE_2$ equals to
\begin{align*}
  2s_\ve(2s_\ve-1)\EE_2:(\BB_3\BB_4+\BB_4\BB_3)+2s_\ve(3-4s_\ve)\EE_2:(\BB_1\BB_2+\BB_2\BB_1)\triangleq L_{34}+L_{12},
\end{align*}
the summation of all terms containing $\EE_3$ equals to
\begin{align*}
2(1-s_\ve)(1-4s_\ve)\EE_3:(\BB_1\BB_3+\BB_3\BB_1)+2(1-s_\ve)(1-2s_\ve)\EE_3:(\BB_2\BB_4+\BB_4\BB_2)\triangleq L_{13}+L_{24},
 \end{align*}
and the summation of all terms containing $\EE_4$ equals to
\begin{align*}
2(1-2s_\ve)\EE_4:(\BB_2\BB_3+\BB_3\BB_2)\triangleq L_{23}.
\end{align*}

Now using the decomposition (\ref{esti:decomp-new}), we obtain
\begin{align*}
  \BB_i=\sum_{\alpha\in\Lambda_i}q_\alpha\EE_\alpha,\qquad i=1,2,3,4.
\end{align*}
We use $\EE_i$ to denote an element in $\mathbb{V}_i(1\le i\le 4)$. Then
the integral  $\frac1\ve\int L_{34}$  can be written as a summation of terms with form
\begin{align}
\frac1\ve\int s_\ve(2s_\ve-1)q_3(r)q_4(r) \QQ_2(r):(\EE_2\EE_3+\EE_3\EE_2)(r)\ud r.
\end{align}
Letting $q_3=(1-s_\ve) \varrho_3$, $q_4=\varrho_4$, and using  Lemma \ref{lem:cross-estimate-34} below and Lemma \ref{lem:lower-bound-23}, we can obtain
\begin{align}\nonumber
&\frac1\ve\int_{-1}^1 s_\ve(2s_\ve-1)q_3(r)q_4(r) \QQ_2(r):(\EE_2\EE_3+\EE_3\EE_2)(r)\ud r\\
&\le
\nu_0  \int_{-1}^1\Big((\partial_rq_3)^2
  +\varepsilon^{-2}\kappa_3(s_\ve)q_3^2+(\partial_rq_4)^2
  +\varepsilon^{-2}\kappa_4(s_\ve)q_4^2\Big)\ud r+C(\nu_0)\int_{-1}^1(q_3^2+q_4^2)\ud r.
\end{align}
Similarly, by using Lemmas  \ref{lem:cross-estimate-24}-\ref{lem:cross-estimate-13} below, Lemma \ref{lem:cross-estimate-23} with Lemmas \ref{lem:lower-bound-23} and \ref{lem:lower-bound-1} , we can bound
\begin{align*}
\frac1\ve\int L_{24}, \quad \frac1\ve\int L_{12}, \quad \frac1\ve\int L_{13},\quad \frac1\ve\int L_{23},
\end{align*}
by the right hand side of (\ref{ineq:1-dim-cross-final-2}).
Then (\ref{ineq:1-dim-nextorder})  follows easily.

\begin{Lemma} \label{lem:cross-estimate-34}
For any $\nu_0>0$ there exists $C_0=C_0(\nu_0, \|a\|_{W^{1,\infty}})>0$ such that
\begin{align}
\frac1\ve\bigg|\int_{-1}^1s_\ve(1-s_\ve)(2s_\ve-1)\varrho_3\varrho_4 a(r) \ud r\bigg|
\leq &~C_0\int_{-1}^1\big((1-s_\ve)^2\varrho_3^2+\varrho_4^2\big)\ud r
\nonumber\\&+\nu_0\int_{-1}^1\big(
(1-s_\ve)^2(\partial_r\varrho_3)^2+(\partial_r\varrho_4)^2\big)\ud r.\label{ineq:cross-estimate-34}
\end{align}
\end{Lemma}
\begin{proof}
  The left side can be written as
  \begin{align*}
&\bigg|\int_{-1}^1\partial_r\big[s_\ve(1- s_\ve)\big]\varrho_3\varrho_4 a(r) \ud r\bigg| \\
&\le   \bigg|\int_{-1}^1s_\ve(1- s_\ve)\big[\partial_r\varrho_3\varrho_4+\varrho_3 \partial_r\varrho_4 a(r)+\varrho_3\varrho_4 \partial_ra(r)\big] \ud r\bigg|+\Big(s_\ve(1- s_\ve)\varrho_3\varrho_4 a(r)\Big)\Big|_{-1}^1.
  \end{align*}
Then the claim follows from  Cauchy-Schwartz inequality and  Lemma \ref{lem:Linf}.
\end{proof}

\begin{Lemma} \label{lem:cross-estimate-24}
For any $\nu_0>0$ there exists $C_0=C_0(\nu_0, \|a\|_{W^{1,\infty}})>0$ such that
\begin{align}
\frac1\ve\bigg|\int_{-1}^1(1-s_\ve)(1-2s_\ve)s_\ve\varrho_2\varrho_4 a(r) \ud r\bigg|
\leq &~C_0\int_{-1}^1\big(s_\ve^2\varrho_2^2+\varrho_4^2\big)\ud r
\nonumber\\&+\nu_0\int_{-1}^1\big(s_\ve^2
(\partial_r\varrho_2)^2+(\partial_r\varrho_4)^2\big)\ud r.\label{ineq:cross-estimate-24}
\end{align}
\end{Lemma}
\begin{proof}
  The proof is similar to Lemma \ref{lem:cross-estimate-34}. We omit the details.
\end{proof}

\begin{Lemma} \label{lem:cross-estimate-12}
For any $\nu_0>0$ there exists $C_0=C_0(\nu_0, \|a\|_{W^{1,\infty}})>0$ such that
\begin{align}
\frac1\ve\bigg|\int_{-1}^1s_\ve^2(3-4s_\ve)\theta_\ve \varrho_1\varrho_2 a(r) \ud r\bigg|
\leq &~C_0\int_{-1}^1\big(\theta_\ve^2\varrho_1^2+ s_\ve^2\varrho_2^2\big)\ud r
\nonumber\\&+\nu_0\int_{-1}^1\big(\theta_\ve^2
(\partial_r\varrho_1)^2+ s_\ve^2(\partial_r\varrho_2)^2\big)\ud r.\label{ineq:cross-estimate-12}
\end{align}
\end{Lemma}
\begin{proof}
  The left side can be written as
  \begin{align*}
&\bigg|\int_{-1}^1\partial_r\big[s_\ve^3(1- s_\ve)\big]\varrho_1\varrho_2 a(r) \ud r\bigg| =
\bigg|\int_{-1}^1\partial_r\big(s_\ve^2\theta_\ve\big)\varrho_1\varrho_2 a(r) \ud r\bigg| \\
&\le   \bigg|\int_{-1}^1s_\ve^2\theta_\ve\big[\partial_r\varrho_1\varrho_2+\varrho_1 \partial_r\varrho_2 a(r)+\varrho_1\varrho_2 \partial_ra(r)\big] \ud r\bigg|+\Big(s_\ve^2\theta_\ve\varrho_1\varrho_2 a(r)\Big)\Big|_{-1}^1.
  \end{align*}
Then the claim follows from  Cauchy-Schwartz inequality and Lemmas \ref{lem:Linf} and \ref{lem:Linf-endpoint}.
\end{proof}

\begin{Lemma} \label{lem:cross-estimate-13}
For any $\nu_0>0$ there exists $C_0=C_0(\nu_0, \|a\|_{W^{1,\infty}})>0$ such that
\begin{align}
\frac1\ve\bigg|\int_{-1}^1(1-s_\ve)^2(1-4s_\ve)\theta_\ve \varrho_1\varrho_3 a(r) \ud r\bigg|
\leq &~C_0\int_{-1}^1\big(\theta_\ve^2\varrho_1^2+(1-s_\ve)^2\varrho_3^2\big)\ud r
\nonumber\\&+\nu_0\int_{-1}^1\big(\theta_\ve^2
(\partial_r\varrho_1)^2+(1-s_\ve)^2(\partial_r\varrho_3)^2\big)\ud r.\label{ineq:cross-estimate-13}
\end{align}
\end{Lemma}
\begin{proof}
 As $(1-s_\ve)^2(1-4s_\ve)\theta_\ve/\ve=\partial_r((1-s_\ve)^3s_\ve)$, the proof is similar to Lemma \ref{lem:cross-estimate-12}. We omit the details.
\end{proof}
\begin{Remark}
The proof of the above lemmas relies heavily on the facts that, all the weights can be written as derivatives of some good functions.
These functions have factors which consist of production of corresponding eigenfunctions, and thus enable us to use integrating by parts to remove the singularities. The behinded mechanism of such coincidence is the cubic null cancellation Lemma \ref{lem:inner-null}.
\end{Remark}

\section{Uniform error estimates}\label{sec:remainder}

Let $\A^\ve$ be a solution to \eqref{eq:main-intro} and $\A^K$ be the approximate solution construct in Section \ref{sec:construct}. Define
\begin{align*}
\mathbf{\Psi}=\frac{1}{\ve^L}(\A^\ve-\A^K).
\end{align*}
Then we have
\begin{align}
\partial_t\mathbf{\Psi}=\Delta\mathbf{\Psi}-\ve^{-2}\mathcal{H}_{\A^{K}}\Bp
-\frac{\ve^{L-2}}{2}\TT_f(\A^{K},\Bp,\Bp)+\ve^{2L-2}\Bp\Bp^{\mathrm{T}}\Bp-\mathfrak{R}^{\ve},\label{equation:remainder}
\end{align}
where $\mathfrak{R}^{\ve}$ is independent of $\Bp$ which satisfies $\partial^i\mathfrak{R}^{\ve}=O(\ve^{K-L-i-1})$ for $i\ge0$. We choose $L=3([\frac{m}{2}]+1)+3$, and $K\ge L+1$.

Let
\begin{align*}
\bar{\mathcal{E}}(\Bp)= \sum_{i=0}^{[\frac{m}{2}]+1}\ve^{6i}\int_\Omega\|\partial^i\Bp\|^2 \ud x.
\end{align*}
Then one has
\begin{align*}
\ve^{L-2}\|\Bp\|_{L^\infty}\le \ve^{3([\frac{m}{2}]+1)+1}\|\Bp\|_{L^\infty} \le C \ve \bar{\mathcal{E}}(\Bp)^{\frac{1}{2}}.
\end{align*}

With the help of  Theorem \ref{thm:main2}, standard energy estimates yield that
\begin{align*}
  \frac{\ud}{\ud t}\int_\Omega\|\Bp\|^2 \ud x\le C (1+\ve \bar{\mathcal{E}}(\Bp)^{\frac{1}{2}}+\ve^2\bar{\mathcal{E}}(\Bp))\int_\Omega\|\Bp\|^2 \ud x +C\le C\big(1+\bar{\mathcal{E}}(\Bp)+\ve^2\bar{\mathcal{E}}(\Bp)^2\big).
\end{align*}
Applying $\partial^i(0\le i\le [\frac{m}{2}]+1)$ on the equation (\ref{equation:remainder}), we get
\begin{align*}\nonumber
\partial_t\partial^i\mathbf{\Psi}&=\Delta\partial^i\mathbf{\Psi}-\ve^{-2}\mathcal{H}_{\A^{K}}\partial^i\Bp
+\ve^{-2}[ \mathcal{H}_{\A^{K}}, \partial^i]\Bp \\
&\quad-\frac12\ve^{L-2}\partial^i\TT_f(\A^{K},\Bp,\Bp)+\ve^{2L-2}\partial^i(\Bp\Bp^{\mathrm{T}}\Bp)-\partial^i\mathfrak{R}^{\ve},
\end{align*}
where
\begin{align*}
\ve^{-2}[ \mathcal{H}_{\A^{K}},
 \partial^i]\Bp=\ve^{-2}\sum_{j+l+k=i, k\le i-1,j\le l} \TT_f(\partial^j\A^K, \partial^l\A^K, \partial^k \Bp),
\end{align*}
whose $L^2$-norm can be bounded by
\begin{align*}
  C\ve^{-2-j-l}\|\partial^k \Bp\|_{L^2}\le \ve^{-3k-j-l-2}\bar{\mathcal{E}}^{1/2}(\Bp)\le \ve^{-3i}\bar{\mathcal{E}}^{1/2}(\Bp).
\end{align*}
Moreover
\begin{align*}
  \|\ve^{L-2}\partial^i\TT_f(\A^{K},\Bp,\Bp)\|_{L^2}&= \|\ve^{L-2}\sum_{j+k+l=i,k\le l}\TT_f(\partial^j\A^{K},\partial^k\Bp,\partial^l\Bp)\|_{L^2}\\
  &\le   C\ve^{1-3i}\bar{\mathcal{E}}(\Bp), \\
  \|\ve^{2L-2}\partial^i(\Bp\Bp^{\mathrm{T}}\Bp)\|_{L^2}&\le \ve^{2L-2}\|\Bp \|_{H^{[\frac{m}{2}]+1}}^2\|\Bp\|_{H^i}\le C\ve^{2-3i}\bar{\mathcal{E}}^{3/2}(\Bp),\\
 \| \partial^i\mathfrak{R}_k^{\ve}\|_{L^2} &\le \ve^{K-L-i-1}.
\end{align*}
Therefore, we have
\begin{align*}
  \frac{\ud}{\ud t}\int_\Omega \ve^{6i}\|\partial^i\Bp\|^2 \ud x
  &\le C (1+\bar{\mathcal{E}}(\Bp)^{\frac{1}{2}}+\ve\bar{\mathcal{E}}(\Bp)+\ve^2\bar{\mathcal{E}}^{3/2}(\Bp))\|\partial^i\Bp\|_{L^2}\\
  &\le C(1+\bar{\mathcal{E}}(\Bp)+\ve^2\bar{\mathcal{E}}(\Bp)^2).
\end{align*}
Summing $i$ from $0$ to $[\frac{m}{2}]+1$, we get
\begin{align*}
  \frac{\ud}{\ud t}\bar{\mathcal{E}}(\Bp)\le C(1+\bar{\mathcal{E}}(\Bp)+\ve^2 \bar{\mathcal{E}}(\Bp)^2).
\end{align*}
Then Theorem \ref{thm:main3} can be concluded by a direct continuous argument.

\appendix
\section{}

\subsection{A key formula of trilinear form $\TT_f$}
We present a lemma, which was used in the proof of Lemma \ref{lem:inner-null} and the estimate for correction terms in Section \ref{subsec:estimate-correction}. Recall that the trilinear form $\TT_f$ is defined by
\beno
\TT_f(\A_1,\A_2,\A_3)=(\A_1\A_2^{\mathrm{T}}+\A_2\A_1^{\mathrm{T}})\A_3
+\A_3(\A_1^{\mathrm{T}}\A_2+\A_2^{\mathrm{T}}\A_1)+(\A_1\A_3^{\mathrm{T}}\A_2+\A_2\A_3^{\mathrm{T}}\A_1).
\eeno
\begin{Lemma}\label{lem:bilinear}
Let $\PP_1=s\EE_2+(1-s)\EE_3+\EE_4$  and $\BB=\sum_{i=1}^4\BB_i$ with $\EE_i,\BB_i\in\mathbb{V}_i$. Then
\begin{align}\nonumber
\TT_f(\PP_0,\PP_1,\BB):\BB = & ~2s\EE_2:\big[(2s-1)(\BB_3\BB_4+\BB_4\BB_3)+(3-4s)(\BB_1\BB_2+\BB_2\BB_1)\big]\\ \nonumber
&+2(1-s)\EE_3:\big[(1-4s)(\BB_1\BB_3+\BB_3\BB_1)+(1-2s)(\BB_2\BB_4+\BB_4\BB_2)\big]\\ \label{eq:bilinear}
&+2(1-2s)\EE_4:(\BB_2\BB_3+\BB_3\BB_2).
\end{align}
\end{Lemma}
\begin{proof}
From a direct calculation, we have
\begin{align}\nonumber
&\TT_f(\PP_0,\PP_1,\BB):\BB\\
&=(\PP_0\PP_1^{\mathrm{T}}+\PP_1\PP_0^{\mathrm{T}}):(\BB\BB^{\mathrm{T}})+(\PP_0^{\mathrm{T}}\PP_1+\PP_1^{\mathrm{T}}\PP_0):(\BB^{\mathrm{T}}\BB)
+2(\BB^{\mathrm{T}}\PP_1):(\PP_0\BB).\label{formula:BB1}
\end{align}
Moreover, one has
\begin{align*}
\PP_0\PP_1^{\mathrm{T}}+\PP_1\PP_0^{\mathrm{T}}&=(1-s)(\PP_1^{\mathrm{T}}+\PP_1)+s\Big((\II-2\nn\nn)\PP_1^{\mathrm{T}}+\PP_1(\II-2\nn\nn)\Big)\\
&=2s(1-s)\big(\EE_2-(\II-2\nn\nn)\EE_3\big).
\end{align*}
Here, we remark that  $(\II-2\nn\nn)\EE_3\in\mathbb{V}_2$ is symmetric. Similarly, we have
\begin{align*}
  \PP_0^{\mathrm{T}}\PP_1+\PP_1^{\mathrm{T}}\PP_0=2s(1-s)\big(\EE_2+(\II-2\nn\nn)\EE_3\big).
\end{align*}
Therefore, if we let $\BB=\DD+\WW$ with $\DD$ symmetric and $\WW$ antisymmetric, the first two terms in \eqref{formula:BB1} read as
\begin{align*}
  4s(1-s)\Big(\EE_2:(\DD^2-\WW^2)+(\II-2\nn\nn)\EE_3:(\DD\WW-\WW\DD)\Big).
\end{align*}
The last term in \eqref{formula:BB1} equals to
\begin{align*}
2\PP_1:(\BB\PP_0\BB)&= 2\PP_1:(\DD\PP_0\DD+\DD\PP_0\WW+\WW\PP_0\DD+\WW\PP_0\WW)\\
&=2s\EE_2:(\DD\PP_0\DD+\WW\PP_0\WW)+2((1-s)\EE_3+\EE_4):(\DD\PP_0\WW+\WW\PP_0\DD).
\end{align*}

As $\DD=\BB_1+\BB_2$, $\WW=\BB_3+\BB_4$, we have
\begin{align*}
  \DD^2&=\BB_1^2+\BB_1\BB_2+\BB_2\BB_1+\BB_2^2,\\
  \WW^2&=\BB_3+\BB_4\BB_3+\BB_3\BB_4+\BB_4^2,\\
\DD\WW-\WW\DD &=(\BB_1\BB_3-\BB_3\BB_1)+(\BB_2\BB_3-\BB_3\BB_2) +\BB_2\BB_4-\BB_4\BB_2,\\
  \DD\PP_0\DD&=(\BB_1+\BB_2)(\II-2s\nn\nn)(\BB_1+\BB_2)\\
  &=(1-2s)\BB_1^2+(1-2s)(\BB_1\BB_2+\BB_2\BB_1)+\BB_2^2-2s\BB_2\nn\nn\BB_2,\\
  \WW\PP_0\WW&=(\BB_3+\BB_4)(\II-2s\nn\nn)(\BB_3+\BB_4)\\
  &=\BB_3^2-2s\BB_3\nn\nn\BB_3+\BB_4\BB_3+\BB_3\BB_4+\BB_4^2,\\
\DD\PP_0\WW+\WW\PP_0\DD&=(\BB_1+\BB_2)(\II-2s\nn\nn)(\BB_3+\BB_4)+(\BB_3+\BB_4)(\II-2s\nn\nn)(\BB_1+\BB_2)\\
  &=(1-2s)(\BB_1\BB_3+\BB_3\BB_1+\BB_2\BB_3+\BB_3\BB_2) +\BB_2\BB_4+\BB_4\BB_2,
  \end{align*}
Then (\ref{eq:bilinear}) follows directly as we have
\begin{align*}
 & \EE_2:\BB_i^2=0, \text{ for }i=1,2,3,4; \\ &\EE_3:(\BB_2\BB_3+\BB_3\BB_2)=0,\quad  \EE_3:\big[(\II-2\nn\nn)(\BB_2\BB_3+\BB_3\BB_2)\big]=0;\\
 & \EE_4:\BB_1\BB_3=\QQ_4:\BB_1\BB_4=\QQ_4:\BB_2\BB_4=0.
\end{align*}
\end{proof}

\subsection{Solvability of ODEs}
We collect some results on solving the ODEs:
\begin{align}
\mathcal{L}_is_{i}(z,x,t)=f_i(z,x,t)\quad  (i=1,2,\cdots, 5), \label{ODE:i}
\end{align}
 in $\mathbb{R}$ for fixed $(x,t)\in\Gamma(\delta)$, which have been proved in \cite{ABC} and \cite{FWZZ18}.

\begin{Lemma}\label{lem:s1}(Solvability of (\ref{ODE:i})$(i=1)$)
If the following decay conditions
 \begin{align*}
&\big|\partial_z^j\partial_x^l\partial_t^m\big(f_1(z,x,t)-f_{1}^{+}(x,t)\big)\leq C|z|^{k}\big(1-s(z)\big),\ \ \text{for}\ \ z\to +\infty ,
\\&\big|\partial_z^j\partial_x^l\partial_t^m \big(f_1(z,x,t)-f_{1}^{-}(x,t)\big)\big|\leq
C|z|^{k}s(z), \quad\quad \quad \text{for}  \ z\to-\infty
\end{align*}
with  some $k\in\mathbb{N}$  and the compatibility condition
 \begin{align*}
\int_{-\infty}^{+\infty}f_1(z,x,t)s'(z)dz =0
\end{align*}
hold, then (\ref{ODE:i})$(i=1)$ has a  unique bounded solution such that $s_{1}(0,x,t)=1$ and
 \begin{align*}
& \big|\partial^j_z\partial^l_x\partial^m_t\big(s_{1}(z,x,t)- \frac12{f_{1}^{+}(x,t)}\big)\big|\leq C|z|^{k+1}\big(1-s(z)\big),\ \ \ \ \ \text{for}\ \ z\to +\infty,
\\& \big|\partial^j_z\partial^l_x\partial^m_t\big(s_{1}(z,x,t)- \frac12{f_{1}^{-}(x,t)}\big)\big|\leq
C|z|^{k+1}s(z),\ \qquad\quad\ \ \ \text{for}\ z\to-\infty,
 \end{align*}
where $j,l,m=0,1,\cdots,$ $C>0$  is independent of $z,x,t$.
More concretely, we have
\begin{align*}
s_{1}(z,x,t)&=\frac{s'(z)}{s'(0)}+s'(z)\int_{0}^{z}(s'(\varsigma))^{-2}\int_\varsigma^{+\infty}f_1(\tau,x,t)s'(\tau)d\tau d\varsigma.
\end{align*}
\end{Lemma}

\begin{Lemma}\label{lem:s2} (Solvability of (\ref{ODE:i})$(i=2)$) If the following decay conditions
 \begin{align*}
&\big|\partial_z^j\partial_x^l\partial_t^mf_2(z,x,t)\big|\leq C|z|^{k}\big(1-s(z)\big),\ \ \qquad\qquad\text{for}\ \ z\to +\infty,
\\&\big|\partial_z^j\partial_x^l\partial_t^m \big(f_2(z,x,t)-f_{2}^{-}(x,t)\big)\big|\leq
C|z|^{k}s(z), \ \quad \ \text{for}\ z\to-\infty
\end{align*}
with  some $k\in\mathbb{N}$  and the compatibility condition
  \begin{align*}
\int_{-\infty}^{+\infty}f_2(z,x,t)s(z)dz =0
\end{align*}
hold, then for any given $s_2^{+}(x,t)$,  (\ref{ODE:i})$(i=2)$ has a unique bounded solution such that
 \begin{align*}
&\big|\partial^j_z\partial^l_x\partial^m_t\big(s_{2}(z,x,t)-s_2^{+}(x,t)\big)\big|\leq C|z|^{k+1}\big(1-s(z)\big),\ \ \ \ \text{for}\ \ z\to +\infty,
\\&\big|\partial^j_z\partial^l_x\partial^m_t\big(s_{2}(z,x,t)-\frac{1}{2}f_{2}^{-}(x,t)\big)\big|\leq C|z|^{k+1}s(z),\ \ \qquad\ \text{for}\ z\to-\infty,
 \end{align*}
where $j,l,m=0,1,\cdots,$ $C>0$  is independent of $z,x,t$. More concretely, we have
\begin{align*}
s_{2}(z,x,t)&=s(z)s_2^{+}(x,t)-s(z)\int_z^{+\infty}(s(\varsigma))^{-2}\int_\varsigma^{+\infty}f_2(\tau,x,t)s(\tau)d\tau d\varsigma
.
\end{align*}
\end{Lemma}

\begin{Lemma}\label{lem:s3} (Solvability of (\ref{ODE:i})$(i=3)$) If the following decay conditions
 \begin{align*}
&\big|\partial_z^j\partial_x^l\partial_t^m\big(f_3(z,x,t)-f_{3}^{+}(x,t)\big)\big|\leq C|z|^{k}\big(1-s(z)\big),\ \ \qquad\text{for}\ \ z\to +\infty,
\\&\big|\partial_z^j\partial_x^l\partial_t^m f_3(z,x,t)\big|\leq
C|z|^{k}s(z),\ \  \ \qquad \ \text{for}\ z\to-\infty
\end{align*}
with  some $k\in\mathbb{N}$  and the compatibility condition
  \begin{align*}
\int_{-\infty}^{+\infty}f_3(z)(1-s(z))dz =0
\end{align*}
hold, then for any given $s_3^{-}(x,t)$, (\ref{ODE:i})$(i=3)$ has a unique bounded solution such that
 \begin{align*}
&\big|\partial^j_z\partial^l_x\partial^m_t\big(s_{3}(z,x,t)-\frac{1}{2}f_{3}^{+}(x,t)
\big)\big|\leq C|z|^{k+1}\big(1-s(z)\big),\ \ \ \ \text{for}\ \ z\to +\infty,
\\&\big|\partial^j_z\partial^l_x\partial^m_t\big(s_{3}(z,x,t)-s_3^{-}(x,t)\big)\big|\leq C|z|^{k+1}s(z),\ \ \ \ \ \ \ \qquad\ \text{for}\ z\to-\infty,
 \end{align*}
where $j,l,m=0,1,\cdots,$ $C>0$  is independent of $z,x,t$. More concretely, we have
\begin{align*}
s_{3}(z,x,t)&=(1-s(z))s_3^{-}(x,t)-(1-s(z))\int_{-\infty}^z(1-s(\varsigma))^{-2}\int_{-\infty}^\varsigma f_3(\tau,x,t)(1-s(\tau))d\tau d\varsigma
.
\end{align*}
\end{Lemma}

\begin{Lemma}\label{lem:s4} (Solvability of (\ref{ODE:i})$(i=4)$) If the following decay conditions
 \begin{align*}
&\big|\partial_z^j\partial_x^l\partial_t^mf_4(z,x,t)\big|\leq C|z|^{k}\big(1-s(z)\big),\ \qquad\text{for}\ \ z\to +\infty,
\\&\big|\partial_z^j\partial_x^l\partial_t^m f_4(z,x,t)\big|\leq
C|z|^{k}s(z),\ \  \ \ \quad\qquad \ \text{for}\ z\to-\infty
\end{align*}
with  some $k\in\mathbb{N}$  and the compatibility condition
  \begin{align*}
\int_{-\infty}^{+\infty}f_4(z,x,t)dz =0
\end{align*}
hold, then for any given $s_4^{+}(x,t)$,  (\ref{ODE:i})$(i=4)$ has a  unique bounded solution such that
 \begin{align*}
&\big|\partial^j_z\partial^l_x\partial^m_t\big(s_{4}(z,x,t)-s_4^{+}(x,t)\big)\big|\leq C|z|^{k+1}\big|1-s(z)\big|,\ \ \ \ \text{for}\ \ z\to +\infty,
\\&\big|\partial^j_z\partial^l_x\partial^m_ts_{4}(z,x,t)\big|\leq C|z|^{k+1}s(z),\ \ \ \ \ \ \ \qquad\ \text{for}\ z\to-\infty,
 \end{align*}
where $j,l,m=0,1,\cdots,$ $C>0$  is independent of $z,x,t$. More concretely, we have
\begin{align*}
s_{4}(z,x,t)&=s_4^{+}(x,t)-s(z)\int_z^{+\infty}(s(\varsigma))^{-2}\int_\varsigma^{+\infty}f_4(\tau,x,t)s(\tau)d\tau d\varsigma
.
\end{align*}
\end{Lemma}

\begin{Lemma}\label{lem:s5}(Solvability of (\ref{ODE:i})$(i=5)$) If  the following decay conditions
 \begin{align*}
&\big|\partial_z^j\partial_x^l\partial_t^m\big(f_5(z,x,t)-f_{5}^{+}(x,t)\big)\big|\leq C|z|^{k}\big(1-s(z)\big),\ \ \text{for}\ \ z\to +\infty,
\\&\big|\partial_z^j\partial_x^l\partial_t^m \big(f_5(z,x,t)-f_{5}^{-}(x,t)\big)\big|\leq
C|z|^{k}s(z), \quad\quad\quad  \text{for}\ z\to-\infty
\end{align*}
with some $k\in\mathbb{N}$ hold, then (\ref{ODE:i})$(i=4)$ has a unique bounded solution such that
 \begin{align*}
& \big|\partial^j_z\partial^l_x\partial^m_t\big(s_{5}(z,x,t)- \frac12{f_{5}^{+}(x,t)}\big)\big|\leq C|z|^{k+1}\big(1-s(z)\big),\ \  \text{for}\ \ z\to +\infty, 
\\& \big|\partial^j_z\partial^l_x\partial^m_t\big(s_{5}(z,x,t)- \frac12{f_{5}^{-}(x,t)}\big)\big|\leq C|z|^{k+1}s(z),\quad\quad\quad \text{for}\ z\to-\infty,
 \end{align*}
where $j,l,m=0,1,\cdots $ and   $C>0$  is independent of $z,x,t.$
\end{Lemma}

\subsection{A sketch proof of  Lemma \ref{lem:mini-pair}}\label{App:proof-lem-minipair}
We define the Euclidean distance between a matrix and a compact manifold $\Sigma\subset\mathbb{M}_n$:
\begin{align*}
  \rho(\A, \Sigma)=\min_{\B\in\Sigma } \|\A-\B\|=\sqrt{\|\A-\B\|_2^2}.
\end{align*}
The following lemma is a straightforward consequence of elementary calculations.
\begin{Lemma}\label{lem:appendix}
If $\A\in\mathbb{M}_n$, and $\rho:= \rho(\A, {O}(n))\le 1$, then $F(\A)\ge F(\II-\rho\nn\otimes\nn)=\rho^2(2-\rho)^2$.
\end{Lemma}
For any curve $\gamma=\{\B(z): a< z< b\}$ in $\mathbb{M}_n$, the quantity
\begin{align*}
e(\gamma)=  \int_{\gamma}\sqrt{F(\BB(z))/2} \|\BB'(z)\|\ud z
\end{align*}
is independent of the parametrization of $\gamma$.
If $\A(z)(z\in\BR)$ is a minimal connecting orbit, then
\begin{align*}
  e(\mathrm{Traj}(\A)) =\min_{\gamma(\pm1)\in\Sigma_\pm } e(\gamma).
\end{align*}
Then  using Lemma \ref{lem:appendix} and the arguments in \cite[Theorem 2.1]{LPW}, we can obtain
\begin{align*}
  \mathrm{Traj}(\A)=\{\tau \A_++(1-\tau)\A_-:\text{for some }\A_-\in{O}^-(n), \A_+\in{O}^+(n), \|\A_--\A_+\|=2\},
\end{align*}
which implies the claim.

\subsection{Formal derivation of Neumann jump condition}\label{App:BC}
 Assume that
\begin{align*}
\A^\ve \rightarrow \A_\pm \text{ strongly in }L^2(\Omega_t^\pm);\quad \partial_t\A^\ve  \rightharpoonup \partial_t\A_\pm, \nabla\A^\ve  \rightharpoonup \nabla\A_\pm \text{ weakly in }L^2(\Omega_t^\pm).
\end{align*}
The equation (\ref{eq:main-intro}) yields that
\begin{align*}
  (\A^{\ve})^{\mathrm{T}} (\partial_t\A^{\ve}-\Delta \A^{\ve})-(\partial_t(\A^{\ve})^{\mathrm{T}} -\Delta (\A^{\ve})^{\mathrm{T}} ) \A^{\ve}=0,
\end{align*}
which gives
\begin{align*}
  (\A^{\ve})^{\mathrm{T}} \partial_t\A^{\ve}- \partial_t(\A^{\ve})^{\mathrm{T}} \A^{\ve}
= \nabla\cdot ((\A^{\ve})^{\mathrm{T}} \nabla \A^{\ve} -\nabla (\A^{\ve})^{\mathrm{T}}  \A^{\ve}).
\end{align*}
Testing the above equation with a smooth matrix-valued function $\Psi$ and taking the limit $\ve\to 0$, one gets
\begin{multline*}
\int_{\Omega_t^+}\Big( (\A_+^{\mathrm{T}}\partial_t\A_+- \partial_t\A_+^{\mathrm{T}} \A_+)\Psi
+ (\A_+^{\mathrm{T}}\nabla \A_+ -\nabla \A_+^{\mathrm{T}}  \A_+)\cdot \nabla\Psi\Big)\ud x\\
+\int_{\Omega_t^-}\Big( (\A_-^{\mathrm{T}}\partial_t\A_-- \partial_t\A_-^{\mathrm{T}} \A_-)\Psi
+ (\A_-^{\mathrm{T}}\nabla \A_- -\nabla \A_-^{\mathrm{T}}  \A_-)\cdot \nabla\Psi\Big)\ud x=0.
\end{multline*}
As $\A_\pm$ obeys the harmonic map heat flow to ${O}^\pm(n)$ on $\Omega_\pm$, one immediately gets the boundary condition
\begin{align*}
\A_+^{\mathrm{T}}\partial_\nu \A_+ -\partial_\nu \A_+^{\mathrm{T}}  \A_+=
\A_-^{\mathrm{T}}\partial_\nu \A_- -\partial_\nu \A_-^{\mathrm{T}}  \A_-.
\end{align*}
Similarly, we have
\begin{align*}
\A_+\partial_\nu \A_+^{\mathrm{T}} -\partial_\nu \A_+  \A_+^{\mathrm{T}}=
\A_-\partial_\nu \A_-^{\mathrm{T}} -\partial_\nu \A_-  \A_-^{\mathrm{T}}.
\end{align*}
These two relations give (\ref{SharpInterfaceSys:Neumann}) under the minimal pair relation (\ref{SharpInterfaceSys:MiniPair}), see Appendix \ref{App:BC}.

Assume that $(\A_-,\A_+)$ is a minimal pair, i.e., $\A_+=\A_-(\II-2\nn\nn)$ for some $\nn\in{S}^{n-1}$.
We prove the following three boundary conditions are equivalent:
\begin{align*}
(i):\quad&\left \{
\begin{array}{ll}
  \A_+^{\mathrm{T}}\partial_\nu \A_+ -\partial_\nu \A_+^{\mathrm{T}}  \A_+=
\A_-^{\mathrm{T}}\partial_\nu \A_- -\partial_\nu \A_-^{\mathrm{T}}  \A_-,\\
\A_+\partial_\nu \A_+^{\mathrm{T}} -\partial_\nu \A_+  \A_+^{\mathrm{T}}=
\A_-\partial_\nu \A_-^{\mathrm{T}} -\partial_\nu \A_-  \A_-^{\mathrm{T}};
\end{array}
\right.\\
(ii):\quad& \partial_\nu \A_+=\partial_\nu \A_-;\\
(iii):\quad& \A_+^{\mathrm{T}}{\partial_\nu \A_+}=\A_-^{\mathrm{T}}{\partial_\nu \A_-}=\WW,\quad \text{for some }\WW\in\mathbb{V}_4.
\end{align*}

{($i$) $\Rightarrow$ ($iii$)}:
As $\A_+^{\mathrm{T}}\partial_\nu \A_+$ and $\A_-^{\mathrm{T}}\partial_\nu \A_-$ are both antisymmetric, the first equation gives
\begin{align*}
 \A_+^{\mathrm{T}}\partial_\nu \A_+=\A_-^{\mathrm{T}}\partial_\nu \A_-=\WW\in\mathbb{A}_n.
\end{align*}
Similarly, we have $\partial_\nu \A_+  \A_+^{\mathrm{T}}=\partial_\nu \A_-  \A_-^{\mathrm{T}}$.
Therefore, we get
\begin{align*}
\A_+  \WW\A_+^{\mathrm{T}}= \partial_\nu \A_+\A_+^{\mathrm{T}}=\partial_\nu \A_-  \A_-^{\mathrm{T}}=\A_- \WW\A_-^{\mathrm{T}}.
\end{align*}
which implies $ (\II-2\nn\nn)\WW(\II-2\nn\nn)=\WW$. Thus, $\nn\WW=0$, i.e. $\WW\in\mathbb{V}_4$.

{($ii$) $\Rightarrow$ ($iii$)}: Assume that
\begin{align*}
  \partial_\nu \A_+=\A_+\WW_+,\quad  \partial_\nu \A_-=\A_-\WW_-,\quad \WW_\pm\in\mathbb{A}_n.
\end{align*}
Then we have
\begin{align*}
\A_+\WW_+=\A_-\WW_-, \text{ or equivalently } (\II-2\nn\nn)\WW_+=\WW_-.
\end{align*}
As $\WW_\pm$ are both antisymmetric, one has
\begin{align*}
\nn\cdot\WW_+=0,\quad \WW_+=\WW_-,
\end{align*}
which implies that $\WW_+=\WW_-\in\BV_4.$

By reversing the above derivation, we immediately obtain {($iii$) $\Rightarrow$ ($i$), ($ii$)}.

\subsection{Derivation of jump condition from solvability condition}\label{sec:deriv-BC}
Firstly, the solvability for $\A_0$ on $\Gamma$ directly implies that
$(\A_+, \A_-)$ forms a minimal pair.

Let $\Psi=\partial_\nu \Phi$, then $\Phi^{\mathrm{T}}\Psi$ is antisymmetric, and so do $\Phi^{\mathrm{T}}\partial_z\Psi$ and $\Phi^{\mathrm{T}}\partial_z^2\Psi$.
As one has $\partial_z\Phi=0$ and $\nabla d_0\cdot\nabla \PP_0=0$ on $\Gamma$, we have
\begin{align*}
\FF_1&= -(\partial_td_0-\Delta d_0)\partial_z\PP_0
+2\Phi^{\mathrm{T}}\nabla d_0\cdot\nabla_x\partial_z(\Psi\PP_0)-(\Phi_1\PP_0+\Phi_2\partial_z\PP_0)(d_1-z)\\
&= -(\partial_td_0-\Delta d_0)\partial_z\PP_0
+2\Phi^{\mathrm{T}}\partial_z(\Psi\PP_0)-(\Phi^{\mathrm{T}}\partial_z^2\Psi\PP_0+2\Phi^{\mathrm{T}}\partial_z\Psi\partial_z\PP_0)(d_1-z).
\end{align*}
Since
\begin{align*}
  \Phi^{\mathrm{T}}\partial_z(\Psi\PP_0):\nn\nn=\Phi^{\mathrm{T}}\partial_z^2\Psi\PP_0:\nn\nn=\Phi^{\mathrm{T}}\partial_z\Psi\partial_z\PP_0:\nn\nn=0,
\end{align*}
we get
\begin{align*}
0=  \int_\mathbb{R}\FF_1:\partial_z\PP_0\ud z=-(\partial_td_0-\Delta d_0)\int_\mathbb{R}|\partial_z\PP_0|^2\ud z,
\end{align*}
which gives on $\Gamma$ that
\begin{align*}
  \partial_td_0-\Delta d_0=0.
\end{align*}
This means $\Gamma_t$ evolves according to the mean curvature flow.

For any $\bl\bot\nn$, from the facts that
\begin{align*}
&  (\nn\bl-\bl\nn)(\II-2s\nn\nn)=\nn\bl+(2s-1)\bl\nn=(1-s)(\nn\bl-\bl\nn)+s(\nn\bl+\bl\nn)\in~\mathrm{Null}~\mathcal{L},\\
&  (\II-2s\nn\nn)(\nn\bl-\bl\nn)=(1-2s)\nn\bl-\bl\nn=(1-s)(\nn\bl-\bl\nn)-s(\nn\bl+\bl\nn)\in~\mathrm{Null}~\mathcal{L},
\end{align*}
we know that $(\nn\bl-\bl\nn)\PP_0$ and $\PP_0(\nn\bl-\bl\nn)$ are orthogonal to $\FF_1$ in $L^2(\mathbb{R})$,
which implies
\begin{align*}
  \int_{\mathbb{R}}\FF_1\PP_0 \ud z :(\nn\bl-\bl\nn)= \int_{\mathbb{R}}\PP_0\FF_1 \ud z :(\nn\bl-\bl\nn)=0\quad \text{on } \Gamma_t.
\end{align*}
Similarly, for any $\bl, ~\mm\bot\nn$,
\begin{align*}
  \int_{\mathbb{R}}\FF_1\PP_0 \ud z :(\bl\mm-\mm\bl)=\int_{\mathbb{R}}\PP_0\FF_1 \ud z :(\bl\mm-\mm\bl)=0\quad \text{on } \Gamma_t.
\end{align*}
Thus, $\int_{\mathbb{R}}\FF_1\PP_0 \ud z$,  $\int_{\mathbb{R}}\FF_1\PP_0 \ud z$ $\bot$ $\mathbb{V}_3\oplus\mathbb{V}_4=\mathbb{A}_n$.  On the other hand,
\begin{align*}
\int_{\mathbb{R}}\FF_1\PP_0 \ud z&=\int_{\mathbb{R}}\Big(2\Phi^{\mathrm{T}}\partial_z(\Psi\PP_0)-(\Phi^{\mathrm{T}}
\partial_z^2\Psi\PP_0+2\Phi^{\mathrm{T}}\partial_z\Psi\partial_z\PP_0)(d_1-z)\Big)\PP_0 \ud z\\
&=\int_{\mathbb{R}}\Big(2\Phi^{\mathrm{T}}\partial_z(\Psi\PP_0)\PP_0
-\Phi^{\mathrm{T}}\partial_z(\partial_z\Psi\PP^2_0)(d_1-z)\Big) \ud z\\
&=\int_{\mathbb{R}}\Big(2\Phi^{\mathrm{T}}\partial_z(\Psi\PP_0)\PP_0
-\Phi^{\mathrm{T}}\partial_z\Psi\PP^2_0\Big) \ud z\\
&=\int_{\mathbb{R}}\Phi^{\mathrm{T}}\partial_z\big(\Psi\PP_0^2\big) \ud z\\
&=\Phi^{\mathrm{T}}_+\Psi_+-\Phi^{\mathrm{T}}_-\Psi_-,
\end{align*}
which is antisymmetric. So, we get
\begin{align}\label{relation:jump1}
  \Phi^{\mathrm{T}}_+\Psi_+-\Phi^{\mathrm{T}}_-\Psi_-=0.
\end{align}

On the other hand, as $\Phi^{\mathrm{T}}\partial_z\Psi$ is anti-symmetric,  we have
\begin{align*}
  \PP_0(\Phi^{\mathrm{T}}\partial_z\Psi)\partial_z\PP_0-\partial_z\PP_0(\Phi^{\mathrm{T}}\partial_z\Psi)\PP_0
\end{align*}
is symmetric. Thus, we get
\begin{align*}
&\int_{\mathbb{R}}\PP_0 \FF_1\ud z\\
&=\int_{\mathbb{R}}\Big(2\PP_0\Phi^{\mathrm{T}}\partial_z(\Psi\PP_0)+\partial_z\big(\PP_0\Phi^{\mathrm{T}}\partial_z\Psi\PP_0\big)(z-d_1)
+\PP_0\Phi^{\mathrm{T}}\partial_z\Psi\partial_z\PP_0-\partial_z\PP_0\Phi^{\mathrm{T}}\partial_z\Psi\PP_0\Big) \ud z\\
&=\int_{\mathbb{R}}\Big(2\PP_0\Phi^{\mathrm{T}}\partial_z(\Psi\PP_0)+\partial_z\big(\PP_0\Phi^{\mathrm{T}}\partial_z\Psi\PP_0\big)(z-d_1)
\Big) \ud z+\MM_{\text{sym}}\\
&=\int_{\mathbb{R}}\Big(2\PP_0\Phi^{\mathrm{T}}\partial_z(\Psi\PP_0)-\PP_0\Phi^{\mathrm{T}}\partial_z\Psi\PP_0
\Big) \ud z+\MM_{\text{sym}}\\
&=\int_{\mathbb{R}}\partial_z\big(\PP_0\Phi^{\mathrm{T}}\Psi\PP_0\big)+\PP_0\Phi^{\mathrm{T}}\Psi\partial_z\PP_0-\partial_z\PP_0\Phi^{\mathrm{T}}\Psi\partial_z\PP_0
 \ud z+\MM_{\text{sym}}\\
&=(\II-2\nn\nn)\Phi_+^{\mathrm{T}}\Psi_+(\II-2\nn\nn)-\Phi_-^{\mathrm{T}}\Psi_-+\MM_{\text{sym}},
\end{align*}
where $\MM_{\text{sym}}$ denotes some symmetric matrix.  Therefore, we also have
\begin{align}\label{relation:jump2}
  (\II-2\nn\nn)\Phi_+^{\mathrm{T}}\Psi_+(\II-2\nn\nn)-\Phi_-^{\mathrm{T}}\Psi_-=0,
\end{align}
which together with (\ref{relation:jump1}) implies $\A_+^{\mathrm{T}}\partial_\nu\A_+=\A_-^{\mathrm{T}}\partial_\nu\A_-\in\mathbb{V}_4$ on $\Gamma$.

\section*{Acknowledgments}
M. Fei is supported by NSF of China under Grant No. 11871075 and 11971357. F.H.Lin is supported by an NSF grant DMS1955249. W. Wang is supported by NSF of China under Grant No. 11922118 and 11871424.

\end{document}